\documentclass[12pt]{amsart}
\usepackage[letterpaper,top=2.5cm,bottom=2.5cm,left=2.5cm,right=2.5cm,marginparwidth=1.75cm]{geometry}

\usepackage[english]{babel}
\usepackage[utf8x]{inputenc}
\usepackage[T1]{fontenc}
\usepackage{graphicx,caption,subcaption}
\usepackage{wrapfig}
\usepackage{abbrmath_insung}

\makeatletter
\newcommand{\proofstep}[1]{%
  \par
  \addvspace{\medskipamount}
  \noindent\textbf{#1\@addpunct{.}}\enspace\ignorespaces
}
\makeatother

\def\R{\mathcal{R}}
\def\ord{\mathrm{ord}}

\def\mbf{\mathbf}
\def\bfe{\mathbf{e}}
\def\bft{\mathbf{t}}

\def\TmapA{f:(S^2,A)\righttoleftarrow}

\def\band{(t;e_1,e_2)}
\def\bandp{(t';e'_1,e'_2)}
\def\bandm{(t^m;e_1^m,e_2^m)}
\def\bandn{(t^n;e_1^n,e_2^n)}

\def\pcf{post-critically finite~}
\def\hCbb{\hat{\mathbb{C}}}

\DeclareMathOperator{\Edge}{Edge}

\DeclareMathOperator{\Aut}{Aut}
\DeclareMathOperator{\V}{Vert}

\usepackage{tikz-cd}
\usepackage{tikz}
\usetikzlibrary{matrix}
\usepackage{amsmath,amssymb,amsthm}
\usepackage[colorinlistoftodos]{todonotes}
\usepackage[colorlinks=true, allcolors=blue]{hyperref}

\usepackage{enumitem}

\usepackage{mathabx,epsfig}

\newtheorem{thm}{Theorem}[section]

\newtheorem{coro}[thm]{Corollary}
\newtheorem{lem}[thm]{Lemma}
\newtheorem{prop}[thm]{Proposition}

\newtheorem{quest}[thm]{Question}

\newtheorem{caut}[thm]{Caution}

\theoremstyle{definition}

\newtheorem{defn}[thm]{Definition}
\newtheorem{eg}[thm]{Example}

\theoremstyle{remark}

\newtheorem{rem}[thm]{Remark}

\theoremstyle{definition}

\title{Levy and Thurston obstructions of finite subdivision rules}
\author{Insung Park}
\address{Department of Mathematics, Indiana University, Bloomington,
IN 47405, USA}
\email{park433@iu.edu}
\date{}
\begin{document}

\maketitle
\begin{abstract}
For a \pcf branched covering of the sphere that is a subdivision map of a finite subdivision rule, we define non-expanding spines which determine the existence of a Levy cycle in a non-exhaustive semi-decidable algorithm. Especially when a finite subdivision rule has polynomial growth of edge subdivisions, the algorithm terminates very quickly, and the existence of a Levy cycle is equivalent to the existence of a Thurston obstruction. In order to show the equivalence between Levy and Thurston obstructions, we generalize the arcs intersecting obstruction theorem by Pilgrim and Tan to a graph intersecting obstruction theorem. As a corollary, we prove that for a pair of \pcf polynomials, if at least one polynomial has core entropy zero, then their mating has a Levy cycle if and only if the mating has a Thurston obstruction.\end{abstract}
\tableofcontents

\section{Introduction}

\indent Obstructions for topological objects to have geometric structures are important subjects of study in topology and geometry. For example, the Geometrization Theorem is about topological obstructions for a 3-manifold to have one out of eight geometries. W.\@ Thurston, who conjectured and proved a large part of the Geometrization Theorem, also proved a geometrization theorem, named {\it Thurston's characterization}, in complex dynamics. He found obstructions, called {\it Thurston obstructions}, for a \pcf branched covering of the $2$-sphere to be isotopic to a rational map \cite{Thurston_obs}. Levy cycles were introduced at first as simple cases of Thurston obstructions in the study of the mating problem \cite{Levy_Thesis, Tan_quad_mating}. Recently, it turned out that a Levy cycle itself is an obstruction for a \pcf branched covering to be isotopic to an expanding dynamical system \cite{BartholdiDudko_expanding}. Therefore it is important to determine the existence of a Levy cycle as well as a Thurston obstruction for \pcf branched coverings. In this paper, we investigate a new method to detect the existence of a Levy cycle for a broad family of branched coverings, called subdivision maps of finite subdivision rules.

\proofstep{Obstructions of \pcf topological branched self-coverings of the sphere}
A continuous map $f:S^2 \rightarrow S^2$ is a {\it topological branched covering} if it locally looks like $z \mapsto z^d$ for some integer $d>0$. A point $x \in S^2$ is a {\it critical point} if $f$ is not locally injective at $x$. The collection of the critical points $\Omega_f$ is the {\it critical set} of $f$ and its forward orbit $P_f:=\bigcup_{k=1}^\infty f^{\circ k}(\Omega_f)$ is the {\it post-critical set}. If $P_f$ is finite, $f$ is a {\it post-critically finite} branched covering, or simply a {\it Thurston map}. A {\it marked \pcf branched covering}, is a map $f:(S^2,A)\righttoleftarrow$ such that $A \supset P_f$, $|A|<\infty$, and $f(A) \subset A$. Every element $a\in A$ is called a {\it marked point} and $A$ is called the {\it set of marked points} of $f:(S^2,A)\righttoleftarrow$. Since $f:(S^2,A)\righttoleftarrow$ contains the information of being \pcf and the set of marked points, we often abbreviate it just as a {\it branched covering} and write more words when they are necessary.

Two branched coverings $f:(S^2,A)\righttoleftarrow$ and $g:(S^2,B)\righttoleftarrow$ are {\it combinatorially equivalent (\,by $\phi_0$ and $\phi_1$)} if there exist homeomorphisms $\phi_0,\phi_1:(S^2,A) \rightarrow (S^2,B)$ such that (1) $\phi_0(A)=\phi_1(A)=B$, (2) $\phi_1$ is homotopic relative to $A$ to $\phi_0$, and (3) the following diagram commutes.
\[
\begin{tikzcd}
    (S^2,P_f) \arrow[r,"\phi_0"] \arrow[d,"f"]
& (S^2,P_g) \arrow[d,"g"] \\
(S^2,P_f) \arrow[r,"\phi_1"]
& (S^2,P_g)
\end{tikzcd}
\]
A \pcf topological branched covering which is not doubly covered by a torus endomorphism is combinatorially equivalent to a post-critically finite rational map if and only if it does not have a Thurston obstruction \cite{Thurston_obs}, see Section \ref{sec:graph int obs}.

\begin{defn}[Levy cycle]
A {\it Levy cycle}, or a {\it Levy obstruction}, of a \pcf branched covering $f:(S^2,A)\righttoleftarrow$ is a collection of simple closed curves $\{\gamma_1,\gamma_2,\dots,\gamma_n\}$ that are essential relative to $A$ with the following property: For each $1\le i \le n$ there is a connected component $\gamma_i'$ of $f^{-1}(\gamma_i)$ which is isotopic to $\gamma_{i+1}$ relative to $A$, and $f|_{\gamma_i'}:\gamma_i' \to \gamma_i$ is a homeomorphism.
\end{defn}

Since a Levy cycle is a homeomorphically periodic cycle, a branched covering cannot be expanding along a Levy cycle. Schwartz lemma implies that every \pcf rational map $g:(\hCbb,P_g)\righttoleftarrow$ is expanding with respect to the conformal metric on $\hCbb \setminus P_g$, except for a few cases. Therefore, Levy cycle is an example of Thurston obstruction. Shishikura and Tan found an example of mating of cubic polynomials that has a Thurston obstruction but does not have a Levy cycle \cite{mating_cubic}. Although Shishikura and Tan's example is not conjugate to a rational map, it has an expanding metric, and many objects in the study of rational maps, such as Julia sets, are still well defined. These branched coverings are called {\it B\"{o}ttcher expanding maps}, see \cite{BartholdiDudko_expanding} for B\"{o}ttcher expanding maps. Rational maps are B\"{o}ttcher expanding maps by Schwartz lemma. Recently, it was shown that a \pcf topological branched covering which is not doubly covered by a torus endomorphism is combinatorially equivalent to a B\"{o}ttcher expanding map if and only if it does not have a Levy cycle \cite{BartholdiDudko_expanding}. Therefore, Thurston and Levy obstructions can be viewed as obstructions for conformal structures and expanding dynamics on branched coverings of the sphere, respectively.

\proofstep{Analogy with surface diffeomorphisms}
There are analogues between surface diffeomorphisms and branched coverings of the sphere. Pseudo-Anosov maps are geometric in a sense that they are affine maps expanding along one dimension and contracting along the other one dimension with respect to appropriate conned Euclidean structures; rational maps are conformal geometric and B\"{o}ttcher expanding maps are metric geometrically defined. In a pseudo-Anosov mapping class, there is a unique pseudo-Anosov map up to conjugation; 
in an isotopy class of \pcf topological branched coverings, a rational map or a B\"{o}ttcher expanding map is unique up to conjugation if exists \cite{Thurston_obs, BartholdiDudko_expanding}. For non-periodic mapping classes, reducing multicurves are obstructions to pseudo-Anosov mapping classes; Thurston obstructions and Levy cycles are also multicurves, which are obstructions to being isotopic to rational maps and B\"{o}ttcher expanding maps respectively.

In spite of this analogy, however, algorithms to determine the existence of obstructions for branched coverings of the sphere are relatively less studied compared with surface diffeomorphisms. Let us review some results on algorithms about branched coverings of the sphere. Exhaustive searches for Levy cycles or Thurston obstructions are decidable \cite{BBY_ThuEqDecid}, \cite{BartholdiDudko_expanding}. For topological polynomials, a non-exhaustive algorithm, that finds either Levy cycles if exist or Hubbard trees otherwise, was developed in \cite{LiftingTrees}.  D.\@ Thurston's positive characterization also gives a non-exhaustive algorithm to detect both Levy cycles and Thurston obstructions for hyperbolic \pcf branched coverings \cite{Thu_poschac}. Although these algorithms work efficiently for many examples in practice, no theoretical upper bound of the complexity is known for any of these algorithms. An upper bound for the computational complexity was studied for nearly Euclidean Thurston maps in \cite{RatlDeciNET}. Poirier proved that an abstract Hubbard tree $H$ is a Hubbard tree of a polynomial if and only if $H$ is expanding \cite{Poirier_HubbardTrees}. This gives an efficient algorithm to check whether a Thurston obstruction (equivalently a Levy cycle in this case) exists, and one can easily find an upper bound for the complexity of this algorithm, though it is not stated in \cite{Poirier_HubbardTrees}. In this paper, Theorem \ref{thm:nonexpspine} provides a new non-exhaustive algorithm to detect Levy cycles when \pcf branched coverings are given as subdivision maps of finite subdivision rules. When edges have polynomial growth of subdivisions, Theorem \ref{thm:nonexpspine_polyedgesubdiv} implies that this algorithm terminates very quickly, and the complexity is polynomial about the number of cells. But we do not compute the complexity in this paper. 

\begin{figure*}[h!]
    \centering
    \begin{subfigure}[b]{0.25\textwidth}
        \centering
        \includegraphics[width=\textwidth]{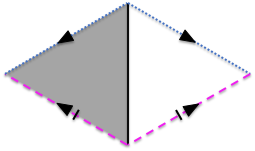}
        \caption*{$S_\R$}
    \end{subfigure}
    \hspace{20pt}
    \begin{subfigure}[b]{0.25\textwidth}  
        \centering 
        \includegraphics[width=\textwidth]{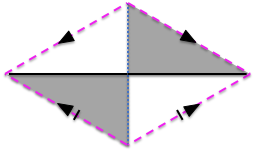}
        \caption*{$\R(S_\R)$}
    \end{subfigure}
    \hspace{20pt}
    \begin{subfigure}[b]{0.25\textwidth}  
        \centering 
        \includegraphics[width=\textwidth]{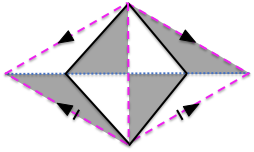}
        \caption*{$\R^2(S_\R)$}
    \end{subfigure}
    \vskip\baselineskip
    \begin{subfigure}[b]{0.25\textwidth}   
        \centering 
        \includegraphics[width=\textwidth]{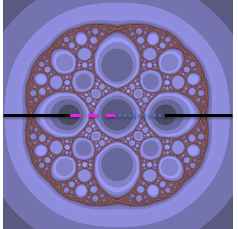}
        \caption*{$S_\R$ with Julia set}
    \end{subfigure}
    \hspace{20pt}
    \begin{subfigure}[b]{0.25\textwidth}   
        \centering 
        \includegraphics[width=\textwidth]{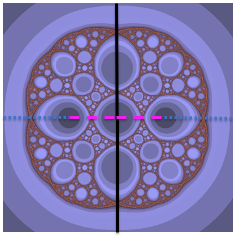}
        \caption*{$\R(S_\R)$ with Jula set}
    \end{subfigure}
    \hspace{20pt}
    \begin{subfigure}[b]{0.25\textwidth}  
        \centering 
        \includegraphics[width=\textwidth]{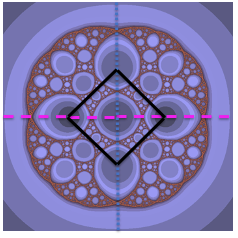}
        \caption*{$\R^2(S_\R)$ with Julia set}
    \end{subfigure}
    
    \caption{A finite subdivision rule of $z\mapsto \frac{z^2-1}{z^2+1}$. The sphere is decomposed into two triangles in $S_\R$. Each triangle subdivides into two triangles under the subdivision $\R$. The subdivision map $f:\R(S_\R)\to S_\R$ sends each shaded or unshaded triangle in $\R(S_\R)$ to the shaded or unshaded triangle in $S_\R$ respectively.} 
    \label{fig:fsrIntro}
\end{figure*}

\proofstep{Finite subdivision rules} A {\it finite subdivision rule $\R$} consists of a partition $S_\R$ of $S^2$ into polygons and its subdivision $\R(S_\R)$ such that a subdivision map $f:\R(S_\R)\to S_\R$ is homeomorpic on each open cell, see Figure \ref{fig:fsrIntro} for an example and Section \ref{sec:fsr} for a precise definition. One can also see a finite subdivision rule as a sort of Markov partition. Because $P_f \subset \V(S_\R)$, the subdivision map is a \pcf topological branched covering. By iterating subdivisions, we have a further subdivision $\R^n(S_\R)$ and an iterated map $f^n:\R^n(S_\R) \to S_\R$ for each $n \in \Nbb$. It is an open question to determine which topological \pcf branched coverings are isotopic to subdivision maps of finite subdivision rules. See Section \ref{sec:thurston maps as subdiv map} for a list of topological branched coverings that can be represented as subdivision maps.

To detect a Levy cycle, for each $n\ge 0$ we define a {\it level-$n$ non-expanding spine $N^n$} which is a graph with a train-track structure encoding non-expanding parts of $\R^n(S_\R)$, see Section~\ref{sec:nonexpspine}. A finite set $A\subset \V(S_\R)$ is called a {\it set of marked points of $\R$} if $P_f\cup f(A)\subset A$. A point $a\in A$ is called a {\it Fatou point} if its forward orbit contains a periodic critical point. Otherwise, $a\in A$ is called a {\it Julia point}. We say that the level-$n$ non-expanding spine $N^n$ is {\it essential relative to $A$} if it contains (more precisely carries as a train-track) a closed curve that is homotopic relative to $A$ neither to a point nor to some iterate of a peripheral loop of a Julia point in $A$.

\newtheorem*{mainthm1}{Theorem \ref{thm:nonexpspine}}
\begin{mainthm1}
{\it
Let $\R$ be a finite subdivision rule and $f:\R(S_\R)\to S_\R$ be its subdivision map which is not doubly covered by a torus endomorphism. Let $A\subset \V(S_\Rcal)$ be a set of marked points, i.e., $P_f \cup f(A)\subset A$. Then the \pcf branched covering $\TmapA$ has a Levy cycle if and only if the level-$n$ non-expanding spine $N^n$ is essential relative to $A$ for every $n\ge0$.
}
\end{mainthm1}

We first prove the equivalence between the existence of a Levy cycle and the existence of a sequence curves with certain properties in Section \ref{sec:alg aspect of Levy cycle} using the theory of self-similar groups. Then we show in Section \ref{sec:nonexpspine} the equivalence between the existence of such a sequence of curves and the level-$n$ non-expanding spine being essential at every level $n\ge0$.

\proofstep{Algorithmic implication}
Theorem \ref{thm:nonexpspine} improves \cite[Algorithm 5.5] {BartholdiDudko_expanding} by replacing the exhaustive semi-decidable search for nuclei of orbisphere bisets by checking if the non-expanding spines are essential, which terminates in finite time if there is no Levy cycle. There is an example showing that an arbitrarily higher level of non-expanding spine is required to be checked, see Proposition \ref{prop:fsr_needbackite_} in Section \ref{sec:fsr_needbackite_}.

\begin{quest}
Is there an upper bound function $U:\Zbb_+ \to \Zbb_+$ such that $f:(S^2,A)\righttoleftarrow$ has a Levy cycle if and only if $N^n$ is essential relative to $A$ for every $n<U(k)$ where $k$ is the number of tiles in $\R$?
\end{quest}

\proofstep{Finite subdivision rules with polynomial growth of subdivisions} We will see that the growth of the subdivision of an edge is either exponential or polynomial in Theorem \ref{thm:expsubexp} and Proposition \ref{prop:subexpedgesubdiv}. If every edge has polynomial growth of subdivisions, then the level-$n$ non-expanding spines $N^n$ are independent of $n \ge 0$. Hence the existence of a Levy cycle is decidable very quickly.

\newtheorem*{mainthm2}{Theorem \ref{thm:nonexpspine_polyedgesubdiv}}
\begin{mainthm2}
{\it
Let $\R$ be a finite subdivision rule with polynomial growth of edge subdivisions and $f$ be its subdivision map which is not doubly covered by a torus endomorphism. Let $A\subset \V(S_\R)$ be a set of marked point, i.e., $f(A) \cup P_f \subset A$. Then the followings are equivalent.
\begin{enumerate}
    \item[(1)] The branched covering $f:(S^2,A)\righttoleftarrow$ does not have a Levy cycle.
    \item[(2)] The level-$0$ non-expanding spine $N^0$ is essential relative to $A$.
    \item[(3)] The branched covering $f:(S^2,A)\righttoleftarrow$ is combinatorially equivalent to a unique rational map up to conjugation by M\"{o}bius transformations. 
\end{enumerate}
}
\end{mainthm2}

\proofstep{Equivalence between Levy cycles and Thurston obstructions} Another important implication of Theorem \ref{thm:nonexpspine_polyedgesubdiv} is the equivalence between the existence of a Levy cycle and the existence of a Thurston obstruction. As explained earlier, there are topological branched coverings which do not have a Levy cycle but have a Thurston obstruction \cite{mating_cubic}. For some families of \pcf topological branched coverings, e.g., post-critically finite topological polynomials or branched coverings of degree $2$, the existence of a Thurston obstruction implies the existence of a Levy cycle, by Levy, Rees, Tan, and Berstein \cite{Tan_quad_mating, hubbard_vol2}. We add two new families to this list: subdivision maps with polynomial growth of edge subdivisions (Theorem \ref{thm:nonexpspine_polyedgesubdiv}) and matings of polynomials one of which has core entropy zero (Corollary \ref{cor:mating}).

\newtheorem*{mainthm4}{Corollary \ref{cor:mating}}
\begin{mainthm4}
{\it Let $f$ and $g$ be post-critically finite hyperbolic (resp.\@ possibly non-hyperbolic) polynomials such that at least one of $f$ and $g$ has core entropy zero. Then $f$ and $g$ are mateable if and only if the formal mating (resp.\@ degenerate mating) does not have a Levy cycle.}
\end{mainthm4}

The equivalence between the existence of a Levy cycle and of a Thurston obstruction follows from the {\it graph intersecting obstruction theorem}, which is a generalization of the {\it arcs intersecting obstruction theorem} by Pilgrim and Tan \cite{PilgrimTan}. Here, $h_{top}$ indicates the topological entropy.

\newtheorem*{mainthm3}{Theorem \ref{thm:graph intersecting obs}}
\begin{mainthm3}[Graph intersecting obstruction]
{\it Let $\TmapA$ be a \pcf branched covering and $G$ be a forward invariant graph such that $h_{top}(f|_G)=0$. Then every irreducible Thurston obstruction intersecting $G$ is a Levy cycle.}
\end{mainthm3}

\proofstep{Examples: Critically fixed anti-holomorphic maps} In Section \ref{sec:eg}, we define an orientation reversing finite subdivision rule with no edge subdivision from every $2$-vertex-connected planar graph $G$. Then $f^2:(S^2,\V(G)) \righttoleftarrow$ and $f_\tau:=\tau \circ f:(S^2,\V(G))\righttoleftarrow$ are \pcf topological branched coverings, where $\tau$ is an orientation-reversing automorphism of $G$. Then we show in Theorem \ref{thm:faceinversion} that these maps do not have Levy cycles (or equivalently, Thurston obstructions) if and only if $G$ is $3$-edge-connected. While this article was being written, two papers \cite{LLMM_GasketSchwRef} and \cite{Geyer_CritFixAnti} were published where it is shown that every critically fixed anti-holomorphic map is constructed in this way and a theorem almost same as Theorem \ref{thm:faceinversion} is proved.

\proofstep{Notation for integer intervals}
We introduce a non-standard but intuitive notation for integer intervals to distinguish them from real intervals. For $a<b\in \Zbb$,
\begin{itemize}
    \item [$\cdot$] $[a,b]_\Zbb:=\{a,a+1,\dots, b\}$
    \item [$\cdot$] $[a,\infty]_\Zbb:=\{a,a+1,\dots\}\cup \{\infty\}$
    \item [$\cdot$] $[-\infty,b]_\Zbb:=\{-\infty\} \cup \{\cdots,b-1,b\}$ 
    \item [$\cdot$] $[-\infty,\infty]_\Zbb:=\Zbb$
\end{itemize}
The interval $[a,b]$ without the subscript $_\Zbb$ indicates the real interval $\{x\in\Rbb~|~a\le x \le b\}$.

\vspace{10pt}
\paragraph{{\bf Acknowledgements.}} The author thanks Dylan Thurston and Kevin Pilgrim for helpful comments and discussions. Without their support and suggestions, this work would not have existed. The author also thanks Dzmitry Dudko for critical comments on the previous version of this article. The author thanks the reviewer for helpful comments and suggestions, which inspired the author to write Section \ref{sec:EdgeEdgeExpVsEdgeSubdiv}. The author also thanks the developers of Xaos and Mathematica, which were used to draw Julia sets in the present article.

\section{Monotonicity of lengths under subdivisions}\label{sec:subdiv of cw-complex}

In this section, we see combinatorial properties of CW-complex without dynamics. We follow some terminology defined in \cite{finite_subdivision_rule_exp1}. Let $\Tcal$ be a finite CW-complex structure on $S^2$. A {\it $n$-gon}, or a {\it polygon} if $n$ is not specified, is a $2$-dimensional CW-complex structure on the closed $2$-disc $D^2$ whose $1$-skeleton consists of $n$ edges on $\partial D^2$. For every closed $2$-cell $t$ of $\Tcal$, there is a polygon $\bft$ and a characteristic map $\phi_t:\bft \to \Tcal$ such that $\phi_t$ is cell-wise homeomorphic and $\phi_t(\bft)=t$.

\begin{defn}[Bands and bones]\label{defn:BandSpine}
A {\it band} of $\Tcal$ is a triple $\band$, where $t$ is a closed $2$-cell and $e_1$ and $e_2$ are edges on the boundary of $t$. We allow $e_1=e_2(=e)$ only when two boundary edges of a polygon $\bft$ is are identified to $e$ by the characteristic map $\phi_t:\bft\to \Tcal$. We say that $e_1$ and $e_2$ are the {\it sides} of the band. The {\it bone of $\band$} is the homotopy class (or ambiguously a representative of the class) of curves which are properly embedded into $(t,\partial t)$ with endpoints on the interiors of $e_1$ and $e_2$.
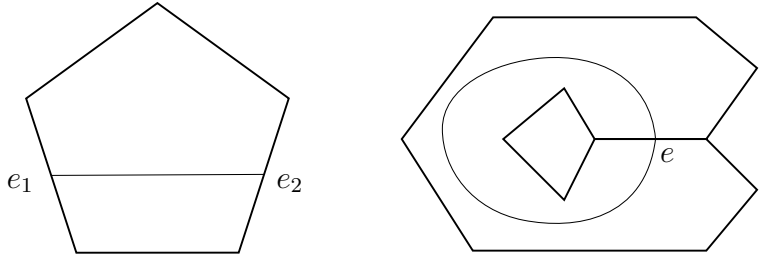
\begin{figure}[h!]
    \centering
        \def\svgwidth{0.6\textwidth}
\begingroup%
  \makeatletter%
  \providecommand\color[2][]{%
    \errmessage{(Inkscape) Color is used for the text in Inkscape, but the package 'color.sty' is not loaded}%
    \renewcommand\color[2][]{}%
  }%
  \providecommand\transparent[1]{%
    \errmessage{(Inkscape) Transparency is used (non-zero) for the text in Inkscape, but the package 'transparent.sty' is not loaded}%
    \renewcommand\transparent[1]{}%
  }%
  \providecommand\rotatebox[2]{#2}%
  \newcommand*\fsize{\dimexpr\f@size pt\relax}%
  \newcommand*\lineheight[1]{\fontsize{\fsize}{#1\fsize}\selectfont}%
  \ifx\svgwidth\undefined%
    \setlength{\unitlength}{553.15197991bp}%
    \ifx\svgscale\undefined%
      \relax%
    \else%
      \setlength{\unitlength}{\unitlength * \real{\svgscale}}%
    \fi%
  \else%
    \setlength{\unitlength}{\svgwidth}%
  \fi%
  \global\let\svgwidth\undefined%
  \global\let\svgscale\undefined%
  \makeatother%
  \begin{picture}(1,0.3365317)%
    \lineheight{1}%
    \setlength\tabcolsep{0pt}%
    \put(0,0){\includegraphics[width=\unitlength,page=1]{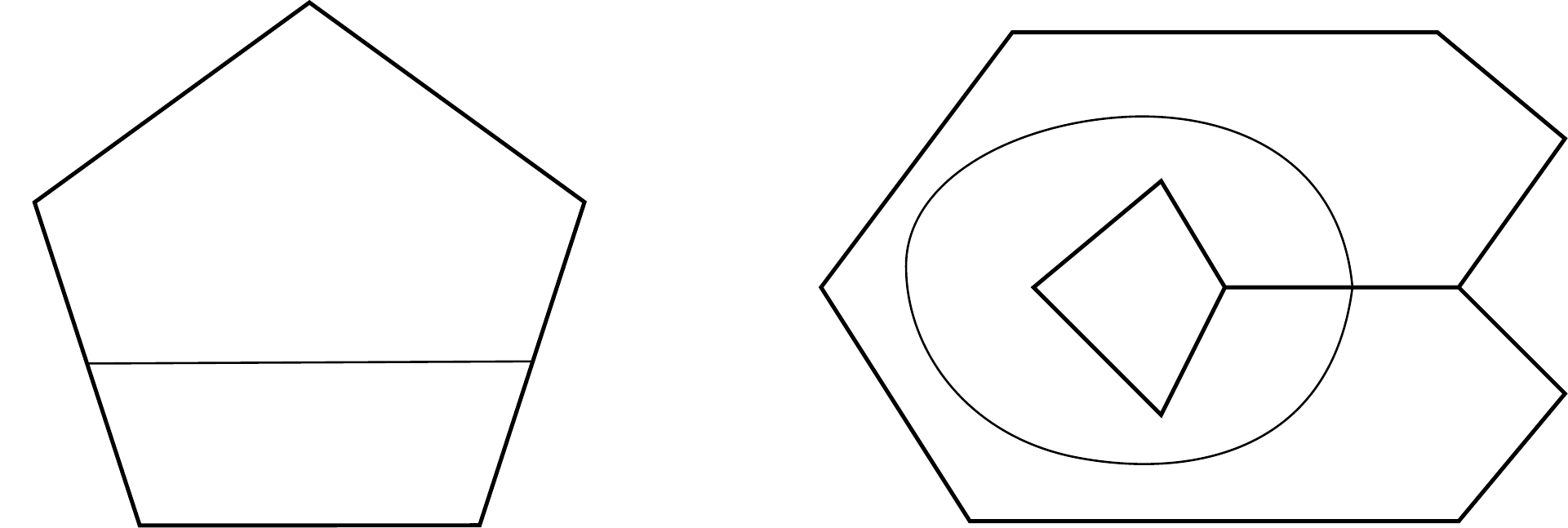}}%
    \put(-0.00361497,0.08844745){\color[rgb]{0,0,0}\makebox(0,0)[lt]{\lineheight{1.25}\smash{\begin{tabular}[t]{l}$e_1$\end{tabular}}}}%
    \put(0.35627044,0.08844745){\color[rgb]{0,0,0}\makebox(0,0)[lt]{\lineheight{1.25}\smash{\begin{tabular}[t]{l}$e_2$\end{tabular}}}}%
    \put(0.86878802,0.12198922){\color[rgb]{0,0,0}\makebox(0,0)[lt]{\lineheight{1.25}\smash{\begin{tabular}[t]{l}$e$\end{tabular}}}}%
  \end{picture}%
\endgroup%

	\caption{Bones of bands. The figure on the right shows the case when two sides of the band are the same.}
\end{figure}
\end{defn}

Let $\Tcal^{(n)}$ denote the $n$-skeleton of $\Tcal$. Any curve $\gamma \subset S^2\setminus \Tcal^{(0)}$ transverse to $\Tcal^{(1)}$ is subdivided by $\Tcal^{(1)}$ into consecutive subcurves $\gamma_1,\gamma_2, \dots ,\gamma_k$ such that each $\gamma_i$ is a maximal subcurve embedded in a closed $2$-cell. The set $\{\gamma_1,\dots,\gamma_k\}$ is the {\it $\Tcal$-decomposition of $\gamma$} and each curve $\gamma_i$ is a {\it $\Tcal$-segment of $\gamma$}. 

If $\gamma$ is not closed, then $\gamma_2,\dots,\gamma_{k-1}$ are called {\it inner $\Tcal$-segments}. The terminal $\Tcal$-segment $\gamma_1$ or $\gamma_k$ is an {\it outer $\Tcal$-segment} if one of its endpoint is in the interior of a closed 2-cell; if both endpoints are on the 1-skeleton, then we still call them inner $\Tcal$-segments. If $\gamma$ is closed, all segments are called inner segments. A curve $\gamma$ is {\it $\Tcal$-taut} if every inner $\Tcal$-segment is the bone of a band, i.e., it cannot be pushed away from the $2$-cell it is contained by an isotopy relative to $\Tcal^{(0)}$.

\begin{defn}\label{defn:CombiEq}
Two curves in $S^2\setminus \Tcal^{(0)}$ are {\it combinatorially equivalent relative to $\Tcal$}, or simply {\it $\Tcal$-combinatorially equivalent}, if they are isotopic by a cellular isotopy of $\Tcal$, i.e., a isotopy from the identity map whose restriction to each cell $X$ is also an isotopy on $X$.
\end{defn}

Define the {\it $\Tcal$-length of $\gamma$}, denoted by $l_\Tcal(\gamma)$, to be the number of inner $\Tcal$-segments. The $\Tcal$-length of a curve is an invariant of a combinatorial equivalence class. The following criterion is straightforward from the bigon criterion \cite{Primer_MCG}.

\begin{prop}
Let $\Tcal$ be a finite CW-complex structure on $S^2$. Let $\gamma$ be a curve in $S^2\setminus \Tcal^{(0)}$ transverse to $\Tcal^{(1)}$. Then $l_\Tcal(\gamma)$ is minimized in its homotopy class within $S^2\setminus \Tcal^{(0)}$, relative to endpoints if $\gamma$ is not closed, if and only if $\gamma$ is taut. Moreover, in the homotopy class, the taut curve is unique up to $\Tcal$-combinatorial equivalence.
\end{prop}


The following lemma is immediate.

\begin{lem}\label{lem:FinUptoCombiEq}
Let $\Tcal$ be a finite CW-complex structure on $S^2$. For every $l>0$, there are only finitely many, possibly closed or non-closed, curves $\delta$ in $S^2\setminus \Tcal^{(0)}$ with $l_\Tcal(\delta)<l$ up to combinatorial equivalence relative to $\Tcal$.
\end{lem}

\begin{defn}[Subbands]
Assume $\Tcal$ is a finite CW-complex structure on $S^2$ and $\Tcal'$ is its subdivision. A band $\bandp$ of $\Tcal'$ is a {\it subband} of $\band$ of $\Tcal$ if $t' \subset t$ and $e_i' \subset e_i$ for $i=1,2$.
\end{defn}

\begin{prop}[Monotonicity of lengths under refinements of CW-complexes] \label{prop:monoton length}
    Let $\Tcal$ and $\Tcal'$ be finite CW-complex structures of the $2$-sphere $S^2$ such that $\Tcal'$ is a subdivision of $\Tcal$.
    Let $\gamma$ be a $\Tcal$-taut curve and $\gamma'$ be a $\Tcal'$-taut curve such that $\gamma$ and $\gamma'$ are $\Tcal$-combinatorially equivalent. Then
    \[
        l_\Tcal(\gamma) \le l_{\Tcal'}(\gamma').
    \]
    Let $\gamma_1,\dots,\gamma_k$ be the inner $\Tcal$-segments of $\gamma$ and $\gamma'_1,\dots,\gamma'_{k'}$ be the inner $\Tcal'$-segments of $\gamma'$. Then the equality holds if and only if, under a proper reordering of indices, $\gamma_i$ is a bone of $(t_i;e_{i,1},e_{i,2})$ and $\gamma_i'$ is a bone of $(t_i';e_{i,1}',e_{i,2}')$ such that $(t_i';e_{i,1}',e_{i,2}')$ is a subband of $(t_i;e_{i,1},e_{i,2})$ for any $1 \le i \le k=k'$
\end{prop}

\begin{proof}

Take unions of consecutive $\gamma_i'$'s to get $\delta_1,\dots,\delta_{l}$ such that each $\delta_j$ is a $\Tcal$-segment of $\gamma'$. If endpoints of $\delta_i$ are on the same edge of $\Tcal$ as below, then remove it by an isotopy pushing $\delta_i$ away from the $2$-cell that it was contained so as to make $\delta_{i-1}\cup \delta_i \cup \delta_{i+1}$ be properly embedded into the $2$-cell where the subcurves $\delta_{i-1}$ and $\delta_{i+1}$ were properly embedded as shown below.

\begin{figure}[h!]
    \centering
        \def\svgwidth{0.4\textwidth}
\begingroup%
  \makeatletter%
  \providecommand\color[2][]{%
    \errmessage{(Inkscape) Color is used for the text in Inkscape, but the package 'color.sty' is not loaded}%
    \renewcommand\color[2][]{}%
  }%
  \providecommand\transparent[1]{%
    \errmessage{(Inkscape) Transparency is used (non-zero) for the text in Inkscape, but the package 'transparent.sty' is not loaded}%
    \renewcommand\transparent[1]{}%
  }%
  \providecommand\rotatebox[2]{#2}%
  \newcommand*\fsize{\dimexpr\f@size pt\relax}%
  \newcommand*\lineheight[1]{\fontsize{\fsize}{#1\fsize}\selectfont}%
  \ifx\svgwidth\undefined%
    \setlength{\unitlength}{466.50028535bp}%
    \ifx\svgscale\undefined%
      \relax%
    \else%
      \setlength{\unitlength}{\unitlength * \real{\svgscale}}%
    \fi%
  \else%
    \setlength{\unitlength}{\svgwidth}%
  \fi%
  \global\let\svgwidth\undefined%
  \global\let\svgscale\undefined%
  \makeatother%
  \begin{picture}(1,0.32394748)%
    \lineheight{1}%
    \setlength\tabcolsep{0pt}%
    \put(0,0){\includegraphics[width=\unitlength,page=1]{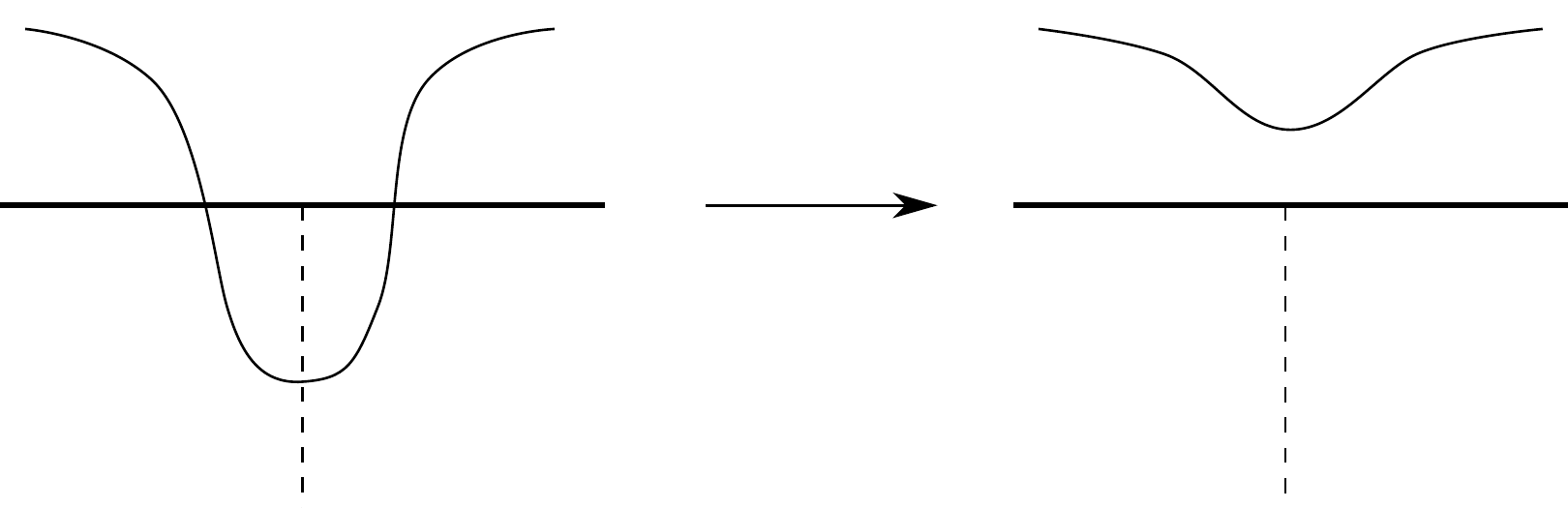}}%
    \put(0.06197102,0.30342484){\color[rgb]{0,0,0}\makebox(0,0)[lt]{\lineheight{1.25}\smash{\begin{tabular}[t]{l}$\delta_{i-1}$\end{tabular}}}}%
    \put(0.2303219,0.07375111){\color[rgb]{0,0,0}\makebox(0,0)[lt]{\lineheight{1.25}\smash{\begin{tabular}[t]{l}$\delta_i$\end{tabular}}}}%
    \put(0.28130943,0.24370967){\color[rgb]{0,0,0}\makebox(0,0)[lt]{\lineheight{1.25}\smash{\begin{tabular}[t]{l}$\delta_{i+1}$\end{tabular}}}}%
  \end{picture}%
\endgroup%

    \caption{The bold edge is an edge of $\Tcal$ and the dotted edge is an edge of $\Tcal'$ that is not an edge of $\Tcal$.}
\end{figure}
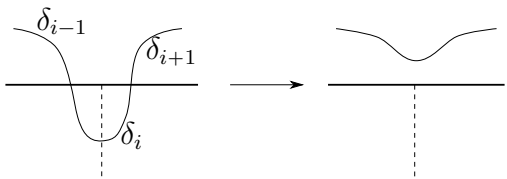

Repeating this reduction, we obtain a $\Tcal$-taut curve $\bar{\gamma}$ homotopic to $\gamma$. Let $\bar{\gamma}_1,\cdots, \bar{\gamma}_m$ be its subcurves with respect to $\Tcal$ and $\bar{\gamma}_i$ be properly embedded into a band $(\bar{t}_i;\bar{e}_{i,1},\bar{e}_{i,2})$. Since taut curves are unique in the homotopy class up to combinatorial equivalence, after reordering indices, we have $k=m$ and $(\bar{t}_i;\bar{e}_{i,1},\bar{e}_{i,2})=(t_i;e_{i,1},e_{i,2})$. Then $k=m \le l \le  k'$. The equality condition immediately follows from the constructions of $\delta_i$ and $\bar{\gamma}_i$.
\end{proof}

\section{Directed graphs and topological entropy of graph maps}\label{sec:DirectedGraphsandTopEnt}

A directed graph will be used throughout this article to understand the dynamics of branched coverings. In this section, we review basic notions of directed graphs and prove properties that we need in subsequent sections.

Let $G$ be a finite directed graph. A {\it path} is a sequence of edges $(e_1, e_2, \dots, e_n)$ such that the terminal vertex of $e_i$ is equal to the initial vertex of $e_{i+1}$ for every $i\in[1,n-1]_\Zbb$. The {\it length} of path is the number of edges in the sequence. The initial vertex of $e_1$ is the {\it initial vertex of the path} and the terminal vertex of $e_n$ is the {\it terminal vertex of the path}. If the initial and terminal vertices of a path $p$ are $v$ and $w$, then we call $p$ a {\it path from $v$ to $w$}. Let $p$ and $p'$ be paths of length $n$ and $n'$ with $n'>n$. If the first $n$ subsequence of edges of $p'$ is equal to the sequence of edges of $p$, we say that $p'$ is an {\it extension of $p$} and $p$ is the {\it first $n$-restriction of $p'$.}

A {\it cycle} is a path whose initial and terminal vertices coincide.
A vertex is {\it periodic} if it is contained in a cycle and {\it preperiodic} if it is not periodic but there is a path from the vertex to a periodic vertex. For a subset $W \subset \V(G)$, the {\it subgraph generated by $W$} is the subgraph of $G$ consisting of $W$ and edges connecting vertices in $W$.

\begin{defn}[Recurrent paths]
For a periodic vertex $v\in \V(G)$, a path $p$ from $v$ is {\it recurrent} if there exists a path from the terminal vertex $w$ of $p$ to $v$, i.e., $v$ and $w$ are contained in one cycle. We also consider a periodic vertex as a recurrent path of length $0$.
\end{defn} 

\begin{defn}[Ideals]
Let $G$ be a directed graph. A subset $X \subset \V(G)$ is an {\it ideal} if the following condition hold: For every $v\in X$, if there is a path from $v$ to $w$ for some $w\in \V(G)$, then $w\in X$.
\end{defn}

For $v \in \V(G)$, the {\it ideal  generated by $v$}, denoted by $\left< v\right>$, is the collection of vertices $w\in \V(G)$ where there is a path from $v$ to $w$.


\begin{eg}
Assume we have a directed graph as below.
\[
\begin{tikzcd}
v_1 \arrow[d,"e_1"] & v_2 \arrow[l,"e_2"] \arrow[d,"e_3"] & v_3 \arrow [l,"e_4"] & v_7 \arrow[l,"e_7"]\\
v_4 & v_5 \arrow[r,"e_5"] & v_6 \arrow [u,"e_6"] &
\end{tikzcd}
\]
Vertices $v_1$ and $v_4$ are neither periodic or preperiodic; $v_2,v_3,v_5,$ and $v_6$ are periodic; $v_7$ is preperiodic. Starting from $v_3$, the paths $(e_4,e_3,e_5)$ and $(e_4,e_3,e_5,e_6,e_4,e_3)$ are recurrent, but the paths $(e_4,e_2,e_1)$ and $(e_4,e_3,e_5,e_6,e_4,e_2)$ are not recurrent. There are only $4$ ideals: $\{v_4\}, \{v_1,v_4\}, \{v_1,v_2,v_3,v_4,v_5,v_6\},$ and $\V(G)$.
\end{eg}

\subsection{Adjacency matrices}
Let $G$ be a finite directed graph and $\V(G)=\{v_1,v_2,\dots,v_n\}$. The {\it adjacency matrix $A$} of $G$ is defined by
\[
    A_{ij}=\mathrm{the~number~of~edges~from~}v_i\mathrm{~to~}v_j,
\]
where the $A_{ij}$ is the entry of $i^{th}$-row and $j^{th}$-column of $A$. The adjacency matrix depends on the choice of indices of $v_i$'s, and matrices defined by different choices of indices differ by conjugations by permutation matrices. In particular, if the index satisfies $i<j$ only when there exists a path from $v_i$ to $v_j$, then the adjacency matrix is an upper triangular block matrix.

\begin{equation}\label{eqn:uppper_tri_block}\tag{UTB-form}
A=
\left(
\begin{array}{ccccc}
A_1 & * & \dots & * & *\\
0 & A_2 & \dots & * & *\\
\vdots & \vdots & \ddots& \vdots & \vdots\\
0& 0 & \dots & A_{k-1} & *\\
0& 0 & \dots & 0 & A_k 
\end{array}
\right),
\end{equation}
where $A_i$ is irreducible or a $0$ matrix. A non-negative $m \times m$ square matrix $M$ is {\it irreducible} if for every $1 \le i,j \le m$, there exists $k\ge1$ such that $(M^k)_{ij}>0$. An irreducible non-negative matrix $M$ has a simple eigenvalue, called the {\it Perron-Frobenius eigenvalue}, which is a positive real number and equal to the spectral radius of $M$. The spectral radius of $A$ is equal to the maximum of Perron-Frobenius eigenvalues of the irreducible $A_i$. See \cite[Chapter 2]{nonnegMat}.

\proofstep{Asymptotic growth of entries of $A^n$.} Let $B$ be a non-negative irreducible matrix and denote by $v$ and $w^T$ the right and the left eigenvectors of $B$ which are normalized by $w^T v=1$. By the Perron-Frobenius theorem, for the Perron-Frobenius eigenvalue $\rho$ of $B$, we have
\[
    \lim_{n\to \infty} B^n/\rho^n=v w^T.
\]
Hence the $(i,j)$-entry of $B^n$ is asymptotic to $\rho^n\cdot v_i \cdot w_j$. This implies the following proposition. The proof is left to the reader.

\begin{prop}\label{prop:AsymRowColumnSum}
Let $A$ be an $N\times N$ non-negative upper triangular block matrix such that $A_1,\dots A_k$ are blocks on the diagonal like \eqref{eqn:uppper_tri_block}. Let $A_i$ be an $N_i \times N_i$ matrix so that $N=\sum_{i=1}^k N_i$. For $i\in \{1,\dots,N\}$, let $l_i \in [1,k]_\Zbb$ such that $A_{l_i}$ contains the $(i,i)$-entry of $A$. Let $R_{n,i}$ (resp. $C_{n,i}$) be the $i^{th}$-row (resp. column) sum of $A^n$, i.e., the sum of the entries in the $i^{th}$-row (resp. column) of $A^n$. Then
\[
\left\{
\begin{array}{l}
    \lim_{n\to \infty} \frac{1}{n} \cdot \log (R_{n,i})=\max_{l_i \le j \le k} \log \rho_j\\[2pt]
    \lim_{n\to \infty} \frac{1}{n} \cdot \log (C_{n,i})=\max_{1 \le j \le l_i} \log \rho_j
\end{array}
\right.
\]
where $\rho_j$ is the spectral radius of $A_j$.
\end{prop}

\subsection{The growth rate of the number of paths}$ $

\proofstep{Polynomial and exponential growth rate} Suppose we have a sequence $\{a_n~|~a_n \in \Rbb_{> 0},~n\ge1\}$. We consider sequences which are uniformly bounded or diverge to the infinity when $n$ tends to the infinity. We say that the sequence {\it has exponential growth}, or {\it grows exponentially fast}, if there exists $C>D>1$ such that for any sufficiently large $n>0$ we have
\[
    D<|\log a_n|<C,
\]
and the sequence {\it has polynomial growth of degree $d$}, or {\it grows polynomially fast with degree $d$}, for a non-negative integer $d$ if there exists $C>D>0$ such that for any sufficiently large $n>0$ we have
\[
    D\cdot n^d < a_n < C \cdot n^d.
\]

\begin{lem}\label{lem:polygrowth}
Suppose $\{a_n~|~n\ge1,~a_n\ge0 \}$ is a non-decreasing sequence. If there exist $k,l,m>0$ such that $a_{k\cdot n+l}\ge m \cdot n^d$ (resp.\@ $a_{k\cdot n+l}\le m \cdot n^d$) for every $n>0$, then there exists $C>0$ such that $a_n>C\cdot n^d$ (resp.\@ $a_n<C\cdot n^d$) for any sufficiently large $n$. 
\end{lem}
\begin{proof}
Since the sequence $\{a_n\}$ is non-decreasing, we have $a_{(k+1)\cdot n}\ge m \cdot n^d$ for any $n>l/k$. Then, for any sufficiently large $n>0$, we have
\[
a_n \ge a_{(k+1)\cdot\lfloor \frac{n}{k+1} \rfloor}\ge m\cdot {\left\lfloor\frac{n}{k+1}\right\rfloor}^d>\frac{m}{(k+2)^d} \cdot n^d,
\]
where $\lfloor x \rfloor$ means the largest integer less than or equal to $x$. The same argument also works for the inverse direction.
\end{proof}

Let $G$ be a directed graph. Two paths in $G$ are {\it different} if their sequences of edges are different. Let $P(v,n)$ be the number of different paths of length exactly $n$ starting from $v$. The following criterion on the growth rate of $P(v,n)$ is well-known. See \cite{Sidki_growth, Ufn_growth}. We restate the criterion with a slight improvement for the case of polynomial growth. 

\begin{thm}\label{thm:expsubexp}
 Let $G$ be a directed graph and $v \in \V(G)$.
 
 \begin{itemize}
     \item [(1)] If there exist two different cycles passing through $v$, then $P(v,n)$ has exponential growth.
     \item [(2)] If there exists $w\in \V(G)$ such that there are two different cycles starting from $w$ and a path from $v$ to $w$, then $P(v,n)$ has exponential growth.
     \item [(3)] If $v$ does not satisfy $(1)$ or $(2)$, then $P(v,n)$ has polynomial growth. Moreover, if the maximum number of (disjoint) cycles that a path from $v$ can intersect is $d+1$, then the degree of the polynomial growth of $P(v,n)$ is $d$. When the maximum number of cycles is zero, then $P(v,n)=0$ for every sufficiently large $n$ and we define the degree of polynomial growth to be $-1$.
\end{itemize}

\end{thm}

\begin{proof}
The number of paths of length $\le n$ is counted in \cite{Ufn_growth} and we slightly modify it. If two cycles of lengths $p$ and $q$ pass through $v$, then there are at least $2^n$ paths of length $npq$ starting from $v$. Let $M$ be the maximal number of outgoing edges from one vertex in $G$. Then the number of paths with length $n$ is less than $M^n$. This proves $(1)$ and $(2)$ is immediate from $(1)$. 

Assume that $v$ does not satisfy $(1)$ or $(2)$. Let $W=\left< v \right>$ be the ideal generated by $v$. There exist totally ordered finite subsets $W_1,\cdots W_k$ such that $v \in W_i$ for each $i$ and $W=\bigcup_{i=1}^k W_i$. More precisely, we think of a directed graph $H$ obtained by collapsing each cycle in $G$ to a vertex. Since $v$ does not satisfy (1) or (2), the graph $H$ does not have any cycle. Let $v'$ be the vertex of $H$ corresponding to $v$. Then $v'$ is the only minimal element of $\V(H)$. Let $w_1,w_2,\dots,w_m$ be the maximal elements of $\V(H)$. Choose any $w\in \V(W)$.

Then the subgraph $G_i$ generated by $W_i$ is isomorphic to the graph in Figure \ref{fig:polynomialgrowth}. Since every path from $v$ is supported in some $G_i$, it suffices to show $(3)$ when $G$ is the graph of the type shown in Figure \ref{fig:polynomialgrowth}.

\begin{figure}[h!]
    \centering
        \def\svgwidth{0.7\textwidth}
\begingroup%
  \makeatletter%
  \providecommand\color[2][]{%
    \errmessage{(Inkscape) Color is used for the text in Inkscape, but the package 'color.sty' is not loaded}%
    \renewcommand\color[2][]{}%
  }%
  \providecommand\transparent[1]{%
    \errmessage{(Inkscape) Transparency is used (non-zero) for the text in Inkscape, but the package 'transparent.sty' is not loaded}%
    \renewcommand\transparent[1]{}%
  }%
  \providecommand\rotatebox[2]{#2}%
  \newcommand*\fsize{\dimexpr\f@size pt\relax}%
  \newcommand*\lineheight[1]{\fontsize{\fsize}{#1\fsize}\selectfont}%
  \ifx\svgwidth\undefined%
    \setlength{\unitlength}{850.55669364bp}%
    \ifx\svgscale\undefined%
      \relax%
    \else%
      \setlength{\unitlength}{\unitlength * \real{\svgscale}}%
    \fi%
  \else%
    \setlength{\unitlength}{\svgwidth}%
  \fi%
  \global\let\svgwidth\undefined%
  \global\let\svgscale\undefined%
  \makeatother%
  \begin{picture}(1,0.14226053)%
    \lineheight{1}%
    \setlength\tabcolsep{0pt}%
    \put(0,0){\includegraphics[width=\unitlength,page=1]{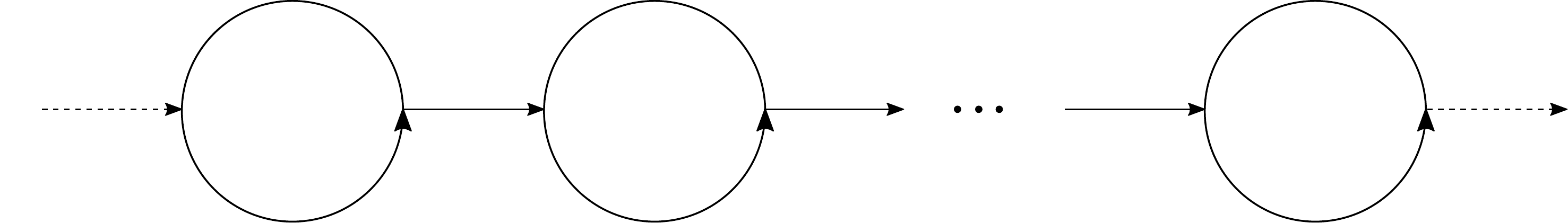}}%
    \put(-0.00266944,0.06457992){\color[rgb]{0,0,0}\makebox(0,0)[lt]{\lineheight{1.25}\smash{\begin{tabular}[t]{l}$v$\end{tabular}}}}%
  \end{picture}%
\endgroup%

	\caption{A graph with polynomial growth rate of $P(v,n)$. Any arrow indicates paths, any dotted arrow means it may not exist but if exists it indicates a path, and any circle indicates a cycle. In each cycle, the incoming vertex and the outgoing vertex could be the same.}
    \label{fig:polynomialgrowth}
\end{figure}

Assume $G$ is a graph of the type in Figure \ref{fig:polynomialgrowth}. Let it have $d+1$ cycles of lengths $p_1,p_2,\dots, p_{d+1}$. Let $p_M$ and $p_m$ be the maximum and minimum of $\{p_1,\dots, p_d\}$ and $K$ and $L$ be the number of vertices and edges of $G$ respectively. Let $X(n)$ be the set of $d$-tuples $(n_1,\dots,n_d)$ of non-negative integers satisfying $n_1 + \dots +n_d \le n$. The set $X(n)$ has $\binom{n+d}{d}$ elements.

\vspace{5pt}
\noindent{\it Claim: For any $n\ge 1$, $P(v,p_m\cdot n)/K \le \binom{n+d}{d} \le P(v,p_M \cdot n +L)$.}
\begin{proof}[Proof of claim] For any $n$, we define an injective map from $X(n)$ to the set of paths from $v$ of length $p_M\cdot n+L$ as follow. For any $(n_1,\dots,n_d)\in X(n)$, we have a path from $v$ which goes around the $i$-th cycle $n_i$-times for $i \le d$. The path has length not greater than $p_M\cdot n + L$. There is a unique extension of the path which (1) does not further go around the first d cycles, i.e., the additional rotations occur only in the last cycle, and (2) has length exactly equal to $p_M\cdot n + L$. We assign the extended path to the element $(n_1,\dots,n_d)\in X(n)$. Hence we have $\binom{n+d}{d} \le P(v,p_M\cdot n+L)$.

Similarly, we define another map from the set of paths of length $p_m\cdot n$ to $X(n)$ by assigning the numbers of times that each path goes around the first $d$ cycles. Suppose that two such paths $p_1$ and $p_2$ have the same image in $X(n)$. Since $p_1$ and $p_2$ start from the same vertex $v$ and have the same length, there terminal vertex are the same if and only if $p_1$ and $p_2$ are the same. It follows that the number of preimage of the map is less than $K$.
\end{proof}

As $\binom{n+d}{d}$ is a degree $d$ polynomial in $n$, it follows from Lemma \ref{lem:polygrowth} that the $P(v,n)$ has polynomial growth of degree $d$.as a sequence in $n$.
\end{proof}

The polynomial growth rate of $P(v,n)$ implies the recurrent extension is unique.

\begin{prop}[Extension of recurrent paths]\label{prop:ExtensionRecurrentPath}
Let $G$ be a directed graph and $v\in \V(G)$. Let $\gamma$ be a path from $v$ of length $m>0$. If $\gamma$ is recurrent, then for any $n>m$, $\gamma$ has at least one extension to a recurrent path of length $n$. Moreover, if $P(v,n)$ grows polynomially fast as $n$ tends to $\infty$, then the extension is unique.
\end{prop}
\begin{proof}
Since $\gamma$ is recurrent, the initial and the terminal points of $\gamma$ belong to one cycle $C$ of $G$. Then we can extend $\gamma$ by repeatedly travelling along $C$. If $P(v,n)$ grows polynomially fast, then $C$ is the only cycle that passes through $v$. Hence travelling along $C$ is the only way to extend $\gamma$.
\end{proof}


\subsection{Graph maps with zero topological entropy}
Let $G$ be a finite graph and $\V(G)$ be its vertex set. A continuous map $f:G \to G$  {\it Markov} if $f(\V(G)) \subset \V(G)$ and $f$ is a homeomorphism or constant on each component of $G \setminus f^{-1}(\V(G))$. Then the edges form a Markov partition of $f$. Denote by $e_1,e_2,\dots$ the edges of $G$. Note that every edge is mapped to a union of edges. The {\it adjacency matrix} $A_f$ of $f:G \to G$ is defined in such a way that $f(e_i)$ covers $e_j$ as many as the $(i,\,j)$-entry of $A_f$. Under a suitable choice of indexing $e_i$'s, we may assume that $A_f$ is an upper triangular block matrix. Let $A_1,\dots A_k$ be the block matrices on the diagonal as in \eqref{eqn:uppper_tri_block}.

The spectral radius $\lambda$ of $A_f$ is either equal to zero if every $A_i$ is zero or equal to the maximum of Perron-Frobenius eigenvalues of the irreducible $A_i$'s. The {\it topological entropy $h_{top}(f)$} of $f$ is equal to zero if $\lambda=0$ or equal to $\log(\lambda)$ if $\lambda>0$. The relationship between the topological entropy and the Perron-Frobenius eigenvalue follows from \cite{MisiuSzlenk_EntropyMonotoneMapping}.

There is a directed graph $\Dcal_f$ such that $\V(\Dcal_f) = \Edge(G)$ and the directed edges are defined as follow: For $e_i,e_j\in \Edge(G)$, if $f(e_i)$ covers $e_j$ exactly $k$-times, then we draw $k$ directed edges of $\Dcal_f$ from $e_i$ to $e_j$. Then $A_f$ is equal to the adjacency matrix of $\Dcal_f$. We refer to $\Dcal_f$ as the directed graph of the Markov map $f:G\to G$.

\vspace{5pt}
The following lemma is elementary, but we prove it for the sake of self-containedness. 

\begin{lem}\label{lem:integerirredmatrix}
Let $M$ be an irreducible non-negative integer matrix and $\lambda(M)$ be its Perron-Frobenius eigenvalue. Then $\lambda(M) \ge 1$ and the equality holds if and only if $M$ is a permutation matrix.
\end{lem}
\begin{proof}
Because the characteristic polynomial of $M$ is a monic polynomial with integer coefficients, the absolute value of the product of eigenvalues is a positive integer. Hence the spectral radius $\lambda(M)$ is at least one.
If $M$ is a permutation matrix, then trivially $\lambda(M)=1$. Assume $\lambda(M)=1$. Let $H$ be a directed graph whose adjacency matrix is $M$. Since $M$ is irreducible, for any pair $(v,w)$ of vertices of $H$ there exists a path $p_1$ from $v$ to $w$ and a path $p_2$ from $w$ to $v$. If there exists a vertex $x\neq v,w$ through which both paths $p_1$ and $p_2$ pass, there are two different cycles passing through $x$. By Theorem \ref{thm:expsubexp}, the number of length-$n$ paths $P(x,n)$ grows exponentially fast. Since $P(x,n)$ is the sum of entries in the row of $M^n$ corresponding to the vertex $v$, it follows that $\lambda(M)>1$. So $p_1$ and $p_2$ form a cycle which passes through vertices exactly once. If there is a vertex $x$ that is not contained in this cycle, then there is a path from $x$ to $v$ and a path from $v$ to $x$. Then two different cycles pass through $v$ so $P(v,n)$ grows exponentially fast and $\lambda(M)>1$. Hence $H$ is a cycle passing through every vertex exactly once, and $M$ is a permutation matrix.
\end{proof}

\begin{prop}\label{prop:entropyzerograph}
Let $f:G \to G$ be a Markov map. Then the followings are equivalent.
\begin{itemize}
    \item [(1)] The topological entropy $h_{top}(f)$ is zero.
    \item [(2)] Every irreducible block $A_i$ of the upper-triangular block form of the adjacency matrix 
    $A_f$ is a permutation matrix.
    \item [(3)] The directed graph $\Dcal_f$ of the adjacency matrix $A_f$ has disjoint cycles, i.e., every pair of different cycles have disjoint vertices.
    \item [(4)] There exists a positive integer $d$ such that $({A_f}^n)_{ij}=O(n^d)$ for all $i,\,j$.
\end{itemize}

\end{prop}

\begin{proof}
$(3) \Leftrightarrow (2)\Rightarrow(1)$ is trivial and $(4) \Leftrightarrow (3)$ is immediate from Theorem \ref{thm:expsubexp}. Assume $h_{top}(f)=0$. Then the Perron-Frobenius eigenvalue of every irreducible block $A_i$ of the adjacency matrix $A_f$ is one. $(1) \Rightarrow (2)$ follows from Lemma \ref{lem:integerirredmatrix}.

\end{proof}

\section{Finite subdivision rules}\label{sec:fsr}

A finite subdivision rule $\R$ consists of the following:

\begin{itemize}
    \item [(1)] a {\it subdivision complex} $S_\R$ which is a $2$-dimensional finite CW-complex such that the underlying space is the union of its closed $2$-cells, i.e., every $0$ or $1$-cell is on the boundary of a $2$-cell,
    \item [(2)] a {\it subdivision} $\R(S_\R)$ of $S_\R$, that is a CW-complex for which every open cell is contained in an open cell of $S_\R$, and
    \item [(3)] a {\it subdivision map} $f:\R(S_\R) \rightarrow S_\R$ which is continuous and {\it cell-wise homeomorphic}, i.e., its restriction to each open cell is a homeomorphism onto an open cell.
\end{itemize}

We say that $\R$ is {\it orientation-preserving} if every $2$-cell can be oriented in such a way that $f$ preserves the orientation. Similarly, $\R$ is {\it orientation-reversing} if $f$ reverses the orientation on every cell. For any closed $2$-cell $t$ of $S_\R$, there exist a {\it $n$-gon} $\bf{t}$ and the characteristic map $\phi_t:\mathbf{t} \to t$ is cell-wise homeomorphic. The CW-complex $\mathbf{t}$ is called the {\it tile type} of the closed $2$-cell $t$. Similarly, for a characteristic map of a $1$-cell $\phi_e:\mathbf{e} \to e$, $\mathbf{e}$ is the {\it edge type} of a closed $1$-cell $e$.

A $2$-dimensional CW-complex $X$ is an {\it $\R$-complex} if it is the union of its closed $2$-cells and there is a continuous cell-wise homeomorphism $g:X \rightarrow S_\R$. By pulling back the subdivision $\R^n(S_\R)$ of $S_\R$ through $g$ for each $n>0$, we also have a subdivision $\R^n(X)$ of $X$ and a cell-wise homeomorphism $f^{\circ n} \circ g:\R^n(X) \rightarrow S_\R$. For example, $S_\R$ itself and any tile type $\mbf{t}$ are $\R$-complexes, so for any $n \in 
\Nbb$ their {\it level-$n$ subdivisions} $\R^n(S_\R)$ and $\R^n(\mbf{t})$ are defined. For any edge type $\mbf{e}$, its level-$n$ subdivision $\R^n(\mbf{e})$ is also similarly defined.

We call a closed $2$-cell (resp.\@ a closed $1$-cell, a $0$-cell) a {\it tile} (resp.\@ an {\it edge}, a {\it vertex}) of a $2$-dimensional complex. Every level-$0$ tile or edge of an $\R$-complex is also a $\R$-complex. See \cite{finite_suvdivision_rule_first} for more details on finite subdivision rules.


\proofstep{Notation.} As we wrote in the previous paragraph, we use bold fonts for the domains of characteristic maps and normal fonts for the corresponding closed cells in the CW-complexes. For example, for a closed $2$-cell $t$ in a CW-complex $X$, we write $\phi_t:\bft\to t$ for the characteristic map. Thus $\bft$ is always homeomorphic to the closed $2$-disk, but $t$ may not. 

\begin{rem}\label{rem:monogonbigon}
Unlike in other articles on finite subdivision rules, every tile type is not assumed to have at least three vertices in this article. 
This modification allows the graphs in Theorem \ref{thm:faceinversion} to have bigon faces, see Example \ref{eg:fsrwithbigon}.
\end{rem}

\proofstep{Subdivision maps as \pcf branched coverings} Throughout this article, we assume $S_\R$ is homeomorphic to the $2$-sphere $S^2$. Considering $\R(S_\R)$ and $S_\R$ as different complexes on the same underlying $2$-sphere, we may think of the subdivision map $f:\R(S_\R) \rightarrow S_\R$ as a topological branched self covering of $S^2$. Since the set of critical points $\Omega_f$ is a subset of the set of vertices of $\R(S_\R)$, $f$ is post-critically finite.

A {\it set of marked points $A$} of $\R$ is a subsets of $\V(S_\R)$ with $P_f \cup f(A) \subset A$. With a choice of a set of marked points $A$, the subdivision map can be considered as a marked \pcf branched covering $\TmapA$.

\subsection{Branched coverings represented as subdivision maps}\label{sec:thurston maps as subdiv map}

If $f$ is a subdivision map, then the $1$-skeleton $S_\R^{(1)}$ is a graph such that (1) it contains $P_f$, (2) it is connected, and (3) it is forward invariant under $f$. Conversely, if there is a graph satisfying the three conditions, then it defines a finite subdivision rule. Below is a list of some forward invariant graphs that are known to exist.

\begin{itemize}
    \item Spiders of polynomials \cite{spider}.
    
    \item Hubbard trees can be augmented to be invariant trees \cite{spanning_tree_quad}.
    
    \item Jordan curves \cite{bonk_meyer_exp_thurston} and trees \cite{hlushchanka_thesis} of expanding Thurston maps.
    \item Tischler graphs of critically fixed rational maps \cite{PilgrimTan, Hlushchanka_criticallyfixed}.
    \item Jordan curves \cite{GHMZ_JCorveCarpetJulia} and trees \cite{hlushchanka_thesis} of Sierpi\'{n}ski Carpet rational maps. 
    \item Extended Newton graphs for post-critically finite Newton maps \cite{lodge_newtonmaps}.
    
    \item A sufficiently large iterate $f^n$ of any \pcf branched covering $f$ without a Levy cycle is homotopic to a subdivision map \cite{finite_subdivision_rule_exp2}. Its 1-skeleton is invariant up to homotopy.
\end{itemize}

\noindent There are \pcf branched coverings whose any iterates cannot be represented as subdivision maps \cite[Section 4]{finite_subdivision_rule_exp2}.

\subsection[alternative title goes here]{Combinatorial properties of \texorpdfstring{$\Rcal$}{~} and  Levy and Thurston obstructions}

A finite subdivision rule $\R$ is {\it edge-separating} if for every tile type $\bft$ and pair of disjoint closed edges $e$ and $e'$ of $\bft$, there exists a positive integer $n$ such that no subtile of $\bft$ in $\R^n(\bft)$ contains both a subedge of $e$ and a subedge of $e'$. Similarly, $\R$ is {\it vertex-separating} if for every tile type $\bft$ and pair of vertices $v$ and $w$ of $\bft$, there exists a positive integer $n$ such that no subtile of $\bft$ in $\R^n(\bft)$ contains both $v$ and $w$. These two separating conditions are a part of sufficient condition for a subdivision map not having a Levy cycle or a Thurston obstruction.

\begin{itemize}
    \item If $\R$ is vertex-separating and edge-separating, then $f$ does not have a Levy cycle \cite[Proposition 5.1]{finite_subdivision_rule_exp1}. There is a finite subdivision rule which is neither edge-separating nor vertex-separating but does not have a Levy cycle, see Example \ref{eg:fsr unsep no levy}.
    \item If $\Rcal$ is vertex-separating, edge-separating and conformal (we do not define this conformality in the article), then $f$ does not have a Thurston obstruction \cite{const_ratlmap_from_fsr}. There is an example \cite[Example 4.6]{const_ratlmap_from_fsr} of finite subdivision rule which is not conformal but does not have a Thurston obstruction, thus it is combinatorially equivalent to a rational map. See \cite{const_ratlmap_from_fsr} for a definition of conformal finite subdivision rules.
\end{itemize}

\begin{eg}\label{eg:fsr unsep no levy}
The finite subdivision rule $R$ given in Figure \ref{fig:expandingmap} is Example 5.3 of \cite{finite_subdivision_rule_exp1}. Its CW-complex $\Scal_\R$ of the $2$-sphere consists of two square tiles. The edges of white and shaded tiles are glued to form a pillowcase. Since the shaded tile does not subdivide, $\R$ is neither edge-separating or vertex-separating. However, it easily follows from Theorem \ref{thm:nonexpspine} that the subdivision map does not have a Levy cycle.
\begin{figure}[h!]
    \centering
        \def\svgwidth{0.6\textwidth}
        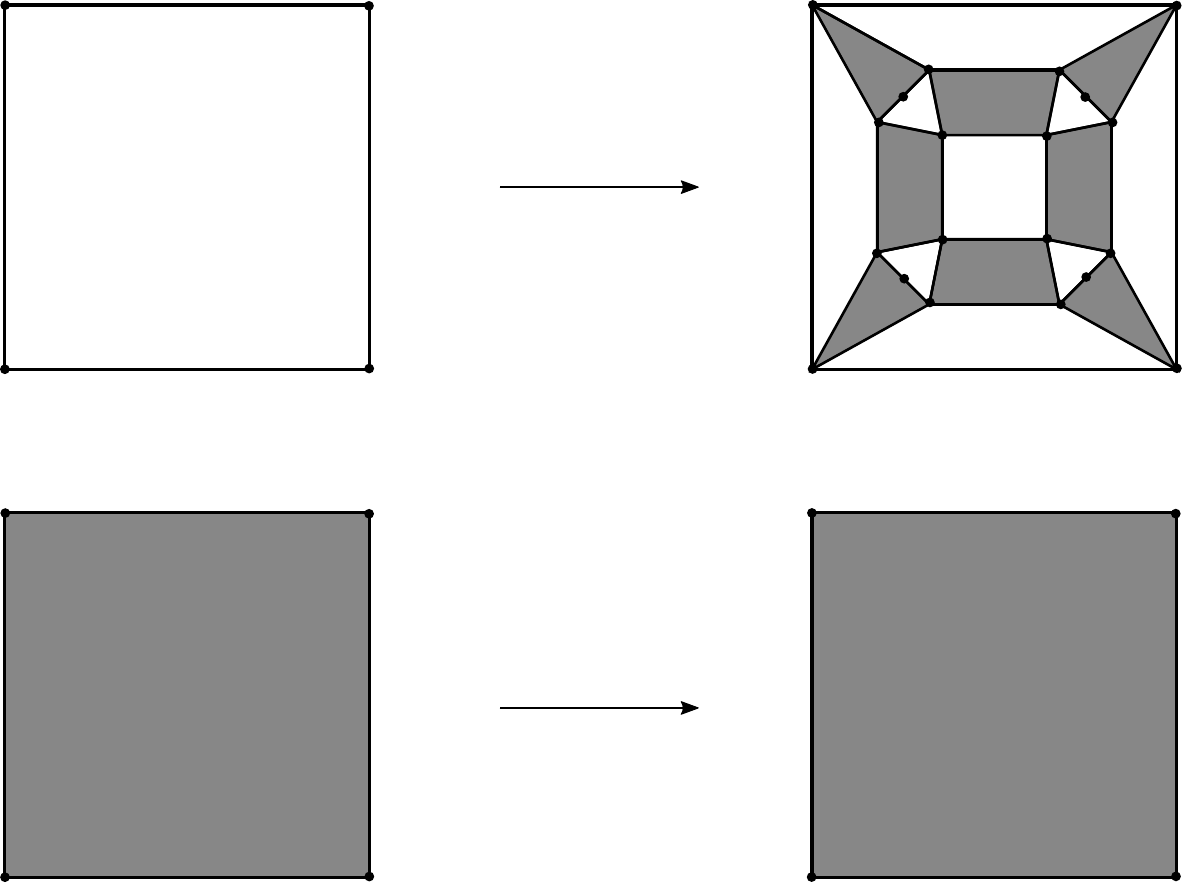
	\caption{An example of expanding finite subdivision rule which is neither edge-separating nor vertex-separating.}
    \label{fig:expandingmap}
\end{figure}
\end{eg}

\subsection{Edges, bands, bones, and curves of subdivision complexes}
For $n\ge 0$, a {\it level-$n$ tile, edge, or band} of $\R$ is a tile, edge or band of $\R^n(S_\R)$. See Definition \ref{defn:BandSpine} for definitions of bands and their bones. There is a bijection between level-$0$ tiles (resp.\@ edge) and tile types (resp.\@ edge types); a level-$0$ tile $t$ is the image of the tile type $\bft$ under the characteristic map $\phi_t:\bft\to t$.

We will use superscripts to indicate the level of tiles, edges, etc. Since frequently considering level-$0$ objects, we sometimes omit the superscript $^0$ for simplicity.

For $n>m$, a level-$n$ tile $t^n$ is a {\it subtile} of a level-$m$ tile $t^m$ if $t^{n} \subset t^{m}$. Let $\bft$ be a tile type and $t$ be the corresponding level-$0$ tile. A level-$n$ tile $t^n$ is {\it of type $\bft$} if $f^n(t^n)=t$. Subedges and their types are similarly defined. A {\it band type} is a level-$0$ band. For a band type $\band$, a level-$n$ band $\bandn$ is {\it of type $\band$} if the $f^n$-image of its bone is the bone of $\band$ (or,  equivalently, if $t^n$ is of type $\bft$ and $e_i^n$ is of type $\bfe_i$ for $i=1,2$. For $n>m$, a level-$n$ band $\bandn$ is a {\it subband} of a level-$m$ band $\bandm$ if $t^n\subset t^m$ and $e_i^n \subset e_i^m$ for $i=1,2$. If $\deg(f)=d$, there are $d^n$ level-$n$ tiles, edges, and bands of the same type.

\begin{defn}[Abbreviations for level-$n$ bands and bones]
 There are many level-$n$ bands that are not subbands of level-$0$ bands. However the only level-$n$ bands that we consider are level-$n$ subbands of level-$0$ bands. Since these objects will be very frequently used, for the sake of simple notation, by a level-$n$ band we mean a level-$n$ subband of a level-$0$ band. Similarly, by a level-$n$ bone we mean the bone of a level-$n$ subband of a level-$0$ band. 
\end{defn}

\begin{defn}[Non-expanded level-$n$ curves]\label{defn:LevelNCurve}
Let $\R$ be a finite subdivision rule. Let $I$ be a closed interval $[k,l]$, $(-\infty,k]$, $[k,\infty)$, or $(-\infty,\infty)$ for $k<l\in \Zbb$. For $n\ge0$, a curve $\gamma^n:I\to \R^n(S_\R)$ is a {\it non-expanded level-$n$ curve}, if $\gamma^n([i,i+1])$ is a level-$n$ bone for every $i\in \Zbb$ with $[i,i+1]\subset I$. A non-expanded level-$n$ curve is {\it recurrent} if it consists of level-$n$ bones that are recurrent. The recurrent bands and bones are defined in Definition \ref{defn:RecurrentBands}.
\end{defn}

\subsection{Two directed graphs defined from finite subdivision rules}\label{sec:DirectedGraphs}

\subsubsection{Directed graphs of edge subdivisions}

Let $\Ecal$ be a directed graph such that $\V(\Ecal)$ is the same as the set of level-$0$ edges. To avoid confusion, we denote by $[e]$ the vertex of $\Ecal$ corresponding to an edge $e$. A directed edge from $[e]$ to $[e']$ corresponds to a level-$1$ subedge of $e$ of type $\bfe'$. We call $\Ecal$ the {\it directed graph of edge subdivision of $\R$}. The next proposition is straightforward from the definitions.

\begin{prop}\label{prop:edgesubdivision}
There is an 1-1 correspondence between the paths in $\Ecal$ of length $n$ starting from $[e]$ and the level-$n$ subedges of $e$. Thus, the number of level-$n$ subedges is equal to $P_\Ecal([e],n)$, the number of paths of length $n$ starting from $[e]$.
\end{prop}

\begin{defn}[Periodic and recurrent edges]\label{defn:PeriodicRecurrentEdges}
Let $\R$ be a finite subdivision rule and $\Ecal$ be the directed graph of edge subdivisions of $\R$. We define level-$0$ {\it periodic edges} and {\it recurrent level-$n$ edges} as follows.
\begin{itemize}
    \item A level-$0$ edge $e$ is {\it periodically (resp.\@ preperiodically) subdividing}, or simply {\it periodic (resp.\@ preperiodic)}, if $[e] \in \V(\Ecal)$ is periodic (resp.\@ preperiodic). Equivalently, $e$ is periodic if and only if there exists a level-$n$ subedge of $e$ of type $\bfe$ for some $n>0$.
    \item A level-$n$ edge $e^n$ is a {\it recurrent subedge} of $e$ if it corresponds, by Proposition \ref{prop:edgesubdivision}, to a recurrent path in $\Ecal$ which starts from $[e]$ and has length $n$, or equivalently if a further subdivision of $e^n$ contains a subedge of type $\bfe$. If $\bfe'$ is the type of a recurrent subedge of $e$, then there is a cycle passing through both $[e']$ and $[e]$. Only periodic level-0 edges have recurrent subedges.
    \item We also refer to periodic level-$0$ edges as recurrent level-$0$ edges, which is sometimes useful for concise statements.
\end{itemize}
\end{defn}


\subsubsection{Directed graphs of bands} Let $\Bcal$ be a directed graph such that $\V(\Bcal)$ is the set of level-$0$ bands $\band$. To avoid confusion, we use bracket $[\band]$ to denote vertices of $\Bcal$. Every directed edge from $[\band]$ to $[\bandp]$ corresponds to a level-$1$ subband of $\band$ of type $\bandp$. We call $\Bcal$ the {\it directed graph of bands} of $\R$. The following proposition is an analogue to Proposition \ref{prop:edgesubdivision}.

\begin{prop}\label{prop:bandproperlyembedding}
There is an 1-1 correspondence between the paths in $\Bcal$ of length $n$ starting from $[\band]$ and the level-$n$ subbands of $\band$.
\end{prop}

\begin{defn}[Periodic and recurrent bands and bones]\label{defn:RecurrentBands}
Let $\R$ be a finite subdivision rule and $\Bcal$ be the directed graph of bands of $\R$. We define level-$0$ periodic bands and level-$n$ recurrent subbands as we did for edges.

\begin{itemize}
    \item A level-$0$ band $\band$ is {\it periodic (resp.\@ preperiodic)}, if $[\band]\in \V(\Bcal)$ is periodic (resp.\@ preperiodic). Equivalently, $\band$ is periodic if and only if there exists a level-$n$ band $\bandn$ of type $\band$ which is a subband of $\band$ for some $n>0$.
    \item A level-$n$ subband $\bandn$ of $\band$ is a {\it recurrent subband} of $\band$ if it corresponds, by Proposition \ref{prop:bandproperlyembedding}, to a recurrent path of length $n$ starting from $[\band]$, or, equivalently, if $\bandn$ is a subband of $\band$ and has a subband in its further subdivision which is also a subband of $\band$. If $\bandp$ is the type of a recurrent subband of $\band$, then there is a cycle passing through both $[\bandp]$ and $[\band]$. Only periodic level-$0$ bands have recurrent subbands.
    \item We also refer to periodic level-$0$ bands as recurrent level-$0$ bands, which is sometimes useful for concise statements.
\end{itemize}
\noindent We say that a level-$n$ bone is recurrent if its corresponding level-$n$ band is recurrent.
\end{defn}

A continuous map between two directed graphs is {\it a graph homomorphism} if it sends vertices to vertices and edges to edges preserving directions. From a finite subdivision rule $\R$, we have defined two directed graphs $\Ecal$ and $\Bcal$. There are natural graph homomorphisms $\iota,\tau:\Bcal \rightarrow \Ecal$ defined by $\iota([\band])=[e_1]$ and $\tau([\band])=[e_2]$. The next lemma follows from the fact that $\iota$ and $\tau$ are homomorphisms.

\begin{lem}\label{lem:recurrentbandandedge}
If $\band$ is a periodic level-$0$ band, then $e_1$ and $e_2$ are periodic edges. If $\bandn$ is a level-$n$ recurrent subband of $\band$, then the sides $e_1^n$ and $e_2^n$ of $\bandn$ are level-$n$ recurrent subedges of $e_1$ and $e_2$.
\end{lem}

\subsection{Parents and children}\label{sec:ParentsChildren}
We define parents and children for various objects regarding finite subdivision rules. The following are some important properties of the parent-child relationship.  For any $i>0$, $\bf{X}^i$ stands for a level-$i$ object, which can be an edge, a band, or a curve consisting of the bones of bands.

\begin{itemize}
    \item[] {\bf (Transitivity)} For $n>m>l\ge0$, if $\bf{X}^n$ is a child of $\bf{X}^m$ and $\bf{X}^m$ is a child of object $\bf{X}^l$, then $\bf{X}^n$ is a child of $\bf{X}^l$. A similar statement holds for parents.
    \item [] {\bf (Unique existence of parents)} For $n>m\ge0$ every level-$n$ object $\bf{X}^n$ has a unique level-$m$ parent $\bf{X}^m$. If $\bf{X}^n$ is recurrent, then so is $\bf{X}^m$.
    \item[] {\bf (Existence of recurrent children)} For $n>m\ge0$ every level-$m$ recurrent $\bf{X}^m$ has at least one level-$n$ child $\bf{X}^n$ that is also recurrent. We note that it does not work for non-expanded curves consisting of more than one bones in general.
\end{itemize}

\proofstep{Edges}
Suppose that a level-$n$ edge $e^n$ is a subedge of a level-$m$ edge $e^m$ where $n>m$. Then we say that $e^n$ is a {\it level-$n$ child} of $e^m$ and $e^m$ is a {\it level-$m$ parent} of $e^n$.

The transitivity is straightforward. If both $e^n$ and $e^m$ are subedges of a level-$0$ edge $e$, then they correspond to directed paths in $\Ecal$ of length $n$ and $m$, say $p$ and $p'$ respectively, such that both $p$ and $p'$ start from $[e]$ and $p'$ is the first length-$m$ restriction of $p$. Then the unique existence of parents follow. The existence of recurrent children follows from Proposition \ref{prop:ExtensionRecurrentPath}.

\proofstep{Bands and bones} Suppose that a level-$n$ band $b^n$ is a subband of a level-$m$ band $b^m$ for some $n>m$. Then we say that $b_n$ is a {\it child} of $b_m$ and $b_m$ is a {\it parent} $b_n$. The transitivity, the unique existence of parents, and the existence of recurrent children follow from a similar argument used in the case of edges.

We define parents and children for bones according to the parents-children relationship of their corresponding bands.

\proofstep{Non-expanded curves}
Let $I$ be a closed interval with integer ends, such as $[k,l]$, $(-\infty,k]$, $[k,\infty)$, or $(-\infty,\infty)$ for $k<l\in \Zbb$. For $n>m\ge0$, let $\gamma^n:I\to \R^n(S_\R)$ and $\gamma^m:I\to \R^m(S_\R)$ be level-$n$ and level-$m$ non-expanded curves respectively. Recall that $\gamma^n([i,i+1])$ (resp. $\gamma^m([i,i+1])$) is a level-$n$ (resp. level-$m$) bone for every $i\in \Zbb$ with $[i,i+1]\subset I$. If $\gamma^n([i,i+1])$ is a level-$n$ child of $\gamma^n([i,i+1])$ for every $i\in \Zbb$ with $[i,i+1]\subset I$, then we say that $\gamma^n$ is a {\it level-$n$ child of $\gamma^m$} and {\it $\gamma^m$ is a level-$m$ parent of $\gamma^n$}.

The transitivity and the unique existence of parents follow from a similar argument used before. However, the existence of recurrent children does not work for curves in general; The level-$n$ children of level-$m$ bones constituting $\gamma^m$ may not be joined as they are at level-$m$.

\begin{defn}[Genealogical sequence of non-expanded curves]
Let $\R$ be a finite subdivision rule. Let $I$ denote a closed interval $[k,l]$, $(-\infty, k]$, $[k,\infty)$, or $(-\infty,\infty)$ for $k<l\in \Zbb$. A sequence of level-$n$ non-expanded curves $\{\gamma^n:I\to \R^n(S_\R)\}_{n\ge0}$ is {\it genealogical} if $\gamma^{n+1}$ is a child of $\gamma^n$ for every $n\ge0$.
\end{defn}

\section{Levy cycle and genealogical sequence of homotopically infinite curves}\label{sec:alg aspect of Levy cycle}

The purpose of this section is to prove the following theorem.

\begin{thm}\label{thm:LevyHomoInf}
Let $\R$ be a finite subdivision rule and $A\subset \V(S_\R)$ be a set of marked points. Suppose that the subdivision map $f:\R(S_\R)\to S_\R$ is not doubly covered by a torus endomorphism. Then $f:(S^2,A)\righttoleftarrow$ has a Levy cycle if and only if there is a genealogical sequence of non-expanded recurrent bi-infinite curves $\{\gamma^n:(-\infty,\infty) \to \R^n(S_\R)\}$ such that each $\gamma^n$ is homotopically infinite with respect to a hyperbolic orbisphere structure $\ord:A\to [2,\infty]_\Zbb$ (hence with respect to any hyperbolic orbisphere structure because the definition of Levy cycles is independent of the choice of orbisphere structures).
\end{thm}

The ``\,only if\,'' direction is not hard. We can use a Levy cycle to construct the desired genealogical sequence of non-expanded curves. The other direction, however, is non-trivial. Even if we have a genealogical sequence of non-expanded curves, it is difficult to explicitly find a Levy cycle. We prove the existence of a Levy cycle in a non-constructive way using an algebraic machinery, called self-similar groups \cite{Nek_selfsimilargroup}. We use the term ``orbisphere bisets'' rather than self-similar groups in order to be consistent with our main reference \cite{BartholdiDudko_expanding}.

\subsection{Contracting orbisphere bisets}\label{sec:ContOrbiBiset}
Let $A$ be a finite subset of the sphere $S^2$. An {\it orbisphere structure on $(S^2,A)$} is an order function $\ord:A\to [2,\infty]_\Zbb$. We say that $\ord$ is an {\it orbisphere structure of a \pcf branched covering $\TmapA$} if it satisfies
\begin{enumerate}
    \item [(1)] $\ord(a)\cdot \deg_f(a)~|~\ord(f(a))$ for every $a\in S^2$ where $\ord(a)=1$ for $a\notin A$, and
    \item [(2)] $\ord(a)= \infty$ only if $a\in A$ is a Fatou point.
\end{enumerate}

\noindent In (1), $\infty$ is considered as a multiple of any integer or $\infty$ itself. It follows that $\ord(a)=\infty$ for every $a$ in a periodic cycle containing a critical point. The triple $(S^2,A,\ord)$ is called an {\it orbisphere}.

The {\it orbisphere group} $\pi_1(S^2,A,\ord)$ of an orbisphere $(S^2,A,\ord)$ is defined by
\[
    \pi_1(S^2,A,\ord)=\pi_1(S^2\setminus A) \left/\, \left<\{\gamma_a^{\ord(a)}~|~a\in A~\mathrm{and}~\ord(a)\neq \infty \}\right>\right.
\]
where $\gamma_a$ is a peripheral loop of $a\in A$ and $\gamma_a^{\ord(a)}=1$ if $\ord(a)=\infty$.

\begin{rem}
When $A=P_f$, the order of $x$, $\ord(x)$, is usually defined as the least common multiple of $\{\deg_{f^n}(y)~|~y\in f^{-n}(x)~\mathrm{for}~n \in \Nbb\}$. When $A$ contains periodic points which do not belong to $P_f$, however, the least common multiples of their degrees equal to one so that it is not an orbisphere structure we care about. The reason that we require $\ord(a)>1$ for every $a\in A$ is that if $\ord(a)=1$ then $\gamma_a$ vanishes in $\pi_1(S^2,A,\ord)$ so that algebraic properties of $\pi_1(S^2,A,\ord)$ cannot carry any information of the $a\in A$.
\end{rem}

The {\it Euler characteristic $\chi(S^2,A,\ord)$} of an orbisphere $(S^2,A,\ord)$ is defined by
\begin{equation}
    \chi(S^2,A,\ord)=2+\sum_{a\in A}\left(\frac{1}{\ord(a)}-1\right).
\end{equation}
The orbisphere $(S^2,A,\ord)$ is {\it hyperbolic} if $\chi(S^2,A,\ord)<0$.

\vspace{5pt}
Let $p$ be a base point of $\pi_1(S^2,A,\ord)$. Define a set $B(f,A,\ord)$ by
\[
\left\{ \gamma:[0,1]\to S^2 \setminus A ~|~ \gamma(0)=f(\gamma(1))=p \right\}/~\mathrm{homotopy~relative~to}~(A,\ord).
\]
By the homotopy relative to $(A,\ord)$ we mean a homotopy relative to $A$ together with one more homotopy condition: For any $a\in A$ with $\ord(a)<\infty$, the $\ord(a)^{th}$ power of the peripheral loop of $a$ is considered to be homotopically trivial.

There is a natural $\pi_1(S^2,A,\ord)$-action on $B(f,A,\ord)$ from both left and right. More precisely, for $\gamma_1,\gamma_2\in \pi_1(S^2,A,\ord)$ and for $\delta \in B(f,A)$, the product $\gamma_1 \cdot \delta \cdot \gamma_2$ is the concatenation of $\gamma_1$, $\delta$, and the lift of $\gamma_2$ through $f$ starting at the endpoint of $\delta$, in order. The left action is free, and the right action is transitive. The set $B(f,A)$ equipped with the left and right actions is called {\it the orbisphere biset} of $(S^2,A,\ord)$. If an orbisphere structure $\ord: A \to [2,\infty]_\Zbb$ is given, we implicitly assume that $B(f,A,\ord)$ has the left and right $\pi_1(S^2,A,\ord)$-actions. When an orbisphere $(S^2,A,\ord)$ is understood in the context, we simply write $B(f)$ for $B(f,A,\ord)$.

\begin{caut}
There are two conventions depending on whether you concatenate curves from right to left or from left to right in the operation of orbisphere group. Many documents, including \cite{Nek_selfsimilargroup}, follow the ``from right to left'' convention, but we will follow the ``from left to right'' convention for the sake of convenience in citing \cite{BartholdiDudko_expanding}. Thus a biset has a free left action and a transitive right action, which is opposite to a bimodule in \cite{Nek_selfsimilargroup}.
\end{caut}

A {\it tensor square} $B(f) \otimes B(f)$ can be defined in two different ways. Topologically, the tensor product $\delta_1 \otimes \delta_2$ for $\delta_1,\delta_2\in B(f)$ is defined as a concatenation of $\delta_1$ and the lift of $\delta_2$ starting at the endpoint of $\delta_1$. Algebraically, it is a quotient of $B(f) \times B(f)$ by the relation $(\delta_1 \cdot \gamma) \otimes \delta_2=\delta_1 \otimes (\gamma \cdot \delta_2)$. The left and right actions naturally extend to $B(f) \otimes B(f)$. Similarly, $B(f)^{\otimes n}$ has a left free and a right transitive $\pi_1(S^2,A,\ord)$-actions for any $n \ge 1$.

A {\it basis $X$ for $B(f)$} is a collection of representatives of left orbits of the biset $B(f)$. Its cardinality $|X|$ is the same as the degree of $f$. For any $n\ge 1$, the tensor power $X^{\otimes n}$ of $X$ is a basis for $B(f)^{\otimes n}$. Topologically, a basis is a choice of curves from the base point $p$ of $\pi_1(S^2,A,\ord)$ to the $d$ preimages $f^{-1}(p)=\{p_1,p_2,\dots,p_d\}$ where $d=\deg(f)$. Let $\delta_i$ be a curve from $p$ to $p_i$ for $i\in[1,d]_\Zbb$. Then $\{\delta_1,\delta_2,\dots,\delta_d\}$ be a basis for $B(f)$, and every basis of $B(f)$ is of this form. Fix $n\ge 1$. Let $i_1,i_2,\dots,i_n\in[1,d]_\Zbb$. We simply write
\[
    \delta_{i_1i_2\dots i_n}:=\delta_{i_1} \otimes \delta_{i_2} \otimes \dots \otimes \delta_{i_n},
\]
which gives a bijection $([1,d]_\Zbb)^n\leftrightarrow X^{\otimes n}$.

\begin{defn}[Contracting biset and nucleus]
Let $f:(S^2,A)\righttoleftarrow$ be a marked post-critically finite branched covering and $\ord:A \to [2,\infty]_\Zbb$ be an orbisphere structure. Let $X$ be a basis for $B(f)$. The orbisphere biset $B(f)$ is {\it contracting} if there exists a finite subset $\Ncal \subset \pi_1(S^2,A,\ord)$ satisfying the following: For every $g \in \pi_1(S^2,A,\ord)$, the inclusion $X^{\otimes n} \cdot g\subset \Ncal \cdot X^{\otimes n}$ holds for every sufficiently large $n>0$. The minimal $\Ncal$ satisfying this property is the {\it nucleus} of $(B(f),X)$.
\end{defn}

The contracting property does not depend on the choice of basis \cite[Corollary 2.11.7]{Nek_selfsimilargroup}, but the nucleus does. See Remark \ref{rem:IndepOrbistr} for the independence of the choice of orbisphere structures.

\begin{defn}[B\"{o}ttcher expanding map and Local rigidity]\label{defn:BottcherExpLocalRigid}
Let $\TmapA$ be a \pcf branched covering and $\ord:A\to [2,\infty]_\Zbb$ be an orbisphere structure. Denote by $A^\infty$ the subset of $A$ consisting of $a\in A$ with $\ord(a)=\infty$. Then $\TmapA$ is {\it B\"{o}ttcher (metrically) expanding} if there is a length metric $\mu$ on $S^2\setminus A^\infty$, satisfying the following conditions.
\begin{itemize}
    \item For every rectifiable curve $\gamma:[0,1]\to S^2\setminus A^\infty$, the length of any lift of $\gamma$ through $f$ is strictly less the length of $\gamma$, and
    \item ({\it Local rigidity near critical cycles}) For every periodic point $a\in A^\infty$, the first return map of $f$ near $a$ is locally topologically conjugate to $z\mapsto z^{\deg_a(f^n)}$, where $n$ is the period of $a$.
\end{itemize}
\end{defn}

Every B\"{o}ttcher expanding map $\TmapA$ also has the Fatou set and the Julia set, which have similar properties of the Fatou and Julia sets of rational maps, see \cite{BartholdiDudko_expanding}.

A \pcf rational map $f$ is B\"{o}ttcher expanding since it has the B\"{o}ttcher coordinates and enjoys the Schwarz lemma about the conformal metric. The next theorem, which follows from \cite[Theorem A, Corollary 1.2]{BartholdiDudko_expanding}, is an analogue of Thurston's characterization and rigidity.

\begin{thm}[{\cite[Theorem\,A,~Corollary\,1.2]{BartholdiDudko_expanding}}] \label{thm:Bartholdi-dudko}
Let $f:(S^2,A) \righttoleftarrow$ be a \pcf branched covering which is not doubly covered by a torus endomorphism and $\ord: A \to [2,\infty]_\Zbb$ be an orbisphere structure. Then the following are equivalent
\begin{itemize}
    \item [(1)] $\TmapA$ is combinatorially equivalent to a B\"{o}ttcher expanding map.
    \item [(2)] The orbisphere biset $B(f,A,\ord)$ is contracting.
    \item [(3)] $\TmapA$ has degree>1 and does not have a Levy-cycle.
\end{itemize}
Moreover, if exists, the B\"{o}ttcher expanding map is unique in the combinatorial equivalent class up to topological conjugacy.
\end{thm}

\begin{rem}
In \cite{BartholdiDudko_expanding}, the orbisphere structure used in Theorem A is required that $\ord(a)=\infty$ if and only if $a$ is a periodic Fatou point, which is a little stronger than the definition of orbisphere structures in this paper. But this slight generalization follows almost immediately.
\end{rem}

\begin{rem}\label{rem:IndepOrbistr}
The definition of $B(f,A,\ord)$ depends on the orbisphere structure $\ord:A\to [2,\infty]_\Zbb$, but the definition of Levy cycles of $\TmapA$ doesn't. Hence Theorem \ref{thm:Bartholdi-dudko} implies that whether or not $B(f,A,\ord)$ is contracting is also independent of the choice of orbisphere structure.
\end{rem}

\subsection{Semi-conjugacy to B\"{o}ttcher expanding maps}
The idea of semi-conjugacy was introduced by Rees \cite{Rees_semiconj} and Shishikura \cite{ShiShi_Rees} to show that, for any mateable pair of \pcf polynomials, the topological mating is topologically conjugate to the corresponding rational map. Then the idea was further developed by Cui-Peng-Tan \cite{CPT_ReesShishikura} to a form that can be applied for not only matings but also general \pcf branched coverings and rational maps. We slightly further generalize the theorem of Cui-Peng-Tan by applying Bartholdi-Dudko's recent work on B\"{o}ttcher expanding maps \cite{BartholdiDudko_expanding}.

The next theorem is a generalization of \cite[Theorem 1.1, Corollary 1.2]{CPT_ReesShishikura} replacing rational maps by B\"{o}ttcher expanding maps.

\begin{thm}[Semi-conjugacies to B\"{o}ttcher expanding maps]\label{thm:semiconj}
Let $f:(S^2,A)\righttoleftarrow$ be a \pcf branched covering which is locally rigid near critical cycles. Suppose $f$ is combinatorially equivalent to a B\"{o}ttcher expanding map $F:(S^2,B)\righttoleftarrow$. Let $\Fcal_F$ and $\Jcal_F$ denote the Fatou and the Julia sets of $F$. Then there exists a semi-conjugacy $h:(S^2,A) \to (S^2,B)$ from $f$ to $F$, i.e., $h \circ f=F\circ h$, such that the following properties are satisfied.
\begin{itemize}
    \item $h^{-1}(w)$ is a singleton for $w\in \Fcal_F$ and a full continuum for $w \in \Jcal_F$.
    \item For $x,y\in S^2$ with $F(x)=y$, the set $h^{-1}(x)$ is a connected component of $f^{-1}(h^{-1}(y))$. Moreover, the degree of the map $f:h^{-1}(x)\to h^{-1}(y)$ is equal to $\deg_x(F)$; more precisely, for every $w\in h^{-1}(y)$ we have
    \[
        \sum\limits_{z\in h^{-1}(x) \cap f^{-1}(w)} \deg_z f=\deg_x(F).
    \]
    \item If $E\subset S^2$ is a continuum, then $h^{-1}(E)$ is a continuum.
    \item $f(h^{-1}(E))=h^{-1}(F(E))$ for every $E \subset S^2$.
    \item $f^{-1}(\widehat{E})=\widehat{f^{-1}(E)}$ for every $E\subset S^2$, where $\widehat{E}:=h^{-1}(h(E))$.
\end{itemize}
\end{thm}
\begin{proof}
In \cite{CPT_ReesShishikura}, the complex structure of the Riemann sphere was used for two purposes: (i) the conformal metric is expanding, and (ii) there are B\"{o}ttcher coordinates near critical cycles. Since B\"{o}ttcher expanding maps also have these two properties, the proof in \cite{CPT_ReesShishikura} still works for this setting. For example,
\begin{itemize}
\item In \cite{CPT_ReesShishikura}, they use \pcf branched coverings on the Riemann sphere $\hCbb$ that are holomorphic near critical cycles. Given a \pcf branched covering (on the topological sphere) which is locally rigid near critical cycles, we may define a holomorphic structure on the sphere so that the branched covering is holomorphic near critical cycles.
\item The orbifold metric in \cite[Section 2]{CPT_ReesShishikura} can be replaced by the Riemannian orbifold metric in \cite{BartholdiDudko_expanding}. Then we still have the expansion property of homotopic lengths of paths.
\end{itemize}

\end{proof}

\begin{defn}[Homotopic length]
Let $(X,\mu)$ be a metric space and $\gamma$ be a curve joining $x$ to $y$. Then the {\it homotopic length} $l_\mu([\gamma])$ is the infimum of the lengths of rectifiable curves that joins $x$ to $y$ and homotopic to $\gamma$ relative to $\{x,y\}$.
\end{defn}

\begin{lem}\label{lem:BasisUnifBdd}
Let $\TmapA$ be a \pcf branched covering of degree $d\ge 2$ which is not doubly covered by a torus endomorphism. Suppose $\ord:A \to [2,\infty]_\Zbb$ is an orbisphere structure and $X$ be a basis for the biset $B(f)$. Suppose that $f$ does not have a Levy cycle such that there exists a semi-conjugacy $h:(S^2,A)\to (S^2,B)$ where $F:(S^2,B)\righttoleftarrow$ is a B\"{o}ttcher expanding map, expanding about a metric $\mu$, that is combinatorially equivalent to $f$. Then there exists $C>0$ such that $l_\mu([h(w)])<C$ for every $n>0$ and $w\in X^{\otimes n}$. Here $w$ is considered as a curve joining the base point $p$ of $\pi_1(X,A,\ord)$ to a point in the preimage $f^{-n}(p)$, as described in Section \ref{sec:ContOrbiBiset}.
\end{lem}

\begin{proof}

Since the metric $\mu$ blows up near marked points of infinite order, we should take a compact subset away the points of infinite order. Let $B^\infty=h(A^{\infty})$ where $A^\infty$ is the subset of $A$ consisting of elements having infinite order. There exists a small neighborhood $U$ of $B^\infty$ such that for $M:=S^2\setminus U$ and $M'=f^{-1}(M)$ we have $M'\subset M$ and $f:M'\to M$ being a branched covering which has a uniform expanding constant $\lambda>1$ in the following sense: For every curve $\gamma\subset M$ and any of its lifting $\gamma'$ through $f$, we have
\[
    \lambda \cdot l_\mu([\gamma'])< l_\mu([\gamma]).
\]

Let $X=\{\delta_1,\delta_2,\dots, \delta_d\}$ where each $\delta_i$ joins the base point $p$ of $\pi_1(S^2,A,\ord)$ to one of the $d$ preimages $f^{-1}(p)$. Define $D>0$ by
\[
    D=\max_{1 \le j \le d} l_\mu([h(\delta_j)]).
\]
Let $w=\delta_{i_1i_2\dots i_n}\in X^{\otimes n}$ where $i_l\in\{1,2,\dots,d\}$. Then $w$ is the concatenation of $\delta_{i_1}$, a lift of $\delta_{i_2}$ through $f$, a lift of $\delta_{i_3}$ through $f^2$, and so one. Every curve $\delta_i$ and its any lifting can be contained in $M$ up to homotopy. Hence we have
\[
    l_\mu([h(w)]<D\cdot\left(1+\frac{1}{\lambda}+\frac{1}{\lambda^2}+\cdots \right)=D\cdot \frac{\lambda}{\lambda-1}.
\]

\end{proof}

\subsection{Homotopically infinite non-expanded curves and Levy cycles}

\begin{defn}[Homotopically infinite curves]
Let $(S^2,A,\ord)$ be a hyperbolic orbisphere and $p:\Dbb \to S^2\setminus A^\infty$ is the orbifold universal covering map. A closed curve $\gamma:[0,1]\to S^2 \setminus A$ is {\it homotopically infinite with respect to $\ord$} if for a connected component $\widetilde{\gamma}$ of $p^{-1}(\gamma)$, both ends of $\widetilde{\gamma}$ have a limit point on the boundary $\partial \Dbb$. A half-infinite curve $\gamma:[0,\infty) \to S^2\setminus A$ (resp.\@ bi-infinite curve $\gamma:(-\infty, \infty) \to S^2 \setminus A$) is {\it homotopically infinte with respect to $\ord$} if the end (resp.\@ both ends) of its lift $\widetilde{\gamma}$ has a limit point. 
\end{defn}

The next proposition is immediate from standard properties of the hyperbolic geometry.

\begin{prop}
	Let $(S^2,A,\ord)$ be a hyperbolic orbisphere. A closed curve $\gamma$ is homotopically infinite if and only if $\gamma$ is neither homotopically trivial in $S^2\setminus A$ nor homotopic relative to $A$ to some iterate of the peripheral loop of $a\in A$ with $\ord(a)<\infty$.
\end{prop}

Each of the following two propositions is each direction of the equivalence in Theorem \ref{thm:LevyHomoInf}. We split them because the ideas of the proofs are quite different.

\begin{prop}\label{prop:LevytoGeneaSeq}
Let $\R$ be a finite subdivision rule and $A\subset \V(S_\R)$ be a set of marked points. Suppose that the subdivision map $f:\R(S_\R)\to S_\R$ is not doubly covered by a torus endomorphism. If $\TmapA$ has a Levy cycle, then there is a genealogical sequence of non-expanded closed curves $\{ \gamma^n:I \to \R^n(S_\R)\}_{n \ge 0}$ that are recurrent and homotopically infinite with respect to any hyperbolic orbisphere structure $\ord:A\to[0,\infty]_\Zbb$. Moreover, by iterating travelling along the closed curves, we may assume that each $\gamma^n$ is a bi-infinite curve. 
\end{prop}

\begin{proof}

Assume there exists a Levy cycle, i.e., there are an integer $p>0$ and an essential simple closed curve $\gamma$ of $(S^2,A)$ such that a connected component $\gamma'$ of $f^{-p}(\gamma)$ is isotopic to $\gamma$ relative to $A$ and $\deg(f^p|_{\gamma'})=1$. We may assume $\gamma$ is $S_\R$-taut so that $\gamma'$ is $\R^p(S_\R)$-taut.

\vspace{5pt}
\noindent{\it Claim: We may assume that $\gamma$ and $\gamma'$ are $S_\R$-combinatorially equivalent.}
\vspace{-5pt}
\begin{proof}[Proof of Claim]
For every $k\ge1$, $f^{-kp}(\gamma)$ has a connected component $\gamma_{kp}$ that is isotopic to $\gamma$ relative to $A$ and $\deg(f^{kp}|_{\gamma_{kp}})=1$. For every $k\ge 1$, it follows from Proposition \ref{prop:monoton length} that $l_0(\gamma_{kp}) \le l_{kp}(\gamma_{kp})$ and from $\deg(f^{kp}|_{\gamma_{kp}})=1$ that $l_{kp}(\gamma_{kp})=l_0(\gamma)$, where $l_n(\cdot)$ means $l_{\R^n(S_\R)}(\cdot)$. By Lemma \ref{lem:FinUptoCombiEq}, there exist $k_1>k_2>0$ such that $\gamma_{k_1p}$ and $\gamma_{k_2p}$ are combinatorially equivalent relative to $S_\R$. Then we can replace $\gamma$ by $\gamma_{k_2 p}$ and $p$ by $(k_1 - k_2)p$.
\end{proof}
It follows from the claim that we can parametrize $\gamma$ and $\gamma'$ such that $\gamma':I\to \R^p(S_\R)$ is a level-$p$ non-expanded closed curve and $\gamma:I \to S_\R$ is the level-0 parent of $\gamma'$ for $I=[0,l]$ for some $l\in \Zbb_{>0}$. Being essential relative to $A$, $\gamma$ and $\gamma'$ are, in particular, homotopically infinite relative to any hyperbolic orbisphere structure.

Let $\gamma^0:=\gamma$ and $\gamma^p:=\gamma'$. By lifting an isotopy between $\gamma^0$ and $\gamma^p$ through $f^p$, we have an isotopy from $\gamma^p$ to a level-$2p$ non-expanded curve $\gamma^{2p}:I\to \R^{2p}(S_\R)$ such that $\gamma^p$ is the level-$p$ parent of $\gamma^{2p}$. This way, we obtain a sequence of level-$kp$ non-expanded curves $\{\gamma^{kp}:I\to \R^{kp}(S_\R)\}_{k\ge0}$ such that (1) $\gamma^{kp}$ is the level-$kp$ parent of $\gamma^{(k+1)p}$ for every $k\ge0$ and (2) $f^p:\gamma^{(k+1)p}\to \gamma^{kp}$ is a homeomorphism. If we identify $\gamma^p$ with $\gamma^0$ via an isotopy preserving the 1-skeleton of $S_\R$, the map $f^p:\gamma^p \to \gamma^0$ can be considered as a rotation of a circle of length $l$ by an integer. Hence, there exists $k_0>0$ such that for every $n>0$ the level-$nk_0 p$ bone $\gamma^{k_0 np}([i,i+1])$ is mapped to $\gamma^0([i,i+1])$ by $f^{k_0np}$, which implies that $\gamma^{kp}$ is recurrent for every $k\ge 0$. For every $m>0$ that is not a multiple of $p$, we define $\gamma^m$ as the level-$m$ parent of $\gamma^{kp}$ for some $k>0$ with $kp>m$, which is well-defined up to $\R^m(S_\R)$-combinatorial equivalence.

Since each $\gamma^m$ is homotopic to an essential simple closed curve of $(S^2,A)$, it is homotopically infinite with respect to any hyperbolic orbisphere structure.
\end{proof}

\begin{prop}\label{prop:GeneaSeqtoLevy}
Let $\R$ be a finite subdivision rule and $A\subset \V(S_\R)$ be a set of marked points. Suppose that the subdivision map $f:\R(S_\R)\to S_\R$ is not doubly covered by a torus endomorphism. Let $\ord:A\to[2,\infty]_\Zbb$ be a hyperbolic orbisphere structure. If there is a genealogical sequence of non-expanded bi-infinite curves $\{ \gamma^n:(-\infty,\infty) \to \R^n(S_\R)\}_{n \ge 0}$ that are recurrent and homotopically infinite with respect to $\ord$, then $\TmapA$ has a Levy cycle.
\end{prop}

\begin{proof}
Suppose that $\TmapA$ does not have a Levy cycle. Then for any basis $X$ for the biset $B(f)$, there is a nucleus $\Ncal$, which is a finite set. For every $k>0$, we use the sequence of finite restrictions $\{\gamma^n|_{[0,k]}\}_{n\ge0}$ to obtain an element $h_k\in \Ncal$ such that $\|h_k\|\to \infty$. Here $\| \cdot \|$ is a distance on $\pi_1(S^2,A,\ord)$ with respect to a generating set, which does not need to be specified. Then the nucleus $\Ncal$ has infinitely many elements, which contradicts to the assumption that the biset is contracting.

Recall that we use $p$ for the base point of $\pi_1(S^2,A,\ord)$, each element of a basis $X$ for $B(f)$ is a curve from $p$ to one of its $f$-preimages $f^{-1}(p)$, and an element $w\in X^{\otimes n}$ is a concatenation of curves connecting an $f^i$-preimage to $f^{i+1}$-preimage that are liftings of elements in $X$.

\proofstep{Step 1: Construction of $h_k$}
Fix $k>0$. There is an infinite sequence $n_1<n_2<\cdots$, which depends on $k$, so that $\gamma^{n_i}([0,k])$ consists of the bones of the bands of the same types, i.e., there exists level-$0$ bands $b_0,b_1,\dots b_{k-1}$ (possibly repeated) such that for every $j\in [0,k-1]$, $\gamma^{n_i}([j,j+1])$ is the bone of a level-$n_i$ band of type $b_j$, which is independent of $i>0$.

For every level-$0$ edge $e$, we fix a point $m_e \in \mathrm{int}(e)$ and call it the {\it midpoint} of $e$. We assume that a bone of a level-$0$ band is chosen in the homotopy class in such a way that their endpoints are the midpoints of level-$0$ edges. For every level-$0$ edge $e$, we also choose a path $\delta_e$ from $p$ to $m_e$. Then, for each level-$0$ band $b=\band$, we can assign an element
\[
    g_{b_i}:=\delta_{e_1} \cdot bone(b_i) \cdot \overline{\delta}_{e_2} \in \pi_1(S^2,A,\ord),
\]
where the overline $\overline{\,\cdot\,}$ means the reverse of the orientation of a curve and $bone(b_i)$ means the bone of a band $b_i$.

For $g\in \pi_1(S,A,\ord)$, we define $N(g)\subset \pi_1(S,A,\ord)$ by the collection of elements $h$ with the following property: For infinitely many $n>0$ there exist $v,w\in X^{\otimes n}$ such that $h \cdot v = w \cdot g$. We remark that $N(g) \subset \Ncal$.

\begin{figure}[h!]
    \centering
       \def\svgwidth{0.7\textwidth}
        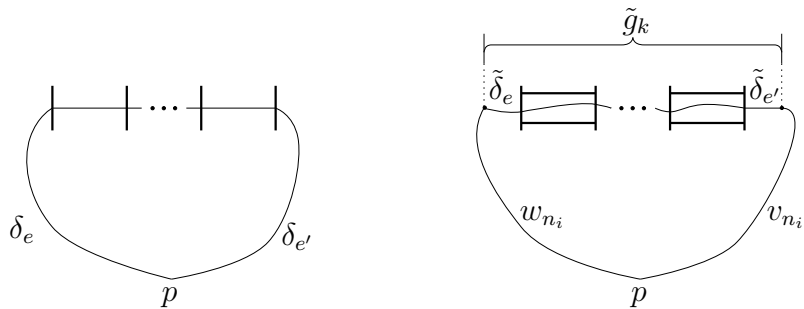
	\caption{The left figure is $g_k$ drawn in $S_\R$, and the right figure is $h_{k,n_i}$ drawn in $\R^{n_i}(S_\R)$. The bold line segments are portions of $1$-skeletons of $S_\R$ and $\R^{n_i}(S_\R)$.}
    \label{fig:liftingsametype}
\end{figure}

\noindent {\it Claim: There exists $C>0$, independent of $k$, such that $N(g_k:=g_{b_1}\cdots g_{b_k})$ contains at least one element $h_k$ with $d(h_k,g_k) < C$.}
\vspace{-5pt}
\begin{proof}[Proof of Claim]
Recall that for every $i>0$ and $j\in [0,k-1]_\Zbb$, $\gamma^{n_i}([j,j+1])$ is the bone of a level-$n_i$ band of type $b_j$, which is independent of $i$. Let $e$ and $e'$ be level-$0$ edges such that their midpoints $m_e$ and $m_{e'}$ are the endpoints of $\gamma^0([0,k])$.

Let $w_{n_i}\in X^{\otimes n_i}$, which will be specified soon. Let $v_{n_i}\in X^{\otimes n_i}$ and $h_{k,n_i}\in \pi_1(S^2,A,\ord)$ be defined by
\[
    h_{k,n_i}\cdot v_{n_i}=w_{n_i}\cdot g_k.
\]
Let $\tilde{g}_k$ be the lift of $g_k$ through $f^{n_i}$ starting from the terminal point of $w_{n_i}$. Let $\tilde{\delta}_e$ and $\tilde{\delta}_{e'}$ be the parts of $\tilde{g}_k$ where $\delta_e$ and $\delta_{e'}$ are lifted. We specify $w_{n_i}$ as an element of $X^{\otimes n_i}$ satisfying the following: the curve $g_k$ with $\delta_e$ and $\delta_{e'}$ being truncated is $S_\R$-combinatorially equivalent to $\tilde{g}_k$ with $\tilde{\delta}_e$ and $\tilde{\delta}_{e'}$ being truncated. See Figure \ref{fig:liftingsametype}.

Then $h_{k,n_i}$ and $g_k$ differ by the pre- and post-composition with loops $w_{n_i}\cdot\tilde{\delta}_e \cdot \overline{\delta}_e$ and $\overline{\delta}_{e'} \cdot \delta_{e'} \cdot \overline{v}_{n_i}$. We have (1) a uniform upper bound on the homotopic length of $v_{n_i}$ and $w_{n_i}$ (in the projection to the B\"{o}ttcher expanding map) by Lemma \ref{lem:BasisUnifBdd}, (2) a uniform upper bound on the intersection between $\delta_e$ and $S_\R^{(1)}$, and (3) the upper bound in (2) is also an upper bound of the intersection between any level-$n$ lift $\tilde{\delta}_e$ and $\R^n(S_\R)^{(1)}$. Therefore, we have $\|w_{n_i}\cdot\tilde{\delta}_e \cdot \overline{\delta}_e\|, \|\overline{\delta}_{e'} \cdot \delta_{e'} \cdot \overline{v}_{n_i}\|< C/2$ for some $C$ so that $d(g_k,h_{k,n_i})<C$. Since there are only finitely many elements of $\pi_1(S^2,A,\ord)$ within the distance $C$ from $g_k$, there exists $h_k$ such that $h_k=h_{k,n_i}$ for infinitely many $i$'s.
\end{proof}

\proofstep{Step 2: Proof of $\|h_k\|\to \infty$}
Let $g_k$ and $h_k$ be as defined in Step 1. Since $d(h_k,g_k)<C$, it suffices to show $\|g_k\|\to \infty$ as $k$ tends to $\infty$.

Let $e_0$ and $e_k$ be the level-$0$ edges whose midpoints are the endpoints of $\gamma^0([0,k])$ so that $g_k=\delta_{e_0} \cdot \gamma^0([0,k]) \cdot \overline{\delta}_{e_k}$. Then $\|g_k\|\to \infty$ follows from the condition that $\gamma^0$ is homotopically infinite.

\end{proof}

\section{Non-expanding spines}\label{sec:nonexpspine}

From Theorem \ref{thm:LevyHomoInf}, we know that the existence of a Levy cycle is equivalent to the existence of a genealogical sequence of homotopically infinite recurrent non-expanded curves. Then, how can we detect the existence such a sequence? The direct search for the genealogical sequence could be more complicated then the search for the Levy cycle. One motivation for non-expanding spines is to have a simpler object with which we can efficiently detect the genealogical sequence.

Since recurrent non-expanded level-$n$ curves are concatenations of level-$n$ recurrent bones, i.e., the bones of level-$n$ recurrent bands, it is natural to consider the union of level-$n$ recurrent bones. The level-$n$ non-expanding spine $N^n$ is, roughly speaking, the union of level-$n$ recurrent bones equipped with a natural train-track structure.

\proofstep{Motivation for the use of train-tracks} Let us see Figure \ref{fig:TripotTraintrack}. Let $t$ denote the hexagonal tile and $e_1,e_2,e_3$ denote three of its boundary edges. Suppose that $b_1=(t;e_1,e_2)$ and $b_2=(t;e_1,e_3)$ are level-0 recurrent bands so that each has two level-1 recurrent subbands of type $b_1$ and $b_2$. Then each level-1 recurrent subband of type $b_1$ or $b_2$ has two level-$2$ recurrent subbands of type $b_1$ and $b_2$. Assume that we draw the bones of these bands. At level-0, we have two curves each of which joins $e_1$ to $e_2$ or $e_3$. As a ``collection'' of these two curves, we might consider a tripod whose leaves are on $e_1,e_2,$ and $e_3$. However, the tripod contains an unintended curve which joins $e_2$ to $e_3$. To exclude such a curve, we use the idea of train-tracks, which makes the curve joining $e_2$ and $e_3$ illegal.

\begin{figure}[h!]
    \centering
        \def\svgwidth{0.8\textwidth}
        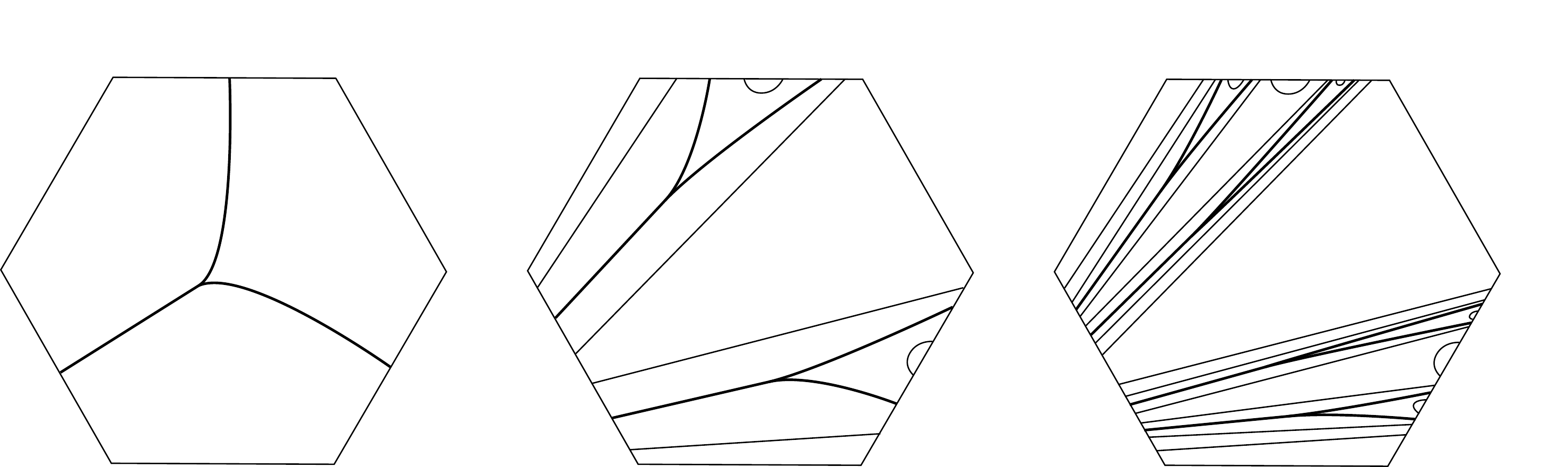
	\caption{}
    \label{fig:TripotTraintrack}
\end{figure}

\subsection{Train-track}

Let $G$ be a graph and $v\in G$ be a point, which could be a vertex or a point in the interior of an edge. A {\it direction at $v$} is a germ of continuous curves starting from $v$. The number of directions at $v$ is equal to the number of connected components of $U\setminus \{v\}$ where $U$ is a sufficiently small neighborhood of $v$. Denote the set of directions at $v$ by $D_v$.

\begin{defn}[Train-tracks and gates]
Let $G$ be a finite graph. For any vertex $v$ of $G$. A {\it train-track structure $\tau$ on $G$} is an assignment of an equivalence relation on $D_v$ for each $v\in \V(G)$. A {\it train-track} is a finite graph $G$ equipped with a train-track structure $\tau$ and denoted by $(G,\tau)$. For any $v\in G\setminus \V(G)$, $D_v$ has two directions and we define each equivalence class of $D_v$ to have each direction. Below is a list of definitions regarding train-tracks.
\begin{itemize}
    \item Each equivalence class of $D_v$ is called a {\it gate at $v$}.
    \item A {\it train path} is an oriented curve $\gamma$ in $G$ such that at every vertex $v$, the gate through which $\gamma$ comes to $v$ is different from the gate through which $\gamma$ goes out.
\end{itemize}
\end{defn}



\begin{defn}[Train-track map]
Let $T_1=(G_1,\tau_1)$ and $T_2=(G_2,\tau_2)$ be train-tracks. A {\it train-track map} $\phi:T_1\to T_2$ is a continuous map $\phi:G_1\to G_2$ that is locally injective on each edge such that for every $v \in G_1$ and $d_1,d_2\in D_v$, $\phi(d_1)$ and $\phi(d_2)$ are in the same gate at $\phi(v)$ if $d_1$ and $d_2$ are in the same gate.
\end{defn}

\begin{defn}[Homotopy relative to $(\partial S,X)$]
Let $f,g:(S,\partial S, X) \to (S,\partial S, X)$ be two continuous maps. We say that $f$ and $g$ are homotopic (resp.\@ isotopic) relative to $(\partial S,X)$ if there is a homotopy (resp. \@ isotopy) $\{H_t:S\to S\}_{t\in[0,1]}$ such that
\begin{itemize}
    \item $H_0=f$ and $H_1=g$,
    \item $H_t|_X:X\to X=id_X:X\to X$, and
    \item For every $t\in[0,1]$, $H_t$ sends every connected component of $\partial S \setminus X$ to itself.
\end{itemize}
\end{defn}

\begin{defn}[Graphs properly embedded in $(S,\partial S,X)$]\label{defn:ProperEmb}
Let $S$ be a compact surface with a finite set of marked points $X$. Some marked points may be on the boundary $\partial S$. A graph $G\subset S$ is {\it properly embedded in $(S,\partial S, X)$} if $(G,\partial G)$ is properly embedded in $(S,\partial S)$ and $G \subset S\setminus X$, where $\partial G$ is the set of leaves of $G$. We say that two graphs $G$ and $H$ properly embedded in $(S, \partial S, X)$ are {\it homotopic (resp. isotopic)} if there is an ambient homotopy $\{H_t:S\to S\}_{t\in[0,1]}$ (resp. isotopy) relative to $(\partial S,X)$ such that $H_0=id_S:S\to S$ and $H_1(G)=H$.
\end{defn}

\begin{defn}[Train-tracks in surfaces]
Let $S$ be a compact surface possibly with boundary $\partial S$ and a finite set of marked points $X\subset S$. By a train-track in $(S,\partial S, X)$, we mean a train-track $(G,\tau)$ of a graph $G$ satisfying (1) $G$ is properly embedded in $(S,\partial S,X)$ and (2) the train-track structure $\tau$ is {\it compatible with the planar structure} in the following sense: For every $v\in \V(G)$, $D_v$ has a cyclic order defined by a local orientation of $S$ near $v$. Then every gate at $v$ consists of edges that are consecutive respect to the cycle order. We note that the consecutiveness is independent of the choice of local orientations, thus the definition also works for non-orientable surface $S$.
\end{defn}

\begin{rem}
Train-tracks are commonly used to describe complicated curves or foliations. For these purposes, train-tracks are often assumed to have additional properties, such as that every vertex $v$ has degree $3$, the number of gates at each vertex is always two, and the complement of a train-track is hyperbolic, all of which are not assumed in this article. See \cite{PennerHarer_TrainTrack}.
\end{rem}

\begin{defn}[Carrying between train-tracks]
Let $S$ be a compact surface with a finite set of marked points $X \subset S$. Let $T_1$ and $T_2$ are train-tracks in $(S,\partial S, X)$. We say that {\it $T_2$ carries $T_1$} if there is a train-track map $\phi:T_1 \to T_2$ such that $\phi$ can be extended to a map $\phi:S\to S$ that is ambient homotopic relative to $(\partial S,X)$ to the identity map. In particular, considering a possibly non-closed curve $\gamma:I\to S$ properly embedded in $(S,\partial S,X)$ as a train-track with no vertex, we can say that a train-track $T$ {\it carries} $\gamma$ if $\gamma$ is contained in $S$ up to homotopy in $(S,\partial S, X)$. 
\end{defn}

\begin{defn}[Homotopically infinite train-tracks]
	Let $(S^2,A,\ord)$ be a hyperbolic orbisphere. A train-track $T$ in $(S^2,A)$ is {\it homotopically infinite} if $T$ carries a homotopically infinite closed curve with respect to the orbisphere structure $\ord$.
\end{defn}

\subsection{Decomposition of graph with crossing condition on the unit disk} In this subsection, we investigate a graph theoretic property which will be used to define non-expanding spines.

Let us consider a unit disk in the Euclidean plane and its boundary circle. A {\it chord} is an Euclidean line segment joining to point on the circle. We say that two chords {\it intersect} if they intersect in the interior of the disk. Similarly, for two sets of chords $S_1$ and $S_2$, we say that {\it $S_1$ and $S_2$ intersect} if there exist chords $s_1\in S_1$ and $s_2 \in S_2$ so that $s_1$ and $s_2$ intersect.

Fix $n\ge 2$ points on the boundary circle $C$ of a unit disk. There are $n(n-1)/2$ different chords joining the $n$ points. Let $S$ be a collection of these chords. The collection $S$ can also be considered as a graph. We abuse notation and use $S$ to indicate the graph also. A {\it decomposition of a graph} $G$ is a collection of subgraphs $G_1,G_2,\dots, G_k$ which gives a partition on the set of edges.

\begin{lem}\label{lem:CrossCondiCplGraph}
Let $v_1,v_2,\dots,v_n$ be $n\ge 2$ points on a circle. Let $S$ be a collection of chords joining pairs of $v_i$'s. Suppose $S$ satisfies the following crossing condition:
\begin{itemize}
    \item [](Crossing condition) If two chords $s,s'\in S$ are intersecting, then all the six chords joining any pairs of four endpoints of $s,s'$ are also contained in $S$.
\end{itemize}
Then, as a graph, $S$ is decomposed into mutually non-intersecting (1) complete graphs with at least four vertices and (2) chords, which can also be considered as complete graphs with two vertices.
\end{lem}
\begin{proof}
The condition implies that if $S$ contains two intersecting chords, then it contains the complete graph of the four vertices. Suppose that a subset $S'$ of $S$ forms a complete graph. We first show that if there is a chord $s$ in $S \setminus S'$ that intersects $S'$, then $S$ also contains the complete graph generated by $S' \cup \{s\}$. Denote by $v$ and $w$ the endpoints of $s$. Since $s$ intersects $S'$, there exists $s'$ that intersects $s$. Let $v'$ and $w'$ denote the endpoints of $s'$. For any vertex $u'$ of $S'$, as a graph, that is not $v$ or $w$, either $\overline{u'v'}$ or $\overline{u'w'}$ intersects $s$. In particular, by the crossing condition, the chords $\overline{u'v}$ and $\overline{u'w}$ are contained in $S$. Since $u'$ was taken arbitrarily, the complete graph with vertex set $\V(S')\cup \{v,w\}$ is also contained in $S$.

Then every complete graph in $S$ can be extended until when it does not intersect other chords in $S$, which proves the conclusion.
\end{proof}

\begin{prop}\label{prop:PolyTrain}
Let $t$ be an $n$-gon for $n\ge 2$. For a curve $\gamma$ properly embedded in $(t,\partial t, \V(t))$, we call the boundary edges of $t$ that contain the endpoints of $\gamma$ the side edges of $\gamma$. Let $S$ be a collection of homotopy classes of properly embedded curves joining different boundary edges. Suppose that $S$ satisfies the crossing condition:
\begin{itemize}
    \item [] (Crossing condition) If $\left< [\alpha], [\beta] \right>=1$, then $S$ contains the six homotopy classes of curves connecting any pairs of the four side edges of $\alpha$ and $\beta$. 
\end{itemize}
Then there is a train-track $T$ properly embedded in $(t,\partial t, \V(t))$ such that for any curve $\gamma$ properly embedded in $(t,\partial t, \V(t))$, $\gamma$ is carried by $T$ if and only if $[\gamma]\in S$.
\end{prop}
\begin{proof}
We may consider $t$ as a closed disk. We also choose a point on each boundary edge and take a representative of a homotopy class of curves properly embedded in $t$ as a chord joining the chosen points on the edges. Then $S$ can be considered as a graph, see Figure \ref{fig:PolyTrain}. By Lemma \ref{lem:CrossCondiCplGraph}, $S$ is decomposed into complete graphs (with at least 4 vertices) and curves that are mutually non-intersecting. 

Suppose that $S'\subset S$ forms a complete graph with $k\ge4$ vertices in the decomposition. Then we transform $S'$ into a star-like tree $T'$ whose leaves are the $k$ vertices of $S'$. We transform all the complete graphs in the decomposition of $S$ to star-like trees as above, see the second figure in Figure \ref{fig:PolyTrain}. Then we have a graph that is the union of some star-like trees and curves which can intersect only in the boundary of $t$.

We define a train-track as follows:
\begin{itemize}
    \item [(1)] Let $v$ be the center of a star-like tree. We define gates at $v$ in such a way that each gate has only one edge.
    \item [(2)] At each intersection in $\partial t$, we zip the intersecting curves up a little bit as the third figure in Figure \ref{fig:PolyTrain}. More precisely, assume that $k$ edges, say $e_1,e_2,\dots e_k$ intersect at a boundary point $p$. The transformation generates one vertex $v$ of degree $d+1$; the vertex $v$ is incident to the (deformed) $k$ edges $e_1,e_2,\dots,e_k$, and to one new edge, say $e$ which joins $p$ to $v$. There are two gates at $v$: $\{\mathrm{the~direction~along~}e_i~|~i\in[1,k]_\Zbb\}$ and $\{\mathrm{the~direction~along~}e\}$.
\end{itemize}
\begin{figure}[h!]
    \centering
        \def\svgwidth{0.8\textwidth}
        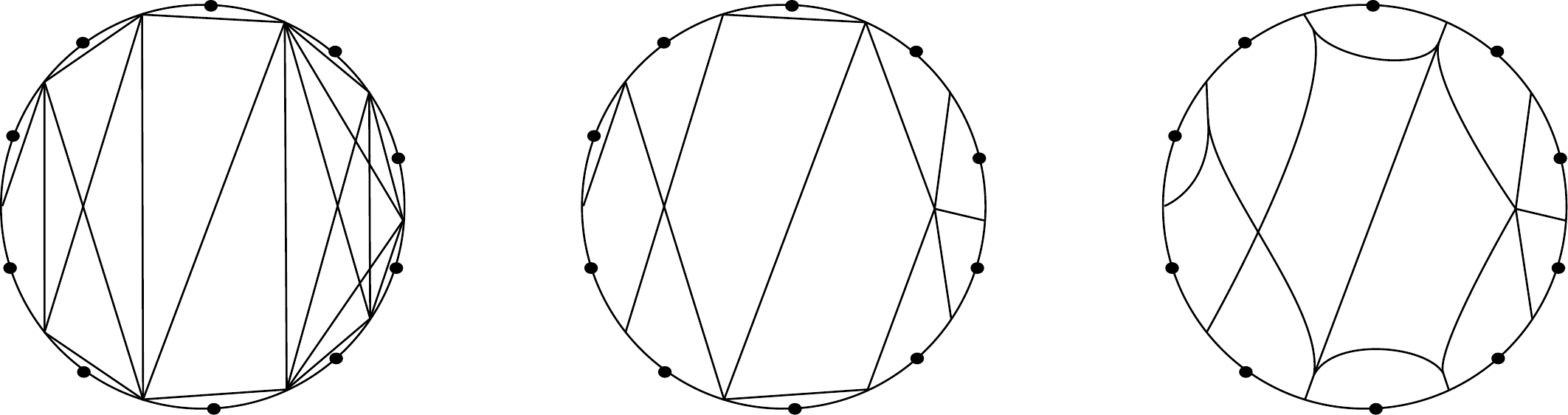
        \caption{Transformation from a graph to a train-track. The dots on the boundary are vertices of a polygon. The graph contains two complete graphs with more than 3 vertices. These graphs are transformed into star-like trees of degree $4$ and $5$ respectively. We ``zip-up'' at boundary points to define a train-track structure.}
        \label{fig:PolyTrain}
\end{figure}
It is immediate from the construction that the train-track satisfies the desired property.
\end{proof}

\subsection{Non-expanding spines of tiles}

Let $t$ be a level-$0$ tile of a finite subdivision rule $\R$ and $n\ge0$. For simplicity, we assume that $t$ is homeomorphic to a closed 2-disc, i.e., boundary edges are not identified. We will define the level-$n$ non-expanding spine of $t$ as a train-track properly embedded in $t$ which is roughly speaking the union of level-$n$ recurrent bones in $t$. If boundary edges are identified, we first define a train-track in the closed 2-disc $\bft$, which is the domain of the characteristic map $\phi_t:\bft\to t$, and then define the non-expanding spine as the image of the train-track by $\phi_t$. 

For a level-$0$ tile $t$ and a level-$n$ band $b^n=\bandn$, we say that {\it $b^n$ is a subband of $t$} if $t^n\subset t$ and $e^n_1,e^n_2\subset \partial t$. We say that two level-$n$ subbands $b^n_0$ and $b^n_1$ of a level-$0$ tile $t$ {\it intersect} if their bones have non-zero intersection number, which must be one, of the homotopy classes of curves properly embedded in $(t, \R^n(\partial t), \V(\R^n(\partial t)))$. It is immediate that if two level-$n$ subbands $b^n_0$ and $b^n_1$ intersect then they are bands of the same level-$n$ subtile of $t$.

The next lemma, in particular the property (3), implies that the set of level-$n$ recurrent bones of $t$ satisfies the crossing condition in the statement of Lemma \ref{lem:CrossCondiCplGraph} or Proposition \ref{prop:PolyTrain}.

\begin{lem}\label{lem:CrossingCondiTile}
Let $t$ be a level-0 tile of a finite subdivision rule. For a level-$0$ recurrent band $b_0=(t;e_{0,1},e_{0,2})$, let $A^0:=\{b_i=(t;e_{i,1},e_{i,2})~|~i=1,\dots, k\}$ be the collection of level-$0$ recurrent bands that intersect $b_0$. Suppose that $A^0$ is non-empty. Then we have the following properties.
\begin{itemize}
    \item [(1)] For any $n \ge 0$ and $i\in [0,k]_\Zbb$, there is a level-$n$ subtile $t^n$ of $t$ such that each $b_i$ has a unique level-$n$ subband $b^n_i=(t^n;e^n_{i,1},e^n_{i,2})$, which is also recurrent. That is, in the directed graph $\Bcal$ of bands, $[b_i]$ belongs to only one cycle, say $C$, and there is no directed path from $C$ to another cycle.
    \item [(2)] There exists $p>0$ such that $b^{np}_i$ is of type $b_i$ for every $i\in[0,k]_\Zbb$ and any $n>0$. Moreover, $p$ can be chosen as the common period of the cycles in $\Bcal$ containing $[b_i]$'s.  
    \item [(3)] For every pair $(i,j), (i',j') \in [0,k]_\Zbb \times \{1,2\}$ with $e_{i,j}\neq e_{i',j'}$, the band $(t^n;e^n_{i,j},e^n_{i',j'})$ is a level-$n$ recurrent subband of $t$ for every $n\ge 0$.
\end{itemize} 
\end{lem}
\begin{proof}
For any $n>0$, every level-$n$ child of $b_0$ intersects any level-$n$ child of any $b_i$. Hence, all the level-$n$ children of $b_0,b_1,\dots,b_k$ are bands of the same level-$n$ subtile, say $t^n$, of $t$.

For $n>0$, we define a set $A^n$ as the collection of level-$n$ recurrent subbands of the level-$0$ bands in $A^0$. Since every recurrent level-$n$ band has at least one recurrent child at each higher level, we have a sequence of surjections $A^0\leftarrow A^1\leftarrow \cdots$ which map recurrent subbands to their parents.

To show the uniqueness in (1), it suffices to show that the surjections ($A^{n+1}\to A^n$)'s are actually bijections. Since $b_0$ is recurrent, for infinitely many $n_0>0$ the level-$0$ band $b_0$ has a level-$n_0$ recurrent subband $b^{n_0}_0$ of type $b_0$. In particular, $t^{n_0}$ is of type $\bft$. Then the types of level-$n_0$ subbands of $b_i$ for $i\in [1,k]_\Zbb$ injectively correspond to $b_j$ for $j\in [1,k]_\Zbb$. This implies that $A^0\leftarrow \cdots \leftarrow A^{n_0}$ is a sequence of bijections. Since there are infinitely many such $n_0$'s, $A^0\leftarrow A^1 \leftarrow \cdots$ is a sequence of bijections also.

(2) Let $p$ be the least positive integer satisfying that $b^p_0=(t^p;e^p_{0,1},e^p_{0,2})$ is of type $b_0=(t;e_{0,1},e_{0,2})$. Such a $p$ exists because $b_0$ is recurrent. We claim that $b^{np}_i$ is of type $b_i$ for every $n \ge 1$. Here is a sketch of the proof and we left the details to the reader. Since tiles and bands are objects embedded in the sphere, we can define an order on $A^n$ according to how close $b^n_i$ is to $b^n_0$. The order is preserved by the bijections $A^{m+1}\to A^m$'s we discussed in (1), which implies $b^{np}_i$ is of type $b_i$ for every $n>0$ and $i\in [1,k]_\Zbb$. By exchanging the roles of $b_0$ with $b_i$ for any $i\in[1,k]_\Zbb$, we obtain that $p$ is the common period of the cycles in $\Bcal$ containing $[b_i]$'s.

(3) Let $p>0$ be the number determined in (2). It follows from (2) that for every $n>0$ and $(i,j)\in [0,k]_\Zbb \times \{1,2\}$, the level-$np$ subtile $t^{np}$ is of type $\bft$ and its boundary edge $e^{np}_{i,j}$ is a subedge of $e_{i,j}$ and of type $\bfe_{i,j}$. Hence, for every pair $(i,j), (i',j') \in [0,k] \times \{1,2\}$ with $e_{i,j}\neq e_{i',j'}$, the level-$0$ band $(t;e_{i,j},e_{i',j'})$ has a level-$np$ subband $(t^{np},e^{np}_{i,j},e^{np}_{i',j'})$ which is of type $(t;e_{i,j},e_{i',j'})$. Then $(t^n,e^n_{i,j},e^n_{i',j'})$ is recurrent for every $n\ge 1$.
\end{proof}

\begin{prop}\label{prop:TrainTile}
Let $t$ be a level-0 tile of a finite subdivision rule $\R$. Then there is a train-track $T^n(t)$ whose underlying graph is properly embedded in $(t,\R^n(\partial t),\V(\R^n(\partial t)))$ such that any curve $\gamma$ properly embedded in $(t,\R^n(\partial t),\V(\R^n(\partial t)))$ is carried by $T^n(t)$ if and only if $\gamma$ is the bone of a level-$n$ recurrent subband of $t$.
\end{prop}
\begin{proof}
We first assume that the tile $t$ is a closed disk, i.e., its boundary edges are not identified. It follows from Lemma \ref{lem:CrossingCondiTile}-(3) that the collection of spines $B^n(t)$ satisfies the crossing condition in Proposition \ref{prop:PolyTrain}. Hence we have the desired train-tack by Proposition \ref{prop:PolyTrain}.

If $t$ is not a closed disk. We can use Proposition \ref{prop:PolyTrain} for its domain $\bft$, which is a closed disk, of the characteristic map $\phi_t:\bft\to t$ to obtain a train-track $T^n(\bft)$ in $\bft$. Then we define $T^n(t)$ as the $\phi_t$-image of $T^n(\bft)$. If two level-$n$ edges that are identified by $\phi_t$ have boundary points of $T^n(\bft)$, then we also identify the boundary points when we define $T^n(t)$.
\end{proof}

\begin{defn}[Non-expanding spine of a tile]
Let $t$ be a level-0 tile of a finite subdivision rule $\R$. The {\it level-$n$ train-track of $t$}, denoted by $T^n(t)$, is a train-track properly embedded in $(t,\R^n(t),\V(\R^n(t)))$ defined in Proposition \ref{prop:TrainTile}. Simply, $T^n(t)$ defined is by the following procedure:
\begin{itemize}
    \item [(1)] Draw the bones of level-$n$ recurrent subbands of $t$.
    \item [(2)] Merge intersecting bones, which form complete graphs by Lemma \ref{lem:CrossCondiCplGraph}, to a star-like trees.
    \item [(3)] Zip-up bones meeting at a boundary point of $t$ as in Figure \ref{fig:PolyTrain}.
\end{itemize}
\end{defn}

\begin{defn}(Non-expanding spines) Let $\R$ be a finite subdivision rule. For every $n\ge0$, the {\it level-$n$ non-expanding spine $N^n$ of $\R$} is a train-track that is defined by the union of level-$n$ non-expanding spines $T^n(t)$ of all the level-$0$ tiles $t$. When two tiles $t$ and $t'$ have the common level-$n$ edge $e^n$ such that both $T^n(t)$ and $T^n(t')$ have boundary points on $e^n$, then we identify the boundary points when we take the union to define $N^n$.
\end{defn}

\begin{rem}\label{rem:NonexpspineUnionofSpines}
From the definition, we may consider bones of level-$n$ recurrent bands as curves contained in the level-$n$ non-expanding spine $N^n$ or non-expanding spine $T^n(t)$ of tiles $t$. We will in particular consider curves supported in $N^n$ as a concatenation of bones of bands.
\end{rem}

\begin{defn}[Essential non-expanding spines]
Let $\R$ be a finite subdivision rule and $f:\R(S_\R)\to S_\R$ be its subdivision map. Let $A\subset \V(S_\R)$ be a set of marked points. We say that the level-$n$ non-expanding spine $N^n$ of $\R$ is {\it essential relative to $A$} if it contains (more precisely carries as a train-track) a closed curve that is homotopic relative to $A$ neither to a point nor to some iterate of a peripheral loop of a Julia point in $A$.
\end{defn}

\begin{eg}\label{eg:nonexpspine1}
\begin{figure}[h!]
    \centering
       \def\svgwidth{0.5\textwidth}
        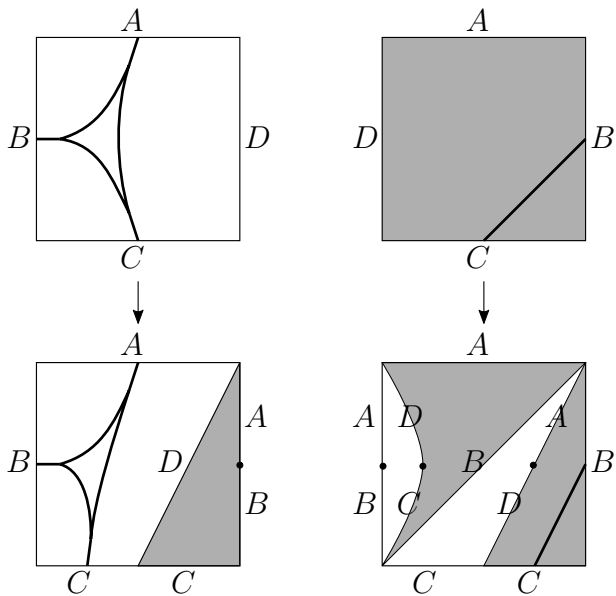
	\caption{An example of level-$0$,$1$ non expanding spines, which are non-essential.}
    \label{fig:nonexpspine1}
\end{figure}
See Figure \ref{fig:nonexpspine1}. The upper two squares are level-$0$ tiles (or tile types). One tile is shaded and the other is not shaded. There are four level-$0$ edges (or edge types) $A,\,B,\,C,$ and $D$. The lower two squares are subdivisions at level-$1$. The bi-recurrent components of level-$0$ and level-$1$ non-expanding spine are both homotopic to a peripheral loop of a Julia vertex. By Theorem \ref{thm:nonexpspine}, the subdivision map does not have a Levy cycle. Later, in Example \ref{eg:nonexpspine1Continued}, we will show that the subdivision map also does not have a Thurston obstruction.

\end{eg}

\begin{defn}[Nested sequence of non-expanding spines]\label{defn:NestSeqNonExpSpine}
For any $n>m\ge0$, there is a natural map $\phi^n_m:N^n\to N^m$ sending every level-$n$ recurrent bone to its level-$m$ parent. Since $\R^n(S_\R)$ is a subdivision of $\R^m(S_\R)$, we may consider $N^n$ and $N^m$ as train-tracks in the same complex $\R^m(S_\R)$. Then the map $\phi^n_m$ is ambient homotopic to the the identity map relative to the $1$-skeleton $\R^m(S_\R)^{(1)}$ of the level-$m$ subdivision complex, i.e., there is an extension $\phi^n_m:\R^m(S_\R)\to \R^m(S_\R)$ that sends, possibly non-homeomorphically, every edge to the same edge and fixes vertices point-wisely. It is straightforward from definitions that the map $\phi^n_m:N^n\to N^m$ is a train-track map. Then we have a sequence of train-track maps
\[
    N^0 \xleftarrow[]{\phi^1_0} N^1 \xleftarrow[]{\phi^2_1} N^2 \xleftarrow[]{\phi^3_2} \cdots.
\]
We call this sequence the {\it nested sequence of non-expanding spines.}
\end{defn}

For example, in Figure \ref{fig:TripotTraintrack}, each tripod is mapped to a curve by $\phi^1_0$ and $\phi^2_1$.

\begin{prop}
Let $\R$ be a finite subdivision rule and $N^n$ be the level-$n$ non-expanding spine of $\R$. For any $n>m$, if $N^n$ is essential relative to $A$ then $N^m$ is also essential relative to $A$
\end{prop}
\begin{proof}
Since $N^n$ is essential, there is a close curve $\gamma$ supported in $N^n$ such that $\gamma$ is neither homotopically trivial nor homotopic to some iterates of the peripheral loop of a Julia point in $A$. Then $\phi^n_m(\gamma)$ is a closed curve supported in $N^m$ with the same property. Then $N^m$ is essential.
\end{proof}

\proofstep{Restriction of the ranges of  non-expanded recurrent curves to $N^n$} Let $\R$ be a finite subdivision rule. Let $A\subset \V(S_\R)$ be a set of marked points. It is straightforward from Proposition \ref{prop:TrainTile} that a closed curve $\gamma:I\to S^2\setminus A$ is homotopic relative to $A$ to a level-$n$ recurrent non-expanded curve if and only if it is carried by $N^n$ in $(S^2,A)$. Hence, when considering a level-$n$ non-expanded recurrent curve $\gamma:I\to \R^n(S_\R)$, we can restrict the range to $N^n \subset \R^n(S_\R)$ and think of $\gamma$ as a curve supported in $N^n$.

\begin{prop}\label{prop:InfGenSeq}
	Let $\R$ be a finite subdivision rule and $\TmapA$ be the subdivision map where $A\subset \V(S_\R)$ is a set of marked points. Let $\ord:A\to [2,\infty]_\Zbb$ be a hyperbolic orbisphere structure of $\TmapA$. Then the level-$n$ non-expanding spine $N^n$ is homotopically infinite for every $n\ge 0$ if and only if there is a genealogical sequence of level-$n$ homotopically infinite non-expanded recurrent curves $\{\gamma^n:(-\infty,\infty)\to N^n \subset \R^n(S_\R)\}_{n \ge 0}$.
\end{prop}

\begin{proof}

If $N^n$ is not homotopically infinite, then every closed curve is homotopically finite, i.e., either trivial or some iterate of the peripheral loop of $a\in A$ with $\ord(a)< \infty$. Since $N^n$ consists of finitely many bones, any infinite curve in $N^n$ is approximated by closed curves. Then any infinite curve supported in $N^n$ also cannot be homotopically infinite, so the desired genealogical sequence of curves does not exist.

Suppose that $N^n$ is homotopically infinite for every $n\ge0$. We are going to define a set $C^n$ of homotopically infinite non-expanded recurrent curves supported $N^n$ with maps $\{\phi^n_m:C^n\to C^m~|~n>m\ge0\}$ which send any level-$n$ non-expanded recurrent curves to their level-$m$ parents. We will also (1) define a metric $d_n$ on each $C^n$ for which the continuity of $\phi^n_m$ easily follows and (2) show that every $C^n$ is compact. Then the inverse limit of $\{\phi^n_m:C^n\to C^m~|~n>m\ge0\}$ is non-empty whose every element yields a desired genealogical sequence of homotopically infinite non-expanded recurrent curves.

Recall that a level-$n$ non-expanded curve $\alpha:(-\infty,\infty) \to \R^n(S_\R)$ is defined to satisfy that for any $n\in \Zbb$, $\alpha([n,n+1])$ is a level-$n$ bone. Let us define the set $C^n$ as the collection of (parametrized) level-$n$ homotopically infinite non-expanded recurrent curves $\gamma^n:(-\infty,\infty)\to N^n$ up to reparametrization such that $\gamma^n$ does not contain a homotopically finite closed curve, which is a closed curve whose free homotopy class corresponds to the conjugate class of a torsion element of $\pi_1(S^2,A,\ord)$. Each $C^n$ is non-empty because $N^n$ is homotopically infinite. We define a metric $d_n$ on $C^n$ by $d_n(\gamma^n,\delta^n)=2^{-m}$ where $\gamma^n,\delta^n\in C^n$ and $m>0$ is the minimal integer satisfying $\gamma^n([-m,m])\neq \delta^n([-m,m])$ as unions of level-$n$ bones. It is easy to show that for $\alpha^n,\beta^n,\gamma^n\in C^n$,
\[
    d_n(\alpha^n,\gamma^n) = \max( d_n(\alpha^n,\beta^n), d_n(\beta^n,\gamma^n) )
\]
so that $d_n$ is indeed a metric. It is immediate from the definition that $\phi^n_m:(C^n,d_n)\to (C^m,d_m)$ is distance non-increasing. Hence $\phi^n_m$ is uniformly continuous.

Lastly, let us show that $(C^n,d_n)$ is sequentially compact. Suppose that $\{\gamma^n_i\in C_n\}_{i\ge 1}$ is any sequence in $C^n$. Recall that $\gamma^n_i([0,1])$ is a bone of level-$n$ band so that, in particular, $\gamma^n_i(0)$ is a point in $N^n \cap \R^n(S_\R)^{(1)}$, which is a finite set. By dropping to a subsequence, we may assume that there exists $x\in N^n$ such that $\gamma^n_i(0)=x$ for any $i>0$. Let $p:\Dbb \to S^2 \setminus A^\infty$ is the orbifold universal covering map of $(S^2,A,\ord)$, where $A^\infty=\{a\in A~|~ \ord(a)=\infty\}$. Choose $\widetilde{x} \in p^{-1}(x)$. For every $i>0$, there exists a unique lifting $\widetilde{\gamma}^n_i:(-\infty,\infty)\to p^{-1}(N^n)$ with $\widetilde{\gamma}^n_i(0)=\widetilde{x}$. Define $S_0:=\{\widetilde{\gamma}^n_i\}_{i\ge 1}$. 

Fix $r>0$. Let $B(\widetilde{x},r)\subset \Dbb$ denote the hyperbolic ball of radius $r>0$ with the center at $\widetilde{x}$. For any $i>0$ and some $k_i,l_i\in \Zbb_{\ge0}$, we say that $\widetilde{\gamma}^n_i([-k_i,l_i])$ is the {\it longest initial subcurve of $\widetilde{\gamma}^n_i$ staying in $B(\widetilde{x},r)$} if $\widetilde{\gamma}^n_i([-k_i,l_i])\subset B(\widetilde{x},r)$ and $\widetilde{\gamma}^n_i(-k_i-1), \widetilde{\gamma}^n_i(l_i+1) \notin B(\widetilde{x},r)$. The intersection  $B(\widetilde{x},r)\cap p^{-1}(N^n)$ has at most finitely many edges of $p^{-1}(N^n)$. Since every element of $C^n$ does not contain a homotopically finite closed curve, any $\widetilde{\gamma}^n_i$ does not contain a closed curve. It follows that there exists $m_1<m_2<\cdots$ such that every element of the subsequence $\{\widetilde{\gamma}^n_{m_i}\}_{i>0}$ of $S_0$ has the same longest initial subcurve staying in $B(\widetilde{x},r)$.

Take a sequence $0<r_1<r_2<\cdots$ with $r_i\to \infty$. For any $m>0$ define $S_m$ as a subsequence of $S_{m-1}$ whose elements have the same longest initial subcurve staying in $B(\widetilde{x},r_m)$. By taking the diagonal of a sequence of subsequences $S_0,S_1,S_2,\dots$, we have a subsequence $\{\widetilde{\gamma}^n_{m_i}\}$ of $S_0$ with the following property: There exist two strictly increasing sequences of positive integers $\{a_i\}_{i\ge0}$ and $\{b_i\}_{i\ge0}$ such that
\begin{itemize}
    \item[(1)] $\widetilde{\gamma}^n_{m_i}([-a_i,b_i])$ is the longest initial subcurve staying in $B(\widetilde{x},r_i)$ and
    \item[(2)] $\widetilde{\gamma}^n_{m_j}([-a_i,b_i])=\widetilde{\gamma}^n_{m_i}([-a_i,b_i])$ for any $j>i$, i.e., the initial subcurves are accumulated.
\end{itemize}
We define a curve $\widetilde{\gamma}^n:(-\infty,\infty)\to p^{-1}(N^n)$ by $\widetilde{\gamma}^n|_{[-a_i,b_i]}=\widetilde{\gamma}^n_i|_{[-a_i,b_i]}$ for every $i\ge 1$, which is well-defined by (2). It follows from (1) that for any $i\ge 1$ we have
\[
    d_\Dbb(\widetilde{x},\widetilde{\gamma}^n(-a_i-1)),~d_\Dbb(\widetilde{x},\widetilde{\gamma}^n(b_i+1))>r_i,
\]
where $d_\Dbb$ is the hyperbolic metric on $\Dbb$. Hence, $\gamma^n:=p \circ \widetilde{\gamma}^n$ is homotopically infinite. Then $\gamma^n_{m_i}\to\gamma^n\in C^n$, which implies $(C^n,d_n)$ is sequentially compact.

\end{proof}

\begin{thm}\label{thm:nonexpspine}
Let $\R$ be a finite subdivision rule and $f:\R(S_\R)\to S_\R$ be its subdivision map which is not doubly covered by a torus endomorphism. Let $A\subset \V(S_\Rcal)$ be a set of marked points, i.e., $P_f \cup f(A)\subset A$. Then the \pcf branched covering $\TmapA$ has a Levy cycle if and only if the level-$n$ non-expanding spine $N^n$ is essential relative to $A$ for every $n\ge0$.
\end{thm}
\begin{proof}
Let $\ord:A\to[2,\infty]_\Zbb$ be an orbisphere structure. Any multiplication of $\ord$ by a positive integer gives rise to another orbisphere structure with strictly decreased Euler characteristic. Similarly, changing the order of every Fatou point in $A$ into the infinity also yields an orbisphere structure with strictly decreased Euler characteristic, if some order was actually changed. Hence, we always have a hyperbolic orbisphere structure $\ord:A\to[2,\infty]_\Zbb$ with the property that $\ord(a)=\infty$ if and only if $a\in A$ is a Fatou point. Then a closed curve is homotopically infinite with respect to $\ord$ if and only if it is neither homotopic relative to $A$ to a point nor to some iterate of a peripheral loop of a Julia point in $A$. Then the theorem follows from Proposition \ref{prop:LevytoGeneaSeq}, \ref{prop:GeneaSeqtoLevy}, and \ref{prop:InfGenSeq}.
\end{proof}

\section{Graph intersecting obstruction}

\subsection{Graph intersecting obstructions}\label{sec:graph int obs}
Suppose that $\TmapA$ is a \pcf branched covering. A graph $G\subset S^2$ is {\it forward invariant under $f$ up to isotopy relative to $A$} if there exist a subgraph $H$ of $f^{-1}(G)$ and a homeomorphism $\phi:S^2\to S^2$ such that $\phi(H)=G$ and $\phi$ is isotopic to the identity map relative to $A$. A graph $G$ is {\it forward invariant under $f$} if $f(G)\subset G$. A multicurve $\Gamma$ on $(S^2,A)$ is {\it forward invariant under $f$ up to isotopy relative to $A$} if it is so as a graph. A multicurve $\Gamma$ is {\it backward invariant under $f$ up to isotopy relative to $A$}, or {\it $f$-stable}, if every connected component of $f^{-1}(\gamma)$ for $\gamma \in \Gamma$ is either isotopic relative to $A$ to an element of $\Gamma$ or peripheral to $A$. When $f$ and $A$ are understood, we omit ``under $f$'' and ``relative to $A$''.

\begin{prop}\label{prop:inv upto iso to inv}
Let $f:(S^2,A)\righttoleftarrow$ be a \pcf branched covering and $G$ be a graph that is forward invariant up to isotopy. Then there exists $\iota:(S^2,A)\righttoleftarrow$ which is isotopic to $id_{S^2}:(S^2,A)\righttoleftarrow$ relative to $A$, such that $G$ is forward invariant under a \pcf branched covering $g:(S^2,A)\righttoleftarrow$ defined by $g:=f \circ \iota$. Especially, $f$ and $g$ are combinatorially equivalent by $id_{S^2}$ and $\iota$.
\end{prop}

\begin{proof}
Let $H$ be a subgraph of $f^{-1}(G)$ isotopic to $G$ rel $\V(G)$. By extending the isotopy to $S^2$, we have $\iota:(S^2,A) \to (S^2,A)$ such that $\iota(G)=H$ and $\iota$ and $id_{S^2}$ are isotopic relative to $A$. Let $g:=f \circ \iota$. Then $id \circ g=f \circ \iota$ and $g(G)=f(\iota(G))=f(H) \subset G$.
\end{proof}

Due to Proposition \ref{prop:inv upto iso to inv}, we may consider forward invariant graphs instead of graphs that are forward invariant up to isotopy when discussing properties of combinatorial equivalence classes, such as Levy cycles and Thurston obstructions.

\vspace{5pt}
Let $\Gamma$ be a multicurve in $S^2 \setminus A$. The {\it Thurston linear transformation of $\Gamma$} is a linear map $f_\Gamma:\Rbb^\Gamma \to \Rbb^\Gamma$ defined by
\[
f_\Gamma (\gamma)=\sum\limits_{\gamma'\subset f^{-1}(\gamma)}\frac{1}{\deg (f|_{\gamma'}:\gamma' \to \gamma)} [\gamma']_\Gamma
\]
where $\gamma'$ is a connected component of $f^{-1}(\gamma)$ and $[\gamma']_\Gamma$ is an element of $\Gamma$ isotopic to $\gamma'$ if exists. If no such connected component exists, then the sum is defined to be zero. Since $f_\Gamma$ is a non-negative matrix, it has a non-negative real eigenvalue $\lambda(f_\Gamma)$ that is the spectral radius of $f_\Gamma$. If $\lambda(f_\Gamma)\ge 1$, then $\Gamma$ is a {\it Thurston obstruction}. A $n\times n$ non-negative square matrix $M$ is {\it irreducible} if for each $i,\,j$ with $1 \le i,\,j\le n$ there exists $k\ge1$ such that the $(i,\,j)$-entry of $M^k$ is positive. An {\it irreducible multicurve  $\Gamma$} is a multicurve whose Thurston linear transformation $f_\Gamma$ is irreducible. An {\it irreducible Thurston obstruction} is an irreducible multicurve that is a Thurston obstruction.

\begin{rem}
A Thurston obstruction $\Gamma$ is usually assumed to be $f$-stable. For any multicurve $\Gamma$ with $\lambda(f_\Gamma)\neq 0$, there exists a sub-multicurve $\Gamma' \subset \Gamma$ such that $\Gamma'$ is irreducible and $\lambda(f_{\Gamma'})=\lambda(f_\Gamma)$. Such $\Gamma'$ is determined as the multicurve of an irreducible diagonal block $A_i$ of the upper-triangular block form \eqref{eqn:uppper_tri_block} of $f_\Gamma$ with $\lambda(A_i)=\lambda(f_\Gamma)$. By Lemma \ref{lem:forwardinvariant to invariant}, $\Gamma'$ extends to a $f$-invariant multicurve $\Gamma''$ with $\lambda(f_{\Gamma''})\ge \lambda(f_{\Gamma})$ . Hence we may drop the $f$-condition condition from Thurston's characterization.
\end{rem}

\begin{lem}\label{lem:irred is forward inv}
If a multicurve $\Gamma$ is irreducible, then $\Gamma$ is forward invariant up to isotopy.
\end{lem}
\begin{proof}
For a contradiction, assume there exists $\gamma$ such that for every $\gamma'\in \Gamma$ no connected component of $f^{-1}(\gamma')$ is isotopic to $\gamma$. Then $f_\Gamma:\Rbb^\Gamma \to \Rbb^{\Gamma \setminus \{\gamma\}} \subset \Rbb^{\Gamma}$, thus $f_\Gamma$ is not irreducible.
\end{proof}

\begin{lem}[{\cite[Lemma~2.2]{Tan_quad_mating}}]\label{lem:forwardinvariant to invariant}
For any multicurve $\Gamma$ of $(S^2,A)$ that is forward invariant up to isotopy, there exists a multicurve $\Gamma'$ which is backward invariant up to isotopy such that $\Gamma' \supset \Gamma$ and $\lambda(f_{\Gamma'}) \ge \lambda(f_{\Gamma})$.
\end{lem}
\begin{proof}
Let $\Gamma_0=\Gamma$ and $\Gamma_n$ be the set of homotopy classes of essential curves in $f^{-n}(\Gamma_0)$. By the forward invariance up to isotopy, $\Gamma_0 \subset \Gamma_1 \subset \dots$ is an increasing sequence of multicurves. Note that $|A|-3$ is the maximal number of non-homotopic essential simple closed curves that can be disjointly embedded into $S^2\setminus A$. Hence there exists $n$ such that $\Gamma_n$ is $f$-invariant. The inequality $\lambda(f_{\Gamma'}) \ge \lambda(f_{\Gamma})$ follows from the following: for non-negative square matrices $M$ and $N$ if $M_{ij} \ge N_{ij}$ for every $(i,\,j)$ then $\lambda(M) \ge \lambda(N)$. 
\end{proof}

\begin{thm}[Arcs intersecting obstructions {\cite[Theorem~3.2]{PilgrimTan}}]\label{thm:arc int obs} Let $f:(S^2,A)\righttoleftarrow$ be a \pcf branched covering and $G$ be an invariant graph such that $f|_G:G \to G$ is a graph automorphism. Then every irreducible Thurston obstruction intersecting $G$ is a Levy cycle.
\end{thm}

We generalize it to a case when $G$ is an $f$-invariant graph with $h_{top}(f|_G)=0$.

\begin{thm}[Graph intersecting obstruction]\label{thm:graph intersecting obs} Let $f:(S^2,A)\righttoleftarrow$ be a \pcf branched covering and $G$ be a forward invariant graph such that $h_{top}(f|_G)=0$. Then every irreducible Thurston obstruction intersecting $G$ is a Levy cycle.
\end{thm}

\begin{rem}
The graphs in Theorem \ref{thm:arc int obs} and Theorem \ref{thm:graph intersecting obs} are possibly disconnected. Moreover, the same statement works for graphs which are forward invariant up to isotopy by Proposition \ref{prop:inv upto iso to inv} with a slight modification to define $h_{top}(f|_G)$.
\end{rem} 

An {\it arc} of  $(S^2,A)$ is a curve embedded in $S^2$ such that its interior is embedded in $S^2 \setminus A$ and its endpoints are in $A$. A {\it geometric intersection number} $\gamma \cdot \gamma'$ between curves (arcs and simple closed curves) is defined as the minimal number of intersection points in their isotopy classes relative to $A$.

\vspace{5pt}
For a multicurve $\Gamma$, the {\it unweighted Thurston transformation} $f_{\#,\Gamma}:\Rbb^\Gamma \to \Rbb^\Gamma$ is defined by

\[
f_{\#,\Gamma}(\gamma)=\sum\limits_{\gamma'\subset f^{-1}(\gamma)} [\gamma']_\Gamma
\]
where $\gamma'$ is a connected component of $f^{-1}(\gamma)$ and $[\gamma']_\Gamma$ is an element of $\Gamma$ isotopic to $\gamma'$ if exists. If there is no such element, then the sum is defined to be zero. For every $(i,\,j)$, (1) $0 \le (f_\Gamma)_{ij} \le (f_{\#,\Gamma})_{ij}$ and (2) $(f_\Gamma)_{ij}=0$ if and only if $(f_{\#,\Gamma})_{ij}=0$. So $f_{\#,\Gamma}$ is irreducible if and only if $f_\Gamma$ is irreducible.

\begin{proof}[Proof of Theorem \ref{thm:graph intersecting obs}]
Let $\mathrm{Edge}(G)=\{e_1,e_2,\dots,e_n\}$. For any simple closed curve $\gamma \subset S^2 \setminus A$, define

\[
(\gamma)_G=\left(\begin{array}{c}\# \{\gamma \cap e_1\} \\ \# \{\gamma \cap e_2\} \\ \vdots \\ \# \{\gamma \cap e_n\} \end{array}\right)
~~~~~~
[\gamma]_G= \left( \begin{array}{c}\gamma \cdot e_1 \\ \gamma \cdot e_2 \\ \vdots \\ \gamma \cdot e_n \end{array}\right)
\]
where $\gamma \cdot e_i$ means the geometric intersection number of $\gamma$ and $e_i$ relative to $A$. The $(~)_G$ and $[~]_G$ are linearly extended to weighted multicurves. Let $T_G$ be the incidence matrix of $f|_G$. Let $\Gamma=\{\gamma_1,\gamma_2,\dots,\gamma_m\}$ be an irreducible Thurston obstruction. From $\# \{f^{-k}(\gamma_i) \cap e_j\} \ge f^{-k}(\gamma_i) \cdot e_j$, for every $k\ge 1$,  we have
\begin{equation}
{T_G}^k \cdot (\gamma_i)_G = (f^{-k}(\gamma_i))_G \ge [f^{-k}(\gamma_i)]_G \ge \sum_{j=1}^m({f_{\#,\Gamma}}^k)_{ij}[\gamma_j]_G.    
\end{equation}
The third term counts the intersection of $G$ with all connected components of $f^{-k}(\gamma_i)$, but the last term counts the intersection of $G$ with connected components of $f^{-k}(\gamma_i)$ isotopic relative to $A$ to some connected components of $\Gamma$.

It follows from Proposition \ref{prop:entropyzerograph} that entries of ${T_G}^k$ grows at most polynomially fast, so  $(f_{\#,\Gamma}^k)_{ij}$ grows at most polynomially fast too. Since $f_{\#,\Gamma}$ is an irreducible non-negative integer matrix, $f_{\#,\Gamma}$ is a permutation by Lemma \ref{lem:integerirredmatrix}. Recall that (1) $0 \le (f_\Gamma)_{ij} \le (f_{\#,\Gamma})_{ij}$ and (2) $(f_\Gamma)_{ij}>0$ if and only if $(f_{\#,\Gamma})_{ij}>0$. Hence the only way to have $\lambda(f_\Gamma) \ge 1$ is $f_{\#,\Gamma}=f_\Gamma$. Then $\Gamma$ is a Levy cycle.
\end{proof}

\subsection{Application in the mating of polynomials}$ $

\proofstep{Formal mating}Let $f$ and $g$ be \pcf polynomials of degree $d$. Consider $f$ and $g$ as maps from the complex plane $\Cbb$ to itself. Let $\overline{\Cbb}$ be the compactification of $\Cbb$ by the circle $S^1$ each point of which corresponds to a linear direction to the infinity. Then, $f$ and $g$ extend to the boundary $S^1$ as the angle $d$-times map. We can parametrize $S^1$ by $\theta\in[0,1]/\{0\sim1\}$ where $\theta$ indicates the angle of an external ray. Let us use subscriptions $-_f$ and $-_g$ to distinguish two compactified complex planes where $f$ and $g$ act on respectively, such as $\overline{\Dbb}_f:=\Cbb_f \cup S^1_f$, $\overline{\Dbb}_g:=\Cbb_g \cup S^1_g$, $f: \overline{\Dbb}_f \righttoleftarrow$ and $g:\overline{\Dbb}_g \righttoleftarrow$. Define a sphere $S^2_{f\uplus g}$ by gluing two compactified planes $\overline{\Cbb}_f$ and $\overline{\Cbb}_g$ by the equivalence relation $\theta_f\sim -\theta_g$ for any $\theta_f\in S^1_f$ and $\theta_g\in S^1_g$ with $\theta_f=\theta_g$ as numbers in $[0,1)$. The dynamics of $f$ and $g$ also glue together to induce a dynamics $f\uplus g:S^2_{f\uplus g}\righttoleftarrow$, which is also a \pcf branched self-covering of the sphere. We call $f\uplus g:S^2_{f\uplus g}\righttoleftarrow$ the {\it formal mating of $f$ and $g$}.

\proofstep{Ray-equivalence class}
Let $f\uplus g:S^2_{f\uplus g}\righttoleftarrow$ be the formal mating of \pcf polynomials $f$ and $g$. External rays of $f$ and $g$ forms a foliation on $S^2_{f\uplus g} \setminus (K_f \cup K_g)$ where $K_f \subset \Cbb_f$ and $K_f \subset \Cbb_g$ are filled Julia sets. Every leaf of the foliation is called a {\it ray-equivalence class} of the formal mating $f \uplus g$. Each ray-equivalence class consists of external rays of $f$ and $g$ of the same period and pre-period.

\proofstep{Degenerate mating} If $f$ or $g$ (or both) is non-hyperbolic, there could be an obvious Levy cycle of $f \uplus g$ which could be removed by collapsing some ray equivalence classes.

Let $F:=f \uplus g$. Suppose that $f$ is not hyperbolic. Then the post-critical set $P_f$ is in the Julia set $\Jcal_f$ so that each post-critical point of $f$ is contained in a ray-equivalence class. Suppose that there is a periodic ray-equivalence class $\xi$ that contains two points of $P_F$ such that $\xi$ is topologically a tree. Then the boundary of a small neighborhood of $\xi$ generates a Levy cycle. Hence, we will collapse $\xi$ to a point. To obtain a topological branched covering on the quotient sphere, we need a little more careful construction as follows, see \cite{ShiShi_Rees} for details.

Let $Y'$ be the set of ray-equivalence classes containing at least to points in $\Omega_F \cup P_F$. Define $Y$ be the set of ray-equivalence class $\xi'$ containing at least one point of $\Omega_F \cup P_F$ such that $F^m(\xi')=F^n(\xi)$ for some $m,n\ge0$ and $\xi\in Y$. If $Y \neq \emptyset$ and every element of $Y$ is topologically a tree, then we define $S'^2$ as the quotient of $S^2_{f \uplus g}$ by collapsing every ray-equivalence class in $Y$ to a point. The map $F$ induces a degree $d$ self-map on $S'^2$ which is not a branched covering near $F^{-1}(\xi)$ for $\xi\in Y$. But we can take a homotopy near $F^{-1}(\xi)$ for $\xi \in Y$ to obtain a branched covering $F':(S'^2,P_{F'})\righttoleftarrow$, which is called the {\it degenerate mating} of $f$ and $g$. We also denote the degenerate mating by $f \uplus' g:S^2_{f\uplus'g}\righttoleftarrow$. When both $f$ and $g$ are hyperbolic, the degenerate mating is equal to the formal mating.

\begin{eg}[$f_{1/2} \uplus' f_{1/4}$] For $\theta\in \Qbb \cap [0,1)$, let $f_{\theta}$ denote the \pcf polynomial at the landing point of the external ray of angle $\theta$ in the parameter plane of the quadratic polynomials $z^2+c$. Let $f=f_{1/2}$ and $g=f_{1/4}$. Let us denote by $R_f(\theta)$ and $R_g(\theta)$ the external rays of $f$ and $g$ of angle $\theta$.

The set $Y'$ defined above consists of 3 ray-equivalence class: $\xi_0:=R_f(0) \cup R_g(0)$, $\xi_1:=R_f(1/2)\cup R_g(1/2)$, and $\xi_2:=R_f(1/4)\cup R_f(3/4) \cup R_g(1/4) \cup R_g(3/4)$. The set $Y$ has one more ray-equivalence class $\xi_3:=R_f(3/8) \cup R_f(7/8) \cup R_g(1/8)\cup R_g(5/8)$ than $Y'$.

Let $F=f \uplus g$ be the formal mating. The boundary of a small disk neighborhood of $\xi_0$ is a Levy cycle of period one. Let us also use $\xi_i$ to indicate the collapsed points in $S^2_{f\uplus'g}$. The degenerate mating $F'$ maps $\xi_i$ to $\xi_{i-1}$ for $i=1,2,3$, where $\xi_2$ and $\xi_3$ are critical points of degree two.
\end{eg}

\begin{defn}
For \pcf polynomials $f$ and $g$ of the same degree, we say that $f$ and $g$ are {\it mateable} if the degenerate mating $F':=f \uplus' g:(S^2_{F'},P_{F'})\righttoleftarrow$ is combinatorially equivalent to a \pcf rational map.
\end{defn}

\begin{coro}\label{cor:mating}
Let $f$ and $g$ be post-critically finite hyperbolic (resp.\@ possibly non-hyperbolic) polynomials such that at least one of $f$ and $g$ has  core entropy zero. Then $f$ and $g$ are mateable if and only if the formal mating (resp.\@ degenerate mating) does not have a Levy cycle.
\end{coro}
\begin{proof}
Assume $f$ and $g$ are hyperbolic and the core entropy of $f$ is zero. Suppose the formal mating of $f$ and $g$ does not have a Levy cycle but has a Thurston obstruction $\Gamma$. We may assume that $\Gamma$ is irreducible. We can think of Hubbard trees $H_f$ and $H_g$ of $f$ and $g$ as invariant trees in the glued sphere $S^2_{f\uplus g}$. By Theorem \ref{thm:graph intersecting obs}, $\Gamma$ is disjoint from $H_f$. Then $\Gamma$ yields a Thurston obstruction of the polynomial $g$, which is a contradiction.

Suppose that $f$ and $g$ may not be hyperbolic and $f$ has core entropy zero. Let $\pi:S^2_{f \uplus g} \to S^2_{f \uplus' g}$ denote the projection from the sphere of the formal mating to the sphere of the degenerate mating. Let $H_f$ and $H_g$ denote the Hubbard tress embedded in $S^2_{f \uplus g}$, and let $H'_f$ and $H'_g$ denote their $\pi$-image in $S^2_{f \uplus' g}$. Some points of $H_f$ and $H_g$ are identified by $\pi$, but $H'_f$ still has entropy zero. By the argument in the previous paragraph, if there is an irreducible Thurston $\Gamma$ obstruction of the degenerate mating that is not a Levy cycle, the $\Gamma$ is disjoint from $H'_f$. For $x\in S^2_{f \uplus' g}$, if $\pi^{-1}(x)$ is not a singleton, then $x\in H'_f \cap H'_g$. Hence the multicurve $\Gamma$ can be lifted to a Thurston obstruction of the formal mating $f\uplus g$ with still being disjoint from $H_f$. Then $\Gamma$ again yields a Thurston obstruction of the polynomial $g$, which is a contradiction.
\end{proof}

\section{Finite subdivision rules with polynomial growth of edge subdivisions}

\begin{defn}[Polynomial growth of edge subdivisions]
Let $\R$ be a finite subdivision rule and $e$ be a level-$0$ edge. The edge $e$ {\it has sub-exponential growth of subdivisions} if
\[
\lim\limits_{n\rightarrow \infty}  \# \left\{ \textup{level-}n~\textup{subedges~of}~e \right\} ^{1/n} = 1.
\]
We say that $\R$ has {\it sub-exponential growth of edge subdivisions} if every level-$0$ edge has sub-exponential growth of subdivisions. By Proposition \ref{prop:subexpedgesubdiv}, we can substitute the term ``sub-exponential'' for ``{\it polynomial}\,''.
\end{defn}

Recall that we defined the directed graph of edge subdivisions $\Ecal$ in Section \ref{sec:DirectedGraphs}. Also recall that a level-$0$ edge $e$ is called periodic (or also called recurrent) if the corresponding vertex $[e]$ in $\Ecal$ is contained in a cycle.

\begin{prop}\label{prop:subexpedgesubdiv}
A finite subdivision rule $\Rcal$ has sub-exponential growth of edge subdivisions if and only if the cycles in $\Ecal$ are disjoint. In this case, for each level-$0$ edge $e$, $\# \left\{ \textup{level-}n~\textup{subedges~of}~e \right\}$ grows polynomially fast as $n\to \infty$.
\end{prop}

\begin{proof}
It is straightforward from Theorem \ref{thm:expsubexp} and Proposition \ref{prop:edgesubdivision}.
\end{proof}

Let $f^{(1)}:\R^{(1)} \rightarrow \R^{(1)}$ be the restriction of $f$ to the $1$-skeleton $\R^{(1)}$. Then $f^{(1)}$ is a Markov map. The adjacency matrix of the directed graph of edge subdivision $\Ecal$ coincides with the incidence matrix of the Markov map $f^{(1)}$. The following proposition is immediate from Proposition $\ref{prop:entropyzerograph}$.

\begin{prop}\label{prop:entpysubexpedgesubdiv}
A finite subdivision rule $\R$ has polynomial growth of edge subdivisions if and only if $h_{top}(f^{(1)})=0$.
\end{prop}

Let $e$ be a level-$0$ periodic edge. For every $n>0$, $e$ has at least one level-$n$ child (subedge) that is recurrent, see Section \ref{sec:ParentsChildren}. If $e$ has polynomial growth of subdivisions, then the recurrent subedges are unique at each level. The same statement also works for periodic bands.

\begin{prop}[Unique recurrent children]\label{prop:uniquerecurrent}
Let $\R$ be a finite subdivision rule and $e$ be a  level-$0$ periodic edge with polynomial growth of subdivisions. For any $n \in 1$, $e$ has a unique level-$n$ subedge that is recurrent. For a level-$0$ periodic band $\band$, if $e_1$ and $e_2$ have polynomial growth of subdivisions, then for any $n>0$ there exists a unique level-$n$ subband of $\band$ that is recurrent.
\end{prop}

\begin{proof}
    By Proposition \ref{prop:subexpedgesubdiv}, there exists a unique cycle in $\Ecal$ passing through $[e]$. Hence, for any $n>0$, there is only one path of length $n$ from $[e]$ and supported within the cycle, which determines a unique level-$n$ recurrent subedge by Proposition \ref{prop:edgesubdivision}. The uniqueness of recurrent subedge can also follow from Proposition \ref{prop:uniquerecurrent}.
    
    If $\band$ is periodic, then it has at least one level-$n$ unique subband $\bandn$. By Lemma \ref{lem:recurrentbandandedge}, the level-$n$ edges $e^n_1$ and $e^n_2$ are recurrent subedges of $e_1$ and $e_2$, which are unique by the previous paragraph. Hence the recurrent subbands are unique at each level.
\end{proof}

\begin{prop}\label{prop:nonexpspinesubexpedge}
Suppose a finite subdivision rule $\R$ has polynomial growth of edge subdivisions. Then every train-track map $\phi^{n+1}_n:N^{n+1} \to N^n$ in the nested sequence of non-expanding spines, defined in Definition \ref{defn:NestSeqNonExpSpine},
\[
    N^0 \xleftarrow[]{\phi^1_0} N^1 \xleftarrow[]{\phi^2_1} N^2 \xleftarrow[]{\phi^3_2} \cdots,
\]
is a homeomorphism.
\end{prop}
\begin{proof}
 Let $\band$ be a level-$0$ periodic band. It follows from Proposition \ref{prop:uniquerecurrent} and Lemma \ref{lem:recurrentbandandedge} that for any $n$ there exists a unique level-$n$ recurrent band of $\band$ such that its sides are unique level-$n$ recurrent subedges of $e_1$ and $e_2$. If two level-$0$ periodic bands share a side $e$ then the level-$n$ recurrent bands also share a side which is the level-$n$ recurrent subedge of $e$. Hence $N^n$ and $N^0$ are made up of the same number of bones of bands which are glued in the same way.
\end{proof}

\begin{thm}\label{thm:nonexpspine_polyedgesubdiv}
Let $\R$ be a finite subdivision rule with polynomial growth of edge subdivisions and $f$ be its subdivision map which is not doubly covered by a torus endomorphism. Let $A\subset \V(S_\R)$ be a set of marked point, i.e., $f(A) \cup P_f \subset A$. Then the followings are equivalent.
\begin{enumerate}
    \item[(1)] $f:(S^2,A)\righttoleftarrow$ does not have a Levy cycle.
    \item[(2)] The level-$0$ non-expanding spine $N^0$ does not carry a closed curve that is neither homotopic relative to $A$ to a point nor to some iterate of a peripheral loop of a Julia point in $A$.
    \item[(3)] $f:(S^2,A)\righttoleftarrow$ is combinatorially equivalent to a unique rational map up to conjugation by M\"{o}bius transformations, i.e., $f$ does not have a Thurston obstruction.
\end{enumerate}
\end{thm}
\begin{proof}
$(1) \Leftrightarrow (2)$ follows from Theorem \ref{thm:nonexpspine} and Proposition \ref{prop:nonexpspinesubexpedge}. The equivalence with $(3)$ follows from Theorem \ref{thm:graph intersecting obs} and Proposition \ref{prop:entpysubexpedgesubdiv}.
\end{proof}

\begin{eg}[Example \ref{eg:nonexpspine1} continued]\label{eg:nonexpspine1Continued} Removing the edge type $C$ from Figure \ref{fig:nonexpspine1}, we have a finite subdivision rule with bounded edge subdivisions. Since the subdivision maps is unchanged, there is no Levy cycle by the discussion in Example \ref{eg:nonexpspine1}. It follows from Theorem \ref{thm:nonexpspine_polyedgesubdiv} that there is also no Thurston obstruction.
\end{eg}

\section{Examples}\label{sec:eg}

\subsection{Critically fixed rational maps}\label{sec:CritFix}

A rational map is {\it critically fixed} if every critical point is a fixed point. It was recently shown that there is a one-to-one correspondence between critically fixed rational functions and planar graphs. The idea started from \cite{PilgrimTan} and was completed in \cite{Hlushchanka_criticallyfixed}.

\begin{thm}[Hlushchanka, Pilgrim et.\@ al.]\label{thm:CritFix}
There is a one-to-one correspondence between the holomorphic conjugacy classes of critically fixed rational functions and the planar isotopy classes of connected planar graphs without loops.
\end{thm}

Let $G$ be a planar graph without loops and $f$ be the corresponding critically fixed rational map in Theorem \ref{thm:CritFix}. At the end of this subsection, we construct a finite rule $\R_G$ such that (1) its subdivision map is $f$ and (2) every edge never subdivides.

Let $f$ be a critically fixed rational map. The {\it Tischler graph} of $f$ is a graph embedded in $\hCbb$ whose edge set is the collection of fixed internal rays in the immediate attracting basins of all critical points. It follows from \cite{Hlushchanka_criticallyfixed} that the Tischler graph any critically fixed rational map is connected.

To construct a critically fixed rational function from a planar graph without loops, we use {\it blowing-up an arc} construction, that is firstly introduced in \cite{PilgrimTan}.

\vspace{5pt}
\proofstep{Blowing-up an arc.}
Let $f:(S^2,A) \righttoleftarrow$ be a \pcf branched covering and $\gamma$ be an arc fixed by $f$. Let $D \subset S^2$ be an open $2$-disc contained in a small neighborhood of $\gamma$ with $\gamma \subset \partial D$. Let $\gamma'= \partial D - \mathrm{int}(\gamma)$. Define an orientation-preserving continuous map $g:S^2 \setminus D \rightarrow S^2$ in such a way that $g$ maps $\gamma$ and $\gamma'$ to $\gamma$, with endpoints fixed. Define another orientation-preserving continuous map $h:\overline{D} \rightarrow S^2$ in a similar way so that $h$ maps $\gamma$ and $\gamma'$ to $\gamma$, with endpoints fixed, and maps the $D$ to $S^2 \setminus \gamma$ homeomorphically. A new branched covering $f_\gamma:(S^2,A) \righttoleftarrow$ is defined by $f_\gamma|_{S^2\setminus D}=f \circ g$ and $f_\gamma|_{\overline{D}}=f \circ h$. We call $f_\gamma$ the {\it $f$ blown-up along an arc $\gamma$}. Note that $\deg(f_\gamma)=\deg(f)+1$.

\vspace{5pt}
Let $G$ be a planar graph without loops and $A=\V(G)$. Define a \pcf branched covering $f_G:(S^2,A)\righttoleftarrow$ by blowing up the identity map $id_{S^2}:(S^2,A)\righttoleftarrow$ along all edges of $G$. The combinatorial equivalence class dis independent of the order of blowing-up. Each vertex $v$ of $G$ is a critical point of $f_G$ such that $\deg_v(f_G)=\deg(v)+1$. If follows from \cite[Corollary 3] {PilgrimTan} that $f_G$ is combinatorially equivalent to a rational map. Because $f_\gamma$ fixes $\gamma$, the branched covering $f_G$ is the identity on $G$. Define a finite subdivision rule $\R_G$ such that

\begin{itemize}
    \item [(1)] $S_{\R_G}$ be the CW-complex whose $1$-skeleton is $G$
    \item [(2)] $\R_G(S_{\R_G})$ be the CW-complex whose $1$-skeleton is $f^{-1}(G)$, and
    \item [(3)] $f_G:\R_G(S_{\R_G}) \rightarrow S_{\R_G}$ is the subdivision map of $\R_G$.
\end{itemize}

\begin{rem}
When an edge is blown-up, there are two choices for $D$, depending on which side of $\gamma$ the disk $D$ is. But the combinatorial equivalence class of the resulting branched covering is independent.
\end{rem}

\begin{eg}
See Figure \ref{fig:critfixed}. The graph $G$ is a triangle with one more edge attached. Figure \ref{fig:critfixed1} indicates the disk that we use in the blowing-up the edge $\gamma$. The other two figures indicate the CW-complex structures at level-$0$ and $1$. The shaded triangles in Figure \ref{fig:critfixed3} are mapped to the shaded triangle in Figure \ref{fig:critfixed2} under the subdivision map $f_G$.

\begin{figure}[h!]
    \centering
        \begin{subfigure}[b]{0.105\textwidth}
        \def\svgwidth{\textwidth}
        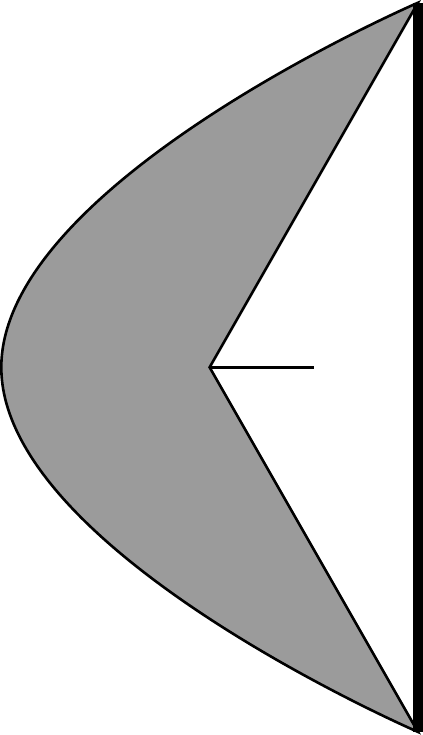
        \caption{}
        \label{fig:critfixed1}
        \end{subfigure}
        \hspace{50pt}
        \begin{subfigure}[b]{0.3\textwidth}
        \def\svgwidth{\textwidth}
        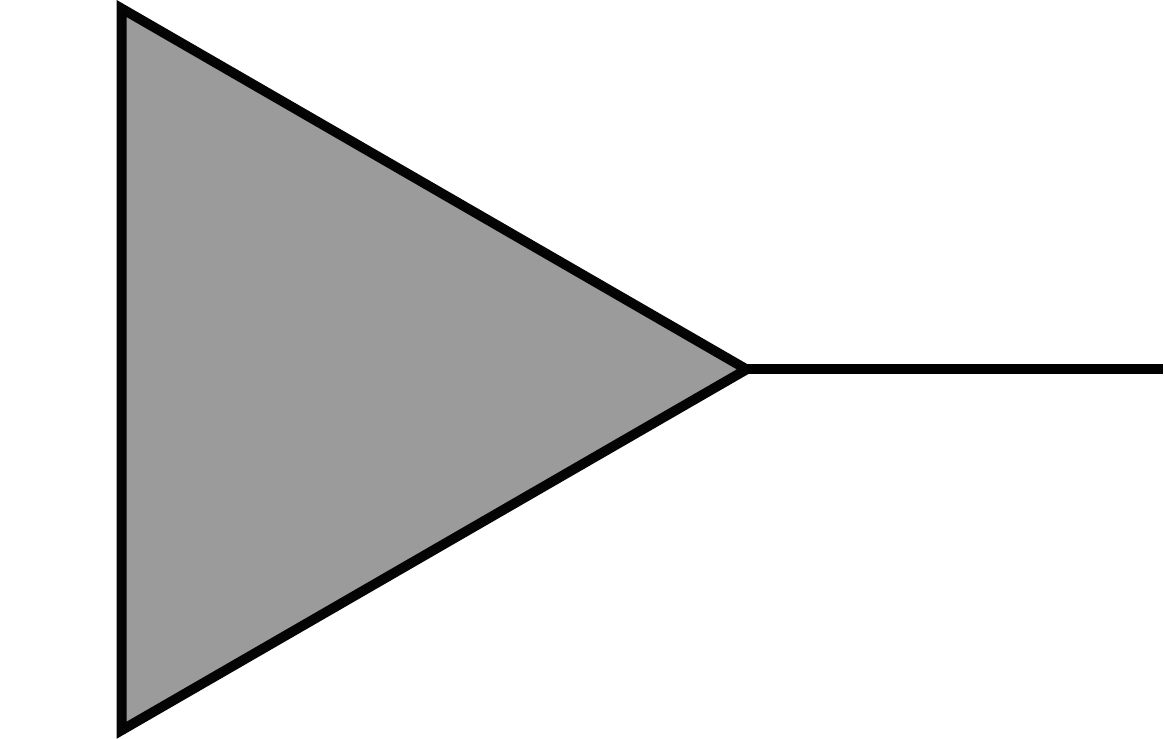
        \caption{$S_{\R_G}$}
        \label{fig:critfixed2}
        \end{subfigure}
        \hspace{15pt}
        \begin{subfigure}[b]{0.35\textwidth}
        \def\svgwidth{\textwidth}
        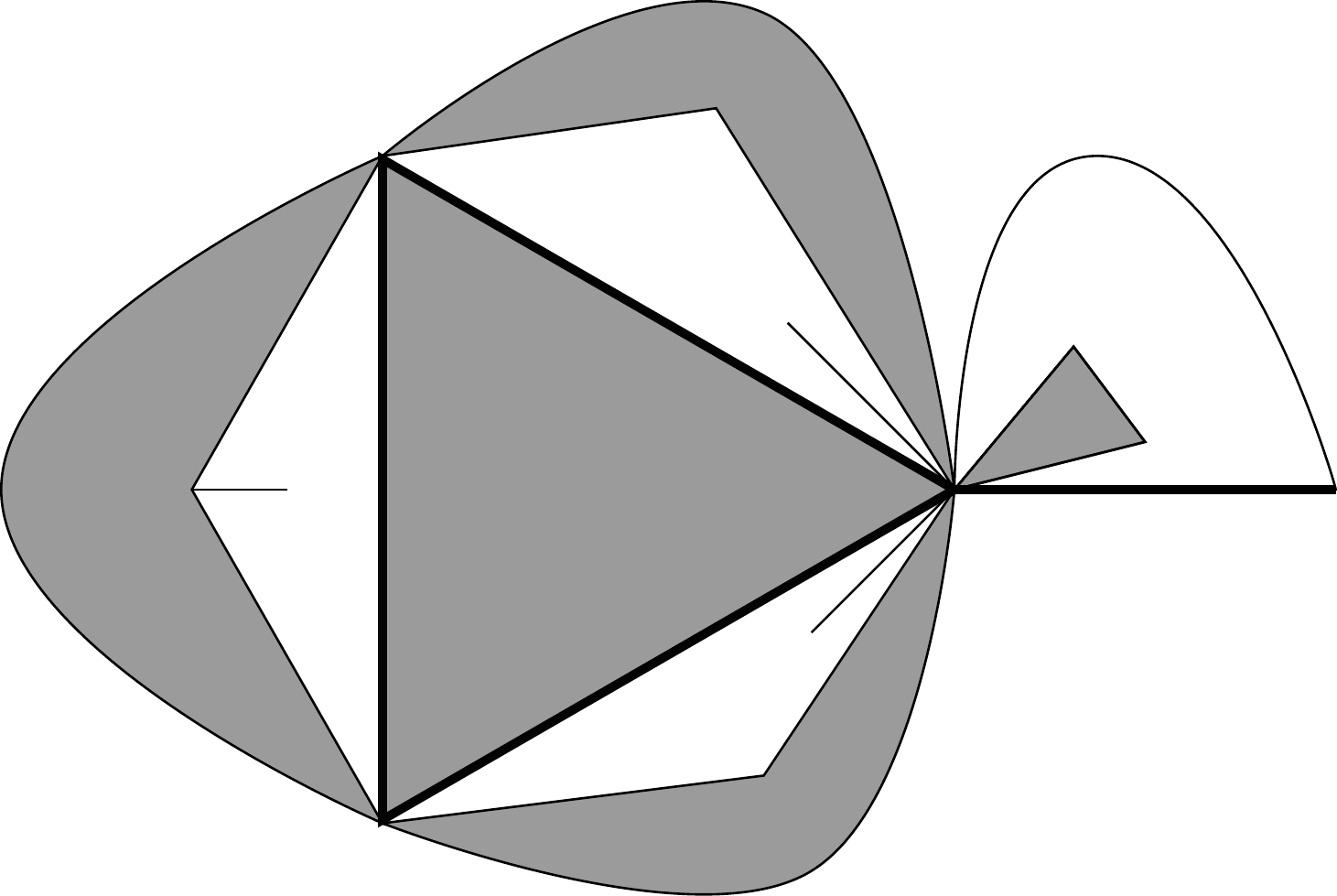
        \caption{$\R_G(S_{\R_G})$}
        \label{fig:critfixed3}
        \end{subfigure}
        
    \caption{}
    \label{fig:critfixed}
\end{figure}
\end{eg}

\subsection{Face-inversion constructions and critically fixed anti-rational maps}\label{sec:faceinversion}

The construction in this section was also investigated in \cite{Geyer_CritFixAnti} and \cite{LLM_KissCritFixAnti} in the study of critically fixed anti-rational maps.

\vspace{5pt}
Let $G$ be a finite graph in the $2$-sphere $S^2$. The graph $G$ defines the CW-complex structure $\Tcal$ with $\Tcal^{(1)}=G$. A graph is {\it $k$-vertex-connected} or {\it $k$-edge-connected} if it is not disconnected by the removal of fewer than $k$ vertices or (open) edges respectively. For the characteristic map $\phi_t:\mbf{t} \rightarrow t$ of a closed $2$-cell $t$, we say that the {\it boundary vertices or edges of $t$ are identified} if more than one vertex or edge are identified under $\phi_t$. The followings are characterizations of $2$- or $3$-connectivity of graphs embedded in $S^2$.

\begin{itemize}
    \item $G$ is $2$-vertex-connected if and only if boundary vertices of every $2$-cell are not identified, i.e., the boundary of every $2$-cell is a Jordan curve.
    
    \item $G$ is $2$-edge-connected if and only if the boundary edges of every $2$-cell are not identified. The $2$-vertex connectedness implies the $2$-edge connectedness.
    
    \item $G$ is $3$-edge-connected if and only if it is $2$-edge-connected and any two $2$-cells do not share more than one edge. It is also equivalent to the dual graph having no cycle of length $\le 2$.
\end{itemize}

Assume $G$ is $2$-vertex-connected and $\deg(v)\ge 3$ for every $v\in \V(G)$. Let $t$ be a $2$-cell of $\Tcal$ and $\sigma_t$ be the reflection of $S^2$ in $\partial t$. This is possible because the $2$-vertex-connectedness implies that $\partial t$ is a simple closed curve. Then $\sigma_t(G)$ is a graph isomorphic to $G$ such that $\sigma_t(G)\cap G = \partial t$. Define a graph $H$ by
\[
H=\bigcup\limits_{t~\textup{is~a~2-cell~of}~\Tcal} \sigma_t(G).
\]
Let $\Tcal'$ be the CW-complex structure on $S^2$ with $\Tcal'^{(1)}=H$. We define a finite subdivision rule as follows: Let $S_\R=\Tcal$ and $\R(S_\R)=\Tcal'$. Define an orientation reversing branched self-covering $f:S^2\righttoleftarrow$ defined by $f|_t=\sigma_t|_t$ for every $2$-cell $t$ of $S_\R$. Then $f$ becomes a subdivision map $f:\R(S_\R) \rightarrow S_\R$. Note that every edge does not subdivide. The degree of $f$ is equal to the number of $2$-cells of $\Tcal$ minus one. We call $\R$ the finite subdivision rule of {\it face-inversion of $G$}. Every vertex $v\in \V(G)$ is a fixed critical point of $\deg_v(f)=\deg(v)-1$.

A natural way to obtain an orientation preserving finite subdivision rule is to take square of the subdivision $\R^2$ and the subdivision map $f^2:\R^2(\Tcal)\rightarrow \Tcal$. We denote by $\R_{sq}$ this squared orientation preserving subdivision rule. Another way is to post-compose with an orientation reversing automorphism of $G$. An automorphism $\tau \in \Aut(G)$ is called {\it orientation reversing} if it extends to an orientation reversing homeomorphism of $S^2$. For any orientation reversing automorphism $\tau\in \Aut(G)$, we have an orientation preserving subdivision map $f_\tau:=\tau \circ f:\R(S_\R) \rightarrow S_\R$ defined on the same subdivision complexes as $\R$. Denote this finite subdivision rule by $\R_\tau$.

\begin{thm}\label{thm:faceinversion}
Let $G$ be a $2$-vertex-connected graph in $S^2$ such that $\deg(v)\ge 3$ for every $v\in \V(G)$. Let $\R$ be the finite subdivision rule of the face-inversion of $G$ and $f:\R(S_\R)\rightarrow S_\R$ be its subdivision map. Let $\tau$ be any orientation reversing automorphism of $G$. Then the followings are equivalent:

\begin{itemize}
    \item [(1)] $G$ is $3$-edge-connected.
    \item [(2)] $f^2:(S^2,\V(G))\righttoleftarrow$ does not have a Levy cycle.
    \item [(2')] $f^2:(S^2,\V(G))\righttoleftarrow$ does not have a Thurston obstruction. 
    \item [(3)] $f_\tau:(S^2,\V(G))\righttoleftarrow$ does not have a Levy cycle.
    \item [(3')] $f_\tau:(S^2,\V(G))\righttoleftarrow$ does not have a Thurston obstruction.
\end{itemize}

\end{thm}

\begin{proof}
A level-$0$ band $b=\band$ is non-separating if and only if there is another \mbox{level-$0$} band $b'=\bandp$ such that $e_1=e_1'$, $e_2=e_2'$ and $t\neq t'$. If such bands $b$ and $b'$ exist, the removal of two edges of $G$ intersecting the bones of these bands disconnects $G$, i.e., $G$ is not $3$-edge-connected. Conversely, if $G$ is not $3$-edge-connected, then such level-$0$ bands $b$ and $b'$ exist. Hence $G$ is $3$-edge-connected if and only if every level-$0$ band is non-separating. In the case, the level-$0$ non-expanding spine for $\R_{sq}$ is an empty set. Then $(1)\Rightarrow(2)\Leftrightarrow(2')$ follow from Theorem \ref{thm:nonexpspine_polyedgesubdiv}. 

Assume $G$ is not $3$-edge-connected so that there are bands $b$ and $b'$ described as in the previous paragraph. The union of bones of $b$ and $b'$ is a homotopically infinite circle contained in the level-$0$ non-expanding spine $N^0$ of $\R_{sq}$. Hence $(2) \Rightarrow (1)$ follows from Theorem \ref{thm:nonexpspine_polyedgesubdiv}. 

The equivalence $(2) \Leftrightarrow (3)$ follows from the fact that the subdivisions $\R^n(S_\R)$ and $\R_\tau^n(S_{\R_\tau})$ have the same CW-complex structure. The level-$2n$ non-expanding spine of $\R_\tau$ is equal to the level-$n$ non-expanding spine of $\R_{sq}$ for $n\ge 0$.
\end{proof}

\begin{rem}
The equivalence $(1)\Leftrightarrow (2) \Leftrightarrow (2')$ is also shown in \cite[Theorem 5.8]{Geyer_CritFixAnti} and \cite[Proposition 4.10]{LLM_KissCritFixAnti}.
\end{rem}

\begin{rem}
For an orientation revering branched covering $f$, $f^2$ is combinatorially equivalent to a rational map if and only if $f$ is combinatorially equivalent to a anti-rational map.
See \cite[Theorem 3.9]{Geyer_CritFixAnti} and \cite[Proposition 6.1]{LLMM_GasketSchwRef}. \end{rem}

\begin{rem}
Even if there exists a vertex $v$ with $\deg(v)=2$, the construction is still well-defined, but $v$ is not a critical point. Note that such vertex $v$ can be removed from the vertex set without any change in the face-inversion construction.
\end{rem}

\begin{eg}\label{eg:fsrwithbigon}
\begin{figure}[h!]
    \centering
    
      \def\svgwidth{0.7\textwidth}
        {\scriptsize
        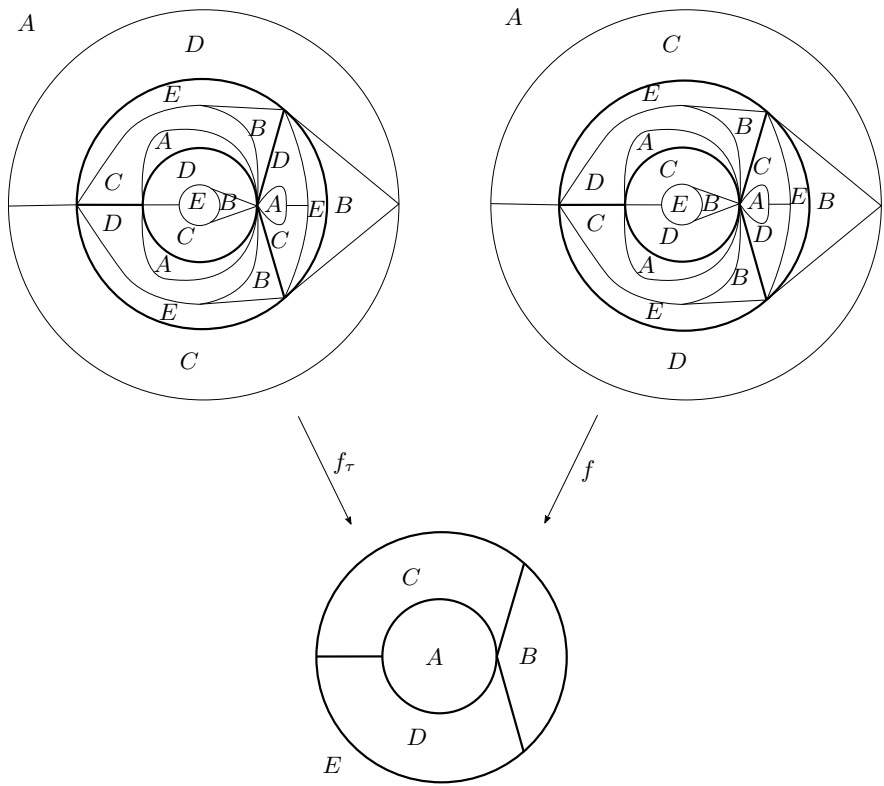
        }
    \caption{Finite subdivision rules defined from the face-inversion of a planar graph.}
    \label{fig:fsrwithbigon}
\end{figure}

\begin{figure}[h!]
    \centering
    \includegraphics[width=0.7\textwidth]{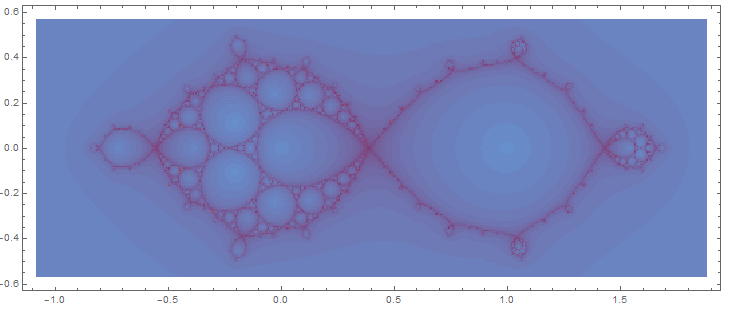}
    \caption{Julia set of $z\mapsto -\frac{1.50351z^2(z^2-1.15757z-0.596204)}{z+0.133305}$}
    \label{fig:hfsrwithbigon_julia}
\end{figure}
See Figure \ref{fig:fsrwithbigon}. Let $G$ be the graph in the bottom and $\tau$ the reflection along the middle horizontal line. Then the left and right subdivisions represent $\R_\tau$ and $\R$ respectively.

In order to obtain an explicit formula of $f_\tau(z)=\frac{p(z)}{q(z)}$, we normalize three vertices on the axis of $\tau$ in Figure \ref{fig:fsrwithbigon} from left to right to $0,1$, and $\infty$. Note that $\deg_z(f_\tau)$ is $2$ at $z=0$ or $1$, and $\deg_z(f_\tau)$ is $3$ at $z=\infty$. Since $0$ and $\infty$ are fixed points, $p(z)$ is a quartic polynomial divided by $z^2$, and $q(z)$ is a linear polynomial. We may assume that $q(z)$ is monic. The conditions that (1) $z=1$ is a critical fixed point and (2) the other two critical points are exchanged by $f(z)$ give rise to a system of equation about coefficients of $p(z)$ and $q(z)$. Solving this numerically, we have 
\[
f_\tau(z)=-\frac{1.50351z^2(z^2-1.15757z-0.596204)}{z+0.133305}.
\]

See Figure \ref{fig:hfsrwithbigon_julia} for the Julia set.

\end{eg}

\subsection{Finite subdivision rules with essential non-expanding spines at higher levels}\label{sec:fsr_needbackite_}
In this subsection, we prove Proposition \ref{prop:fsr_needbackite_} by constructing an example.

\begin{prop}\label{prop:fsr_needbackite_}
For every $N>0$, there exists a finite subdivision rule $\R_N$ with the subdivision map $f_{\R_N}:\R_N(S_{\R_N}) \to S_{\R_N}$ of degree $6$ such that (1) $\V(S_{\R_N})=P_{f_{\R_N}}$, (2) the level-$k$ non-expanding spine $N^k$ is essential relative to $\V(S_{\R_N})$ for $k<N$, and (3) $N^k$ is not essential relative to $\V(S_{\R_N})$ for $k\ge N$.
\end{prop}

\begin{figure}[h!]
    \centering
    \def\svgwidth{\textwidth}
\begingroup%
  \makeatletter%
  \providecommand\color[2][]{%
    \errmessage{(Inkscape) Color is used for the text in Inkscape, but the package 'color.sty' is not loaded}%
    \renewcommand\color[2][]{}%
  }%
  \providecommand\transparent[1]{%
    \errmessage{(Inkscape) Transparency is used (non-zero) for the text in Inkscape, but the package 'transparent.sty' is not loaded}%
    \renewcommand\transparent[1]{}%
  }%
  \providecommand\rotatebox[2]{#2}%
  \newcommand*\fsize{\dimexpr\f@size pt\relax}%
  \newcommand*\lineheight[1]{\fontsize{\fsize}{#1\fsize}\selectfont}%
  \ifx\svgwidth\undefined%
    \setlength{\unitlength}{892.29452315bp}%
    \ifx\svgscale\undefined%
      \relax%
    \else%
      \setlength{\unitlength}{\unitlength * \real{\svgscale}}%
    \fi%
  \else%
    \setlength{\unitlength}{\svgwidth}%
  \fi%
  \global\let\svgwidth\undefined%
  \global\let\svgscale\undefined%
  \makeatother%
  \begin{picture}(1,0.28655312)%
    \lineheight{1}%
    \setlength\tabcolsep{0pt}%
    \put(0,0){\includegraphics[width=\unitlength,page=1]{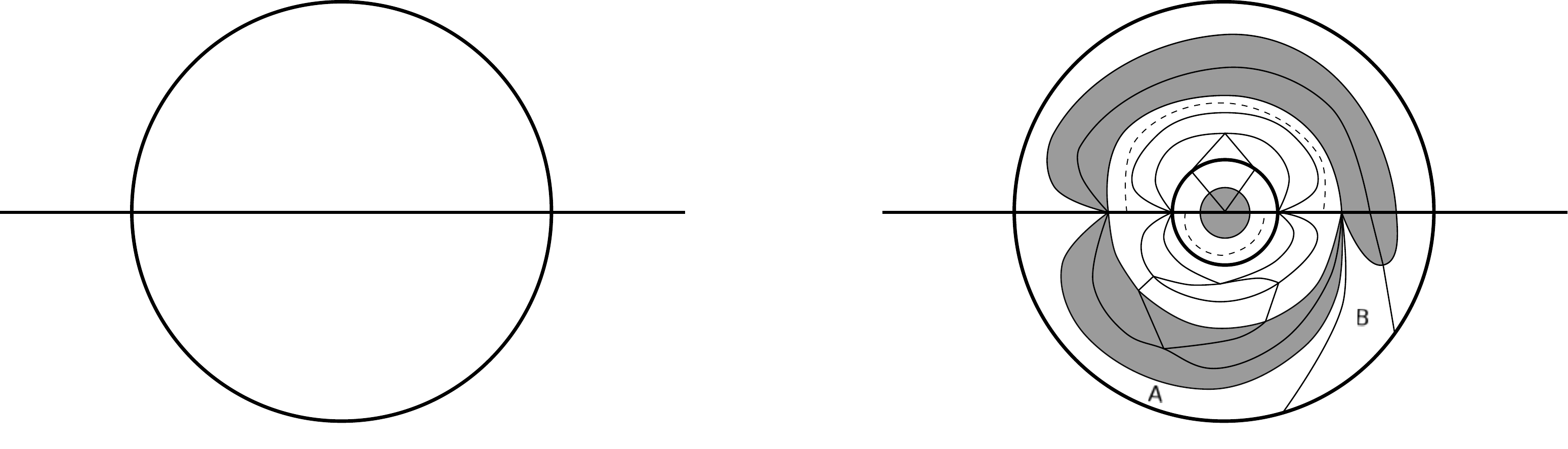}}%
    \put(0.20976831,0.19933369){\color[rgb]{0,0,0}\rotatebox{-3.2358668}{\makebox(0,0)[lt]{\lineheight{1.25}\smash{\begin{tabular}[t]{l}B\end{tabular}}}}}%
    \put(0.21101718,0.08260848){\color[rgb]{0,0,0}\rotatebox{-3.2358668}{\makebox(0,0)[lt]{\lineheight{1.25}\smash{\begin{tabular}[t]{l}A\end{tabular}}}}}%
    \put(0,0){\includegraphics[width=\unitlength,page=2]{fsr_needbackite_.pdf}}%
    \put(0.49467804,0.16381072){\color[rgb]{0,0,0}\makebox(0,0)[lt]{\lineheight{1.25}\smash{\begin{tabular}[t]{l}$f$\end{tabular}}}}%
    \put(0,0){\includegraphics[width=\unitlength,page=3]{fsr_needbackite_.pdf}}%
  \end{picture}%
\endgroup%

        \caption{A degree $6$ finite subdivision rule with six tile types.}
        \label{fig:fsr_needbackite_}
\end{figure}

Let us see the finite subdivision rule $\R$ in Figure \ref{fig:fsr_needbackite_}. The $1$-skeleton at level-$0$ is drawn by bold curves. The non-expanding spines $N^0$ and $N^1$ are drawn by dotted curves. The $N^0$ is essential but $N^1$ is homotopically trivial. Let $f$ be the subdivision map described in Figure \ref{fig:fsr_needbackite_}. We modify the finite subdivision rule $\Rcal$ into $\Rcal'$ as follows:

\begin{enumerate}
    \item Change the labels $A$ and $B$ into $A_1$ and $B_1$.
    \item For $2 \le i\le n$, we draw $n-1$ copies of annuli consisting of $A_i$ and $B_i$ in a row on the left of the annulus consisting of $A_1$ and $B_1$ at level-$0$. Denote by $S_{\R'}$ the modified level-$0$ CW-complex. Define $\R'(S_{\R'})=f^{-1}(S_{\R'})$.
    
    \begin{figure}[h!]
    \centering
    \def\svgwidth{0.6\textwidth}
        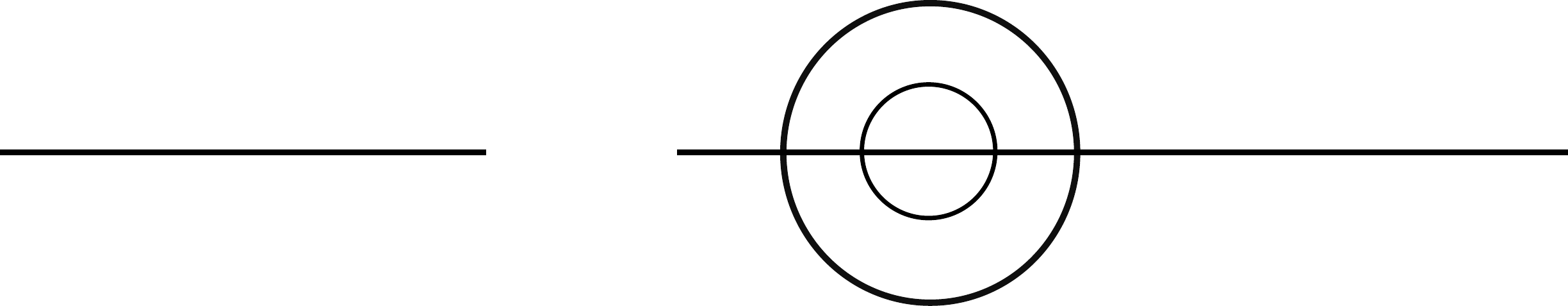
    \end{figure}

    \item Let $\sigma$ be an orientation preserving homeomorphism of the $2$-sphere such that $\sigma(S_{\R'})=S_{\R'}$ and $\sigma(A_i)=A_{i+1}$ and $\sigma(B_i)=B_{i+1}$ for any $1 \le i\le n$, where indices are considered modulo $n$. That is, $\sigma$ is a $1/n$-rotation. Define the subdivision map $f':\R'(S_{\R'})\to S_{\R'}$ by $f'=\sigma \circ f$.
\end{enumerate}

Let $N'^i$ be the level-$i$ non-expanding spine of $\R'$. The $N'^0$ is $n$-copies of circles. The $N'^1$ is the union of $n-1$-copies of circles with three non-closed curves. Similarly, for $k<n$, the level-$k$ non-expanding spine $N'^k$ has $n-k$ circles and some non-closed curves. Therefore, $N'^i$ is essential if $i<n$ and non-essential if $i\ge n$. 


\subsection{Edge-edge expansion vs. edge subdivisions}\label{sec:EdgeEdgeExpVsEdgeSubdiv}
	Let us further investigate the equivalence between the existence of Levy cycles and of Thurston obstructions. Recall that the coefficients of Thurston linear transformation are defined by
	
	\[
	f_\Gamma (\gamma)=\sum\limits_{\gamma'\subset f^{-1}(\gamma)}\frac{1}{\deg (f|_{\gamma'}:\gamma' \to \gamma)} [\gamma']_\Gamma.
	\]
	
	In the setting of finite subdivision rules, the summands $\frac{1}{\deg}$ are related to the expansion between edges and the number of summands is related to the growth rate of edge subdivisions. Hence, we can expect there is no Thurston obstruction if the edge-edge expansion dominates the edge subdivisions. See \cite[Theorem 8.4]{CuiPengTan_Wandering} for a similar comparison.
	
	\vspace{10pt}	
	Let $\R$ be a finite subdivision rule.
	\begin{description}
		\item[Edge-edge expansion] We say that $\R$ is {\it edge-edge $\lambda$-expanding} for $\lambda\ge1$ if there exists $C>0$ such that for any $n\ge0$ and any bone $\gamma$ of level-$0$ band, the level-$n$ subdivision complex $\R^n(S_\R)$ subdivides $\gamma$ into at least $C\cdot \lambda^n$ segments .
		\item[Edge subdivision rate] For any level-$0$ edge $e$, the {\it exponential growth rate of subdivisions of $e$} is the number $\nu_e\ge1$ with $\lim_{n\to\infty}(\#\{\mathrm{level\mbox{-}n~subedges~of~}e\})^{1/n}= \nu_e$.
		The maximum $\nu:=\max~\nu_e$ over level-$0$ edges $e$ is called the {\it maximal exponential growth rate of edge subdivisions}.
	\end{description}
	
	\begin{prop}\label{prop:EdgeEdgeExpWinEdgeSubdiv}
		Let $\R$ be a finite subdivision rule. Let $\nu$ be the maximal exponential growth rate of edge subdivisions. If $\R$ is edge-edge $\lambda$-expanding for some $\lambda>1$, then the non-expanding spine of $\R$ is empty so that the subdivision map $\TmapA$ does not have a Levy cycle for any set of marked points $A$. Moreover, if $\lambda>\nu$, then the subdivision map $\TmapA$ does not have a Thurston obstruction for any set of marked points $A$.
	\end{prop}
	\begin{proof}
		Let $\Gamma=\{\gamma_1,\gamma_2,\dots,\gamma_k\}$ be a multicurve of $(S^2,A)$. For a closed curve $\gamma$ transverse to $\R^n(S_\R)$, we denote by $l_n(\gamma)$ the cardinality of the intersection $(\R^n(S_\R))^{(1)} \cap \gamma$. Let $C_0:=\min_{i\neq j} \frac{l_0(\gamma_i)}{l_0(\gamma_j)}$.
		
		Let $\gamma_i\in \Gamma$ and $\gamma_i'$ be a connected component $f^{-n}(\gamma_i)$ that is isotopic to $\gamma_j$. Since $\R$ is edge-edge $\lambda$-expanding, there exists $C_1>0$, which is independent of the choices of $\gamma_i$ and $\gamma'_i$, such that

		The first part about the emptiness of non-expanding spines immediately follows from the definition of non-expanding spine. 
		Let us assume $\lambda>\nu$ and show the second part. Let $\Gamma=\{\gamma_1,\gamma_2,\dots,\gamma_k\}$ be a multicurve of $(S^2,A)$. For a closed curve $\gamma$ transverse to $\R^n(S_\R)$, we denote by $l_n(\gamma)$ the cardinality of the intersection $(\R^n(S_\R))^{(1)} \cap \gamma$. Let $C_0:=\min_{i\neq j} \frac{l_0(\gamma_i)}{l_0(\gamma_j)}$. Since $\R$ is edge-edge $\lambda$-expanding, there exists $C_1>0$ such that for any $\gamma_i\in \Gamma$ and for any connected component $\gamma_i'$ of $f^{-n}(\gamma_i)$, we have

		\[
		\begin{array}{ccl}
			l_n(\gamma'_i) &\ge &C_1\dot \lambda^n \cdot l_0(\gamma_j)\\
			& \ge & C_0 C_1 \dot \lambda^n \cdot l_0(\gamma_i)
		\end{array}
		\]
		for any $n\ge0$ where $\gamma_j\in \Gamma$ is isotopic to $\gamma'_i$. Then $\deg(f:\gamma_i'\to \gamma_i)\ge C_0C_1\lambda^n$. 
		
	Let $C_2$ be the minimal number satisfying $|e\cap \Gamma|\le C_2$ for any level-$0$ edge $e$. It follows that (a) for any level-$n$ edge $e^n$, $|e^n \cap f^{-n}(\Gamma)| \le C_2$. By the definition of $\nu$, there exists $C_3>0$ such that for any level-$0$ edge $e$ and any $n\ge0$, the number of level-$n$ subedges of $e$ is at most $C_3 \cdot {\nu'}^n$ for some ${\nu'}$ with $\lambda > \nu'>\nu$. Consider a concatenation $\alpha$ of level-$0$ edges connecting two points $a_1,a_2\in A$ with $a_1$ and $a_2$ being in the different Jordan domains of $\gamma_i$. Note that (b) every simple closed curve isotopic to $\gamma_i$ has at least one intersection point with $\alpha$. We have
	\[
	\begin{array}{ccl}
		\# \{\mathrm{components~of}~ f^{-n}(\Gamma)~\mathrm{homotopic~to~} \gamma_i\} & \le &\mathrm{the~cardinality~of~}f^{-n}(\Gamma) \cap \alpha\\
		&\le& C_2 \cdot \#\{\mathrm{level\mbox{-}n~subedges~of~}\alpha\}\\
		&\le& C_2 C_3\cdot|\Edge(S_\R)|\cdot {\nu'}^n.
	\end{array}
	\]
	The first inequality follows from (b), the second follows from (a) and the third follows from the fact that $\alpha$ is a concatenation of at most $|\Edge(S_\R)|$ edges at level-$0$. 
		
	Hence, every entry in the Thurston linear transformation is bounded from above by
	\[
		\frac{C_2 C_3\cdot|\Edge(S_\R)|\cdot {\nu'}	^n}{C_0C_1\lambda^n},
	\]
	which tends to 0 as $n \to \infty$. Then the map $f$ does not allow any Thurston obstruction.
	\end{proof}

There are many ways to improve Proposition \ref{prop:EdgeEdgeExpWinEdgeSubdiv}. Here are two possible directions.

\begin{description}[leftmargin=0cm]
    \item[Suggestion 1] Proposition \ref{prop:EdgeEdgeExpWinEdgeSubdiv} can be compared with \cite[Proposition 5.1]{finite_subdivision_rule_exp1}, which states that if $\R$ is edge separating and vertex separating, then the subdivision map does not have a Levy cycle. One difference is that the subdivision map in Proposition \ref{prop:EdgeEdgeExpWinEdgeSubdiv} has to be of hyperbolic-type, i.e., every critical point is a Fatou point, but \cite[Proposition 5.1]{finite_subdivision_rule_exp1} works for any subdivision maps. The definition of edge separation in \cite{finite_subdivision_rule_exp1} is the edge-edge expansion, defined in this article, only for pairs of edges that do not share end points. The vertex separation might be necessary only for Julia vertices. One might be able to obtain a stronger result by combining these two propositions.\vspace{5pt}
    \item[Suggestion 2]
	It would be possible to combine Theorem \ref{thm:graph intersecting obs} and Proposition \ref{prop:EdgeEdgeExpWinEdgeSubdiv} to obtain a stronger sufficient condition for the equivalence between the existence of Levy and Thurston obstructions. We might be able to (1) have the equivalence on the part where edges subdivide polynomially fast and (2) exclude Thurston obstructions where edges subdivide exponentially fast by assuming the condition in Proposition \ref{prop:EdgeEdgeExpWinEdgeSubdiv}.
\end{description}

\begin{eg}
    \begin{figure}[h!]
        \centering 
        \def\svgwidth{0.5\textwidth}
\begingroup%
  \makeatletter%
  \providecommand\color[2][]{%
    \errmessage{(Inkscape) Color is used for the text in Inkscape, but the package 'color.sty' is not loaded}%
    \renewcommand\color[2][]{}%
  }%
  \providecommand\transparent[1]{%
    \errmessage{(Inkscape) Transparency is used (non-zero) for the text in Inkscape, but the package 'transparent.sty' is not loaded}%
    \renewcommand\transparent[1]{}%
  }%
  \providecommand\rotatebox[2]{#2}%
  \newcommand*\fsize{\dimexpr\f@size pt\relax}%
  \newcommand*\lineheight[1]{\fontsize{\fsize}{#1\fsize}\selectfont}%
  \ifx\svgwidth\undefined%
    \setlength{\unitlength}{549.87465734bp}%
    \ifx\svgscale\undefined%
      \relax%
    \else%
      \setlength{\unitlength}{\unitlength * \real{\svgscale}}%
    \fi%
  \else%
    \setlength{\unitlength}{\svgwidth}%
  \fi%
  \global\let\svgwidth\undefined%
  \global\let\svgscale\undefined%
  \makeatother%
  \begin{picture}(1,0.44672864)%
    \lineheight{1}%
    \setlength\tabcolsep{0pt}%
    \put(0,0){\includegraphics[width=\unitlength,page=1]{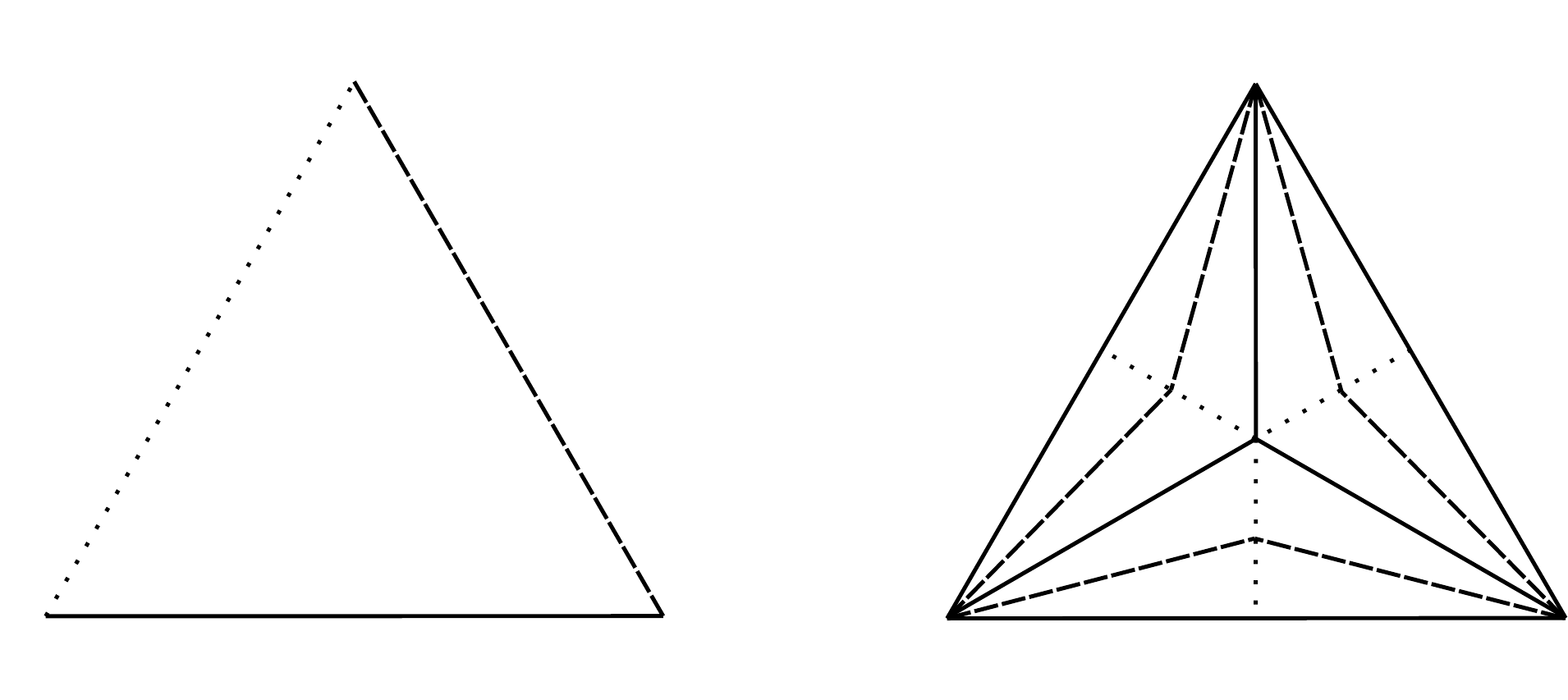}}%
    \put(0.20515545,0.40474478){\color[rgb]{0,0,0}\makebox(0,0)[lt]{\lineheight{1.25}\smash{\begin{tabular}[t]{l}$u$\end{tabular}}}}%
    \put(-0.00412912,0.00961686){\color[rgb]{0,0,0}\makebox(0,0)[lt]{\lineheight{1.25}\smash{\begin{tabular}[t]{l}$v$\end{tabular}}}}%
    \put(0.39525385,0.0096171){\color[rgb]{0,0,0}\makebox(0,0)[lt]{\lineheight{1.25}\smash{\begin{tabular}[t]{l}$w$\end{tabular}}}}%
  \end{picture}%
\endgroup%

        \caption{A finite subdivision rule with $\lambda>\nu$}\label{fig:EdgeEdgeExp}
    \end{figure}
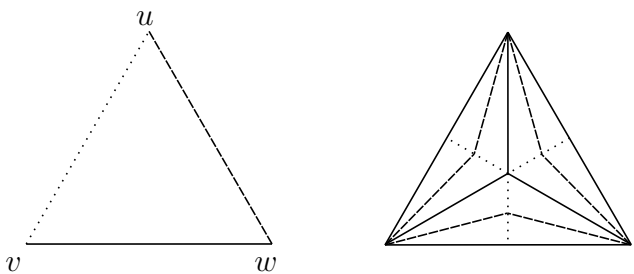
See Figure \ref{fig:EdgeEdgeExp}. We think of the doubles of the the left triangle and get the level-$0$ subdivision complex $S_\R$ with two tiles. Similarly, take the double of the right large triangle, which is subdivided into 12 small triangles, and define it as the level-$1$ complex $\R(S_\R)$. Then Figure \ref{fig:EdgeEdgeExp} defines a finite subdivision rule $\R$ with the subdivision map $f:\R(S_\R)\to S_\R$ which is defined by a map sending each small triangle on the right to a triangle on the left or its copy with the types of edges being preserved. Then $\deg(f)=12$ and $f$ has three critical values $u,v$, and $w$, which are vertices of the level-$0$ triangles. Moreover, $f(u)=f(v)=f(w)=w$.

It is immediate that the non-expanding spine is an empty set. By Theorem \ref{thm:nonexpspine}, $f$ does not have a Levy cycle. Since $\R$ has exponential growth rate of edge subdivisions, we cannot apply Theorem \ref{thm:nonexpspine_polyedgesubdiv} to claim that $f$ does not have a Thurston obstruction. However, it is easy to show that $\lambda=4$ and $\nu=2$, and then $f$ does not have a Thurston obstruction by Proposition \ref{prop:EdgeEdgeExpWinEdgeSubdiv}.
\end{eg}

\bibliography{hfsr}
\bibliographystyle{amsalpha}

\end{document}